\documentclass[11pt,leqno]{amsart}
\usepackage{amsmath,amstext, amsthm,amsfonts,amssymb,amsthm,color}
\numberwithin{equation}{section}
\def\C{\mathbb{C}} \def\P{\mathbb{ P}} \def\N{\mathbb{ N}}\def\K{\mathcal K}
\def\M{\mathcal M}\def\R{\mathbb{ R}}  \def\n{\mathbb{n}}
  
\def\D{\mathcal{D}}\def\O{\mathcal{O}}\def\L{\mathcal{L}}

\def\H{\tilde H}
\def\s{\mathfrak{S}}

\usepackage{epsfig}
\newtheorem{thm}{Theorem}[section]
\newtheorem{lem}[thm]{Lemma}
\newtheorem{cor}[thm]{Corollary}

\newtheorem{prop}[thm]{Proposition}
\newtheorem{defi}[thm]{Definition}
\newtheorem{definition}[thm]{Definition}
\newcommand{\be}{\begin{equation}}
\newcommand{\ee}{\end{equation}}
\newtheorem{fact}[thm]{Fact}

\newcommand{\ba}{\begin{array}}
\newcommand{\ea}{\end{array}}

\newcommand{\al}{\alpha}
\newcommand{\bet}{\beta}

\renewcommand{\O}{\mathcal{O}}
\newcommand{\la}{\lambda}
\newcommand{\La}{\mathbf{\Lambda}}

\newtheorem{rem}[thm]{Remark}
\newcommand{\bea}{\begin{eqnarray}}
\newcommand{\eea}{\end{eqnarray}}

\newcommand{\Sum}{\sum_{n=0}^\infty}

\def\bl{\mathop{\mathrm{bl}}}
\def\tr{\mathop{\mathrm{tr}}}
\def\sg{\mathop{\mathrm{sg}}}
\newcommand{\cyc}{\mathop{\mathrm{cyc}}}
\newcommand{\drop}{\mathop{\mathrm{drop}}}
\newcommand{\exc}{\mathop{\mathrm{exc}}}
\newcommand{\wex}{\mathop{\mathrm{wex}}}
\newcommand{\Fix}{\mathop{\mathrm{Fix}}}
\newcommand{\sing}{\mathop{\mathrm {sing}}}
\newcommand{\rc}{\mathop{\mathrm{cr}}}

\newcommand{\ninv}{\mathop{\mathrm{ninv}}}

\def\d{{\rm dp}}
\def\P{\mathcal{P}}
\def\O{{\mathcal O}}
\def\C{\mathcal{C}}

\def\F{\mathcal{F}}

\def\n{\boldsymbol{n}}
\def\m{\boldsymbol{m}}
\def\x{\boldsymbol{x}}

\def\blue{\textcolor{blue}}
\def\red{\textcolor{red}}
\def\pf{\noindent {\it Proof.} }
\begin{document}

\title[Separation of variables in combinatorics]{Separation of variables  and
combinatorics of linearization coefficients of orthogonal  polynomials}
\author[M. Ismail]{Mourad E.H. Ismail$^{*}$}
\address{City University of Hong Kong,
Tat Chee Avenue, Kowloon, Hong Kong
and King Saud University, Riyadh, Saudi Arabia}
\email{mourad.eh.ismail\@gmail.com}
\thanks{$^*$Research supported by NPST Program of King Saud University; project number 10-MAT1293-02 and
  King Saud University in Riyadh  and by
 Research Grants Council of Hong Kong under  contract \#  101410}

\author[A. Kasraoui]{Anisse Kasraoui$^{\dagger}$}
\address{Fakult\"at f\"ur Mathematik, Universit\"at Wien,
Nordbergstrasse 15,
A-1090 Vienna,
Austria}
\email{anisse.kasraoui@univie.ac.at}
\thanks{$^{\dagger}$ Research supported by the grant S9607-N13 from Austrian Science Foundation FWF
 in the framework of the National Research Network  ``Analytic Combinatorics and Probabilistic Number theory".}
\author[J. Zeng]{Jiang Zeng$^\ddagger$}
\address{Universit\'{e} de Lyon,  Universit\'{e} Lyon 1,  Institut Camille Jordan,
UMR 5028 du CNRS,  69622 Villeurbanne,
France}
\email{zeng@math.univ-lyon1.fr}
\thanks{$^\ddagger$Research supported  by  the grant ANR-08-BLAN-0243-03.}



\date{}

\dedicatory{Dedicated to Richard Askey on his 80th birthday$^1$}\thanks{$^1$This is in appreciation of his fundamental
  mathematical contributions to special functions and orthogonal polynomials, for his tireless efforts of promoting
  the subject,  and for his outstanding mentoring of younger mathematicians.
  }
\begin{abstract}
We propose a new approach to the combinatorial interpretations of  linearization coefficient problem
of orthogonal polynomials.
We   first establish a difference system and then solve it combinatorially and analytically
using the method of separation of variables.  We illustrate our  approach
by applying it to determine the number of perfect matchings, derangements, and other weighted permutation problems.
The separation of variables technique naturally leads to integral representations of combinatorial
numbers where the integrand contains a product of one or more types of orthogonal polynomials.
This also establishes the positivity of such integrals.
\end{abstract}

\maketitle


\tableofcontents


\section{Introduction}
In the late 1960's Askey formulated several conjectures about the nonnegativity
of integrals of products of orthogonal polynomials times certain functions. An excellent
survey of the research in this area, which was spearheaded by Askey, is Askey's CBMS
lecture notes \cite{Ask}, see also \cite{And:Ask:Roy}. In the 1970's it was realized that
some of the integrals considered by Askey and his coauthors have combinatorial interpretations.
Even and Gillis \cite{Eve:Gil} showed that the number of derangements of sets of sizes $n_1, n_2, \ldots, n_m$ is
\be
\label{eqderang}
(-1)^{n_1+\cdots +n_m}  \int_0^\infty e^{-x} \prod_{j=1}^m L_{n_j}(x) dx,
\ee
where $L_n(x)$'s are the simple Laguerre polynomials,
while  Azor, Gillis, and Victor \cite{Azo:Gil:Vic} and independently Godsil~\cite{God}
showed that the number of perfect matchings of sets of
sizes $n_1, n_2, \ldots, n_m$ is
$$
2^{-(n_1+\cdots +n_m)/2} \int_\R \frac{e^{-x^2}}{\sqrt{\pi}}\prod_{j=1}^mH_{n_j}(x)dx,
$$
where $H_n(x)$'s are the Hermite polynomials. Askey and Ismail
\cite{Ask:Ism}  used the MacMahon Master theorem  to give a
systematic combinatorial treatment of the integrals of products of
the classical polynomials  with respect to certain measures.  One of
them   generalized the Even and Gillis result to Meixner
polynomials. Foata and Zeilberger \cite{Foa:Zei}  considered the
general Laguerre numbers
$$
 (-1)^{\sum_{j=1}^m n_j} \;  n_1!\cdots n_m!\; \int_0^\infty \frac{x^\al e^{-x} }{\Gamma(\alpha+1)}\prod_{j=1}^m L^{(\al)}_{n_j}(x) dx,
$$
where $L^{(\al)}_{n}(x)$'s are the Laguerre polynomials.
 Zeng and,  Kim and Zeng extended this study to all Sheffer polynomials in  \cite{Kim:Zen,Zen90, Zen}.

In their combinatorial study of  integrals of products of orthogonal polynomials Askey and Ismail \cite{Ask:Ism}
pointed out another source of positivity results. Recall that a system $\{Q_n(x)\}$
of birth and death process polynomials,  \cite{Kar:McG1}, \cite[\S 5.2]{Ism2},  is generated by
  \begin{equation}
\label{eqBandDPol}
\begin{split}
Q_0(x) =1, \quad Q_1(x) = [b_0 +d_0 -x]/d_0, \\
-xQ_n(x) = b_n Q_{n+1}(x) + d_n Q_{n-1}(x) -(b_n+d_n) Q_n(x),
\end{split}
  \end{equation}
where $\{b_n\}$ and  $\{d_n\}$ are the birth and death rates, respectively, are such that
\be
 \la_n > 0, n \ge 0, \quad \textup{and} \quad d_n > 0, n >0, d_0 \ge 0.
\ee
Karlin and McGregor \cite{Kar:McG1} showed that the probability to go from
state (population) $m$ to state (population) $n$ in time $t$ is given by
\be
p_{m,n}(t) = \frac{b_0 b_1 \cdots b_{n-1}}{d_1 d_2 \cdots d_n}
 \int_0^\infty e^{-xt} Q_m(x) Q_n(x) \, d\mu(x), \quad t >0,
\label{eqpmn}
\ee
where $\mu$ is the orthogonality measure of $\{Q_n(x)\}$. This proves that
\be
\int_0^\infty e^{-xt} Q_m(x) Q_n(x) \, d\mu(x) \ge 0.
\ee
The Laguerre polynomials correspond to $b_n = n+1, d_n = n+ \al$.  Thus
\be
\label{eqposLag}
\int_0^\infty e^{-xt} x^\al e^{-x} L_m^{(\al)} (x) L_n^{(\al)}(x) dx > 0, \quad \al \ge 0, \quad t > 0.
\ee
This and the derangement number  \eqref{eqderang} motivated us to consider
the combinatorial interpretation of the numbers
\be
A^{(\al)}(m,n,s) =  \frac{(-1)^{m+n}}{\Gamma(\al+1)} \int_0^\infty \frac{x^s}{s!} L_m^{(\al)}(x) L_n^{(\al)}(x) x^{\al} e^{-x} dx.
\ee

 One important tool used in the combinatorial study of the integrals of orthogonal
 polynomials is  MacMahon's Master theorem and its $\beta$-extension due to Foata--Zeilberger~\cite{Foa:Zei}.
 When the $\beta$-extension of MacMahon's Master theorem is combined with the exponential formula~\cite{Sta, Wil},
all  the known combinatorial interpretations of the  linearization coefficients of the orthogonal
Sheffer polynomials  can be deduced by computing their generating functions.
Another way to gain insight into the combinatorial interpretation of the linearization coefficients
is from their corresponding moment sequences, see \cite{Kim:Zen,Vie, Zen},
\be
\mu_n=\int_\R x^n\,d\mu(x).
\ee
However the generating function approach  fails when one tries to  extend the previous results
to their $q$-analogues,  even though  a conjecture for the combinatorial interpretation is formulated. For example,
 an  important $q$-analogue for the linearization coefficients of Hermite polynomials was given by
Ismail, Stanton and Viennot~\cite{Ism:Sta:Vie}, but their proof remains difficult.
We are grateful to a  referee for pointing  out that Effros and Popa rediscovered  the
lsmail--Stanton--Viennot result in \cite{Eff:Pop}.
Another proof  due to  Anshelevich~\cite{Ans} uses  stochastic processes, and
 is also far from  being elementary.  Our  paper provides a fresh approach to linearization questions.
 Indeed, one of the main results of this paper is to give an elementary proof of the Ismail--Stanton--Viennot  result.

Separation of variables is a standard technique to solve linear partial differential equations.
The idea is to seek solutions which are products of single variables then by the principle of
linear superposition the general solution is a linear combination of these products. The only
problem left is to use initial and boundary conditions to determine the coefficients.  This
technique can be used to solve difference or differential equations.  One important application
of this method is to solve the Chapman--Kolmogorov equations for birth and death processes,
see \cite[\S 5.2]{Ism2}.  The latter equations is a system of differential equations in time and
partial difference equations in two discrete variables whose solution is given by \eqref{eqpmn}.

In this paper we show how the separation of variables gives integral representations for solutions
of certain combinatorial problems.

Our approach is explained in detail in Section~2.   The integrands in our integral representations
are constant multiples of products of orthogonal polynomials times a measure with respect to which
the polynomials are orthogonal.  The integral representations arise naturally through separation of
variables of the solution of systems of difference equations satisfied by the combinatorial numbers.
We may reverse the process by starting with integrals of products of orthogonal polynomials times their
orthogonality measure and reach the combinatorial numbers.  Some of these integrals arose in problems
involving linearizations of products of orthogonal polynomials where the focus of attention was their
nonnegativity~\cite{Ask:Ism:Koo, Koo}.  Most of the positivity results originated from work by Askey and his coauthors in the late
1960's and 1970's. For references we refer  the interested reader to Askey's monograph \cite{Ask}, and
to Ismail's book \cite{Ism2}.

The integral representations studied in this work are of the form
\be\label{eqIS}
\int_\R  f(x) \prod_{j=1}^m p_{n_j}(\la_j x) d\mu(x),
\ee
where $\mu$ is a discrete or absolutely continuous measure and $f$ is some integrable function. Ismail and Simeonov~\cite{Ism:Sim}  studied the large
$k$ behavior of integrals of the form~\eqref{eqIS} when the $n_j$'s are all equal.
Since the integral in \eqref{eqIS} represents the number of  ways a certain configuration occurs, one can calculate
the probability that such configuration occurs. We shall also study integrals of the type
\eqref{eqIS} where the polynomials $p_{n_j}(x)$ come from two different families of orthogonal polynomials.
The positivity results which we establish are not only  new but seem to be the first of its type.

The rest of this paper is organized as follows.
As we already mentioned in the above paragraph our approach is outlined in Section~2, where we characterize
the linearization coefficients of orthogonal polynomials as the unique solution of  some
partial differential equations with boundary conditions.
Then we  apply the results of Section~2 to various combinatorial problems in Sections~3--8. More precisely,
by  solving the  corresponding partial difference equations combinatorially,
 we deduce the combinatorial interpretations of Hermite  and Charlier polynomials, Laguerre polynomials, Meixner polynomials, Meixner--Pollaczek polynomials, $q$-Hermite polynomials,
 $q$-Charlier polynomials, and
 $q$-Laguerre polynomials, respectively.  In each case we start with a combinatorial problem involving
 multisets, deduce a difference equation for the combinatorial numbers involved, then identify the
 orthogonal polynomials  which arise through the machinery developed in Section~2.
 Furthermore,  in Section~9, we extend the previous results to some more general integrals  to include
 the moments, inverse coefficients and linearization coefficients. We also compute the corresponding
 generating functions for
 the corresponding integrals of Lagurre and Meixner polynomials and deduce their combinatorial
 interpretation by applying MacMahon's Master theorem.
In Section~10, we give a further extension of the integrals of Laguerre and Meixner polynomials.
Finally, in Section~11,  we  prove the  crucial step, Lemma~\ref{lem:sym-charlier},   towards the
combinatorial solution of the partial difference equations of $q$-Charlier polynomials.

We follow the standard notation for shifted  factorials, hypergeometric functions and their $q$-analogues as in the books \cite{And:Ask:Roy, Gas:Rah, Ism2}.
The work of  Koekoek--Lesky--Swarttouw~\cite{Koe:Swa} is also a standard reference for formulas involving orthogonal polynomials and their basic analogues.

  \section{Separation of variables and  linearization coefficients}
  Let $\{p_n(x)\}$ be a sequence of orthogonal polynomials
  \begin{equation}
  \int_\R p_m(x) p_n(x) d\mu(x) = \zeta_n \delta_{m,n}, \quad \zeta_0=1.
  \label{eqorth}
  \end{equation}
  The condition $\zeta_0=1$ amounts to  normalizing  total mass of $\mu$ to be  1.
Then the polynomials $\{p_n(x)\}$ must satisfy a three term recurrence relation of the form
  \begin{equation}
 \label{eqgen3trr}
 p_{n+1}(x) = [A_n x + B_n] \; p_n(x) - C_n p_{n-1}(x),  \quad n >0,
  \end{equation}
and we will always assume $p_0(x) := 1, p_1(x) = A_0 x+B_0$.  Therefore
  \begin{equation}
\label{eqzetan}
\zeta_n = \frac{A_0}{A_n} C_1 C_2 \cdots C_n.
  \end{equation}

We consider the linearization coefficients in the expansion of $\prod_{j=1}^{m-1} p_{n_j}(\la_j x)$ in $\{p_n(x)\}$. Equivalently we consider the numbers
\be
\label{eqdefnI}
I({\n}) := \int_\R\biggl( \prod_{j=1}^{m-1} p_{n_j}(\la_j x)\biggr)p_{n_m}(x) \; d\mu(x),
\ee
where  ${\n}=(n_1, \ldots, n_m)$, $n_j$ is a nonnegative integer for $1\le j \le m$.    
We  shall use the following notation:
$$
I_j^{\pm}({\n})=I(n_1, \ldots, n_{j-1}, n_j\pm 1, n_{j+1}, \ldots, n_m).
$$
Moreover we assume that $\la_m = 1$.  It is clear that
\be
\begin{split}
I_j^+(0,\ldots, 0, n) = \la_j C_1\, \frac{A_0}{A_1}\;  \delta_{n, 1} +  B_0(1-\la_j) \;
\delta_{n,0}, \quad   \textup{if} \quad n=0, 1,\quad \textup{and}\\
I(0,\ldots, 0)=1, \quad I(\n) = 0 \quad  \textup{if} \quad \sum_{j=1}^{m-1} n_j < n.
\end{split}
\ee
\begin{thm}
The numbers $I({\n})$ satisfy the system of difference equations
\begin{equation}\label{eqIplusminus}
I_j^{+}({\n}) -  u_{jk}(\n) I_k^+({\n}) \\
= \left[B_{n_j}   - u_{jk}  B_{n_k}\right] I({\n}) -C_{n_j}I_j^-({\n})
+  u_{jk}(\n)  C_{n_k}I_k^-({\n}),
\end{equation}
where $u_{jk}(\n)=v_j(\n)/v_k(\n)$ and $v_j(\n)=A_{n_j}\la_j$.
\end{thm}
\begin{proof}
For $1\leq t\leq m$, we have by \eqref{eqgen3trr}
\begin{align*}
I_t^{+}({\n})&= \int_\R [(A_{n_t} \la_t x +B_{n_t}) \, p_{n_t}(\la_t
x)-C_{n_t}p_{n_t-1}(\la_t
x)] \prod_{r\neq t}p_{n_r}(\la_r x) d\mu(x)\\
&=v_t(\n) \int_\R   x\prod_{r=1}^m p_{n_r}(\la_r x) d\mu(x)
+B_{n_t}I(\n)- C_{n_t} I_t^{-}({\n}).
\end{align*}
Specializing the above equation at $t=j$ and $t=k$ immediately leads
to \eqref{eqIplusminus}.
\end{proof}
Observe that in the system  \eqref{eqIplusminus} we assume $j, k \ge 1$. It is more convenient to write
\eqref{eqIplusminus} in the more symmetric form
\begin{equation}
\label{eqIplusminus2}
\frac{1}{v_j(\n)}I_j^{+}({\n}) -  \frac{1}{v _k(\n)}   I_k^+({\n})
= \biggl[\frac{B_{n_j}}{v_j(\n)}    - \frac{B_{n_k}}{v_k(\n)} \biggr] I({\n}) -
\frac{C_{n_j}}{v_j(\n)}       I_j^-({\n})
+   \frac{C_{n_k}} {v _k(\n)}  I_k^-({\n}).
\end{equation}
We will show that the system \eqref{eqIplusminus2} describes many combinatorial problems.
From now on we will consider different combinatorial problems and derive a system of equations
of the type \eqref{eqIplusminus2} for the combinatorial numbers under consideration. Theorems~\ref{existence}
and \ref{Uniqueness} identify the combinatorial numbers as integrals of products of  orthogonal polynomials.
\begin{thm} \label{existence}
One solution to
\begin{equation}\label{eqIplusminusy}
y_j^{+}({\n}) -  u_{jk}(\n)  y_k^+({\n})
= \bigl[B_{n_j}   - u_{jk}(\n)B_{n_k}\bigr] y({\n}) -C_{n_j}y_j^-({\n})
+  u_{jk}(\n)  C_{n_k}y_k^-({\n}),
\end{equation}
is given by
\be
y({\n}) =
 \int_\R \prod_{j=1}^m p_{n_j}(\la_j x) \; d\nu(x),
 \label{eqallsol}
 \ee
 for any measure $\nu$ having finite moments of all orders.
\end{thm}
\begin{proof}
We try the separation of variables $
 y({\n})=\prod_{j=1}^m F_j(n_j)$.  When we substitute in \eqref{eqIplusminusy} we get
  \begin{align} \label{eqsep}
 \frac{F_j(n_j+1)}{v_j(\n) F_j(n_j)} +  \frac{C_{n_j}F_j(n_j-1)}{v_j(\n)F_j(n_j)}
 -\frac{B_{n_j}}{v_j(\n)} = \frac{F_k(n_k+1)}{v_k(\n)F_k(n_k)} +
\frac{C_{n_k}F_k(n_k-1)}{v_k(\n)F_k(n_k)}
 -\frac{B_{n_k}}{v_k(\n)}.
 \end{align}
 Thus each side of the above equation is a constant independent of $j$ or $k$, so we denote the constant by  $x$.
 This leads to the difference equation
 $$
 F_j(n_j+1)=(\la_j A_{n_j}x+B_{n_j}) F_j(n_j)-C_{n_j}F_j(n_j-1),
 $$
 and the $F's$ now depend on $x$. Comparing with  \eqref{eqgen3trr} and noting
 that $ F_j(-1) =0$ and  $ F_j(0) =1$,
 we see that the above recurrence relation has a solution given by
 $F_j(n_j) = p_{n_j}(\la_j x)$ and by the principle of linear superposition
 the function in \eqref{eqallsol} is a solution.
\end{proof}
\begin{thm}\label{Uniqueness}
The system of equations \eqref{eqIplusminusy} and the boundary conditions
\begin{equation}
\label{eqsys}
\begin{split}
y_j^+(0,\ldots, 0, n) = \la_j C_1\, \frac{A_0}{A_1}\;  \delta_{n, 1} +  B_0(1-\la_j) \;
\delta_{n,0}, \quad   \textup{if} \quad n=0, 1,\quad \textup{and}\\
y(0,\ldots, 0)=1, \quad y(\n) = 0\quad  \textup{if} \quad \sum_{j=1}^{m-1} n_j < n,
\end{split}
\end{equation}
have a unique solution which is given by  \eqref{eqdefnI}.
\end{thm}
\begin{proof}
We know that  the multisequence \eqref{eqdefnI} satisfies the system of equations \eqref{eqIplusminusy}
and boundary condition  \eqref{eqsys}, hence a solution exists.
The  second  boundary condition defines $y$  for $n\geq 0$ and when the rest  are zero.
The first boundary condition in  \eqref{eqsys} defines $y$ for  $n$ and when one other entry
$=1$ and the rest are zero in a unique way.
Letting ${\n} = (0, \ldots, 1, 0, \ldots, n)$, $n>0$ in  \eqref{eqIplusminusy} with $k =m$ we evaluate
$y(0, \ldots, 2, 0, \ldots, n)$ and by induction we evaluate $y(0, \ldots, n_s, 0, \ldots, n), 1 \le s < m$.
Next we use \eqref{eqIplusminusy} to evaluate $y$ for general $n_r, n$
 and another nonzero $n_s$: if $n > n_r +1$, then $y({\n})$ with nonzero entries in the positions $r,  m$
 of $\n$  is zero when we have 1 in the  position $s$;
if $n \leq n_r +1$, we use  \eqref{eqIplusminusy} with $j =s, k =m$ to  evaluate $y$.  Thus
we can evaluate $y$  inductively  when  $\n$ has three nonzero entries.
We continue this argument until we reach any desired general ${\n}$.
\end{proof}

\begin{rem}
It is important to note that \eqref{eqIplusminusy} is satisfied by solutions of the form \eqref{eqallsol}
where $\nu$ is any probability measure with finite moments. It is the boundary conditions \eqref{eqsys} that force $\nu$ to be an
orthogonality measure of $\{p_n(x)\}$.
\end{rem}

An important class of orthogonal polynomials is the class of birth and death process polynomials. They are generated by
\eqref{eqBandDPol}.
These polynomials have only positive  zeros so they  are orthogonal
with respect to a probability measure supported on a subset of  $[0, \infty)$.
The idea of separation of variables is also used to solve the differential-difference equations describing this model, see \S 5.2 and Theorem 7.2.1 in \cite{Ism2}.
Birth and death processes have many applications in applied probability and queueing theory.

An immediate consequence of Theorem \ref{Uniqueness} is the following result for the polynomials $\{Q_n(x)\}$ generated by
\eqref{eqBandDPol}.

\begin{thm}
The system of difference  equations
\begin{equation}
\label{eqIplusminusBandD}
\frac{b_{n_j}}{\la_j} y_j^{+}({\n}) -  \frac{b_{n_k}}{\la _k}   y_k^+({\n})
= \left[\frac{b_{n_j}d_{n_j}}{\la_j}    - \frac{b_{n_k} d_{n_k}}{\la _k}  \right] y({\n}) -
 \frac{d_{n_j}}{\la_j}y_j^-({\n})
+   \frac{d_{n_k}}{\la _k}  y_k^-({\n}),
\end{equation}
and the boundary conditions
\begin{equation}\label{eqsysBandD}
\begin{split}
y_j^+(0,\ldots, 0,n) = \la_j\, \frac{d_1}{b_0}\;  \delta_{n, 1} + \left(1+ \frac{d_0}{b_0}\right) \;
\delta_{n,0}, \quad   \textup{if} \quad n=0, 1,\; \textup{and}\;  \\
y(0,\ldots, 0)=1, \quad y(\n) = 0 \quad  \textup{if} \quad \sum_{j=1}^{m-1} n_j < n,
\end{split}
\end{equation}
have a unique  solution which is given by
\be
 y({\n}) = \int_0^\infty  \prod_{j=1}^m Q_{n_j}(\la_j x) \; d\mu(x),
\ee
where $\mu$ is an orthogonality measure for the polynomials $\{Q_n(x)\}$ in \eqref{eqBandDPol}.
\end{thm}

From combinatorial point of view, sometimes it is easier to establish a different  kind of difference equations from \eqref{eqIplusminus}.
Since   $p_1(x)=A_0x+B_0$  and
$p_{n+1}(x)=[A_nx+B_n]p_n(x)-C_np_{n-1}(x)$, we have
\begin{align}
\la_jp_1(x)p_{n_j}(\la_jx)&=\frac{A_0}{A_{n_j}}p_{n_j+1}(\la_jx)\\
&\quad+\Big(\la_jB_0-\frac{A_0}{A_{n_j}}B_{n_j}\Big)p_{n_j}(\la_jx)+\frac{A_0}{A_{n_j}}C_{n_j}p_{n_j-1}(\la_jx).\nonumber
\end{align}
Substituting in \eqref{eqdefnI} yields
\begin{align}\label{eq1}
I(1,\n)=\frac{A_0}{\la_j A_{n_j}}I_j^+(\n)
+\bigg[B_0-\frac{A_0}{A_{n_j}}\frac{B_{n_j}}{\la_j}\bigg]I(\n)
+\frac{A_0}{A_{n_j}}\frac{C_{n_j}}{\la_j}I_j^-(\n).
\end{align}
Subtracting \eqref{eq1}   from itself  with $j$ replaced by  $k$,  we obtain \eqref{eqIplusminus}.
For the Laguerre polynomials, $q$-Charlier polynomials and $q$-Laguerre polynomials, we shall first establish
combinatorially \eqref{eq1} before passing to  \eqref{eqIplusminus}. Finally we have the following result.

\begin{thm} \label{thmbis} Let $\n=(n_1, \ldots, n_m)$ with  $m\geq 1$ and $n_1, \ldots, n_m\geq 0$.
Any sequence $I(\n)$ satisfying the system \eqref{eq1} is  uniquely determined
by its special values at $\n=(1, \ldots, 1)$ and the
symmetry with respect to the indices $n_1, \ldots, n_m$.
\end{thm}

\begin{rem} Evaluating
the special values  of the linearization coefficients at $\n=(1, \ldots, 1)$ amounts to computing
the moments of the corresponding orthogonal polynomials, while the boundary
condition \eqref{eqsysBandD} is much easier to check and does not need  the knowledge of the moments,
though the latter would be a source of inspiration for the linearization coefficients.
\end{rem}


\section{Linearization coefficients of Hermite and Charlier  polynomials}
In this section we consider the linearization coefficients of
Hermite and Charlier  polynomials.  We  start with some  combinatorial setup, which will also be used in the later sections.

\subsection{Combinatorial definitions}

 In the sequel, we denote by $\M_n$, $\Pi_n$ and $\s_n$ the set of perfect matchings, of partitions and of permutations, respectively,
of $[n]:=\{1,2,\ldots,n\}$. Recall that a perfect matching of $[n]$
is just a set partition of $[n]$ the blocks of which have exactly
two elements.
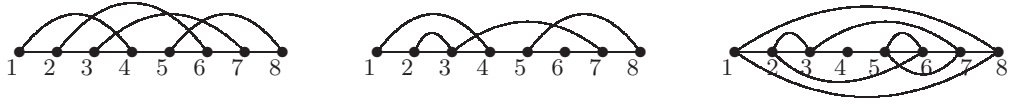
\begin{figure}[h!]
\unitlength=1mm
\begin{picture}(35,15)(0,-5)
\put(0,0){\line(1,0){35}}
\qbezier(0,0)(7.5,10)(15,0)\qbezier(5,0)(15,13)(25,0)
\qbezier(10,0)(20,10)(30,0)\qbezier(20,0)(27.5,10)(35,0)
\put(0,0){\circle*{1,3}}\put(0,0){\makebox(-2,-4)[c]{\footnotesize1}}
\put(5,0){\circle*{1,3}}\put(5,0){\makebox(-2,-4)[c]{\footnotesize2}}
\put(10,0){\circle*{1,3}}\put(10,0){\makebox(-2,-4)[c]{\footnotesize3}}
\put(15,0){\circle*{1,3}}\put(15,0){\makebox(-2,-4)[c]{\footnotesize4}}
\put(20,0){\circle*{1,3}}\put(20,0){\makebox(-2,-4)[c]{\footnotesize5}}
\put(25,0){\circle*{1,3}}\put(25,0){\makebox(-2,-4)[c]{\footnotesize6}}
\put(30,0){\circle*{1,3}}\put(30,0){\makebox(-2,-4)[c]{\footnotesize7}}
\put(35,0){\circle*{1,3}}\put(35,0){\makebox(-2,-4)[c]{\footnotesize8}}
\end{picture}
\hspace{1cm}
\begin{picture}(35,15)(0,-5)
\put(0,0){\line(1,0){35}}
\qbezier(0,0)(7.5,10)(15,0)\qbezier(5,0)(7,5)(10,0)
\qbezier(10,0)(20,8)(30,0)\qbezier(20,0)(27.5,10)(35,0)
\put(0,0){\circle*{1,3}}\put(0,0){\makebox(-2,-4)[c]{\footnotesize1}}
\put(5,0){\circle*{1,3}}\put(5,0){\makebox(-2,-4)[c]{\footnotesize2}}
\put(10,0){\circle*{1,3}}\put(10,0){\makebox(-2,-4)[c]{\footnotesize3}}
\put(15,0){\circle*{1,3}}\put(15,0){\makebox(-2,-4)[c]{\footnotesize4}}
\put(20,0){\circle*{1,3}}\put(20,0){\makebox(-2,-4)[c]{\footnotesize5}}
\put(25,0){\circle*{1,3}}\put(25,0){\makebox(-2,-4)[c]{\footnotesize6}}
\put(30,0){\circle*{1,3}}\put(30,0){\makebox(-2,-4)[c]{\footnotesize7}}
\put(35,0){\circle*{1,3}}\put(35,0){\makebox(-2,-4)[c]{\footnotesize8}}
\end{picture}
\hspace{1cm}
\begin{picture}(35,15)(0,-5)
\put(0,0){\line(1,0){35}}
\qbezier(0,0)(17.5,12)(35,0)\qbezier(5,0)(7,5)(10,0)\qbezier(20,0)(22,5)(25,0)
\qbezier(10,0)(20,8)(30,0)
\qbezier(5,0)(12.5,-8)(25,0)\qbezier(20,0)(25,-6)(30,0)
\qbezier(0,0)(17.5,-12)(35,0)
\put(0,0){\circle*{1,3}}\put(0,0){\makebox(-2,-4)[c]{\footnotesize1}}
\put(5,0){\circle*{1,3}}\put(5,0){\makebox(0,-4)[c]{\footnotesize2}}
\put(10,0){\circle*{1,3}}\put(10,0){\makebox(-1,-4)[c]{\footnotesize3}}
\put(15,0){\circle*{1,3}}\put(15,0){\makebox(-2,-4)[c]{\footnotesize4}}
\put(20,0){\circle*{1,3}}\put(20,0){\makebox(-3,-4)[c]{\footnotesize5}}
\put(25,0){\circle*{1,3}}\put(25,0){\makebox(1,-4)[c]{\footnotesize6}}
\put(30,0){\circle*{1,3}}\put(30,0){\makebox(1.5,-4)[c]{\footnotesize7}}
\put(35,0){\circle*{1,3}}\put(35,0){\makebox(1,-4)[c]{\footnotesize8}}
\end{picture}
\caption{Diagrams of, from left to right, the matching $M=1\,4/2\,6/3\,7/5\,8$, the partition  $\pi=1\,4/2\,3\,7/5\,8/6$ and the
permutation $\sigma=8\, 3\, 7\,4\,6 \,2\,5\,1$}\label{fig:diagrams}
\end{figure}
 It is often convenient to represent pictorially set  partitions and permutations of $[n]$.
We first draw $n$ elements on a line labeled $1,2,\ldots,n$ in
increasing order. Then, the diagram of a partition of $[n]$ is
obtained by joining \textit{successive}  elements of each block by
arcs drawn in the upper half-plane. Here, we say that two elements
$i<j$ in the block $B$ are successive, or more precisely that $j$
\emph{follows} $i$, if there is no element $p\in B$ such that
$i<p<j$.
 We denote by $(i,j)$ the arc whose extremities are $i$ and $j$.
The diagram of a  permutation $\sigma\in\s_n$ is obtained by drawing
an arc $i\to \sigma(i)$ above (resp. under) the line if
$i<\sigma(i)$ (resp. $i> \sigma(i)$). Arcs are always drawn in a way
such that any two arcs cross at most once.

 In  what follows, we fix an $m$-tuple of  nonnegative  integers $\n=(n_1,\ldots, n_m)$
 such that $n=n_1+\cdots +n_m$ and partition the $n$ balls $\{1, \ldots, n\}$ into $m$ boxes $S_1, \ldots, S_m$
 where   $ S_j=\{n_1+\cdots +n_{j-1}+1, \ldots, n_1+\cdots +n_{j}\}$, $n_0=0$,  for $j=1, \ldots, m$.
 We denote by  $[\n]$ the set $\{1, \ldots, n\}$ with underlying 
 boxes $S_1,\ldots, S_m$, and the corresponding sets of matching, partitions and permutations by
 \be
 \M(\n):=\M_n,  \quad \Pi(\n):=\Pi_n\quad  \textrm{and} \quad \s(\n):=\s_n.
\ee
   A partition $\pi$ of $[\n]$ is said to be \emph{inhomogeneous}  if each  block of $\pi$
contains at least two elements and no  two elements in the same
block belong to  the same box  $S_i$ ($1\leq i\leq m$). Similarly, a
permutation $\sigma$  of $[\n]$ is an \emph{inhomogeneous
derangement} if $\sigma(S_i)\cap S_i=\emptyset$ for all $i\in [m]$.
We let ${\mathcal K}(\n)$ (resp., $\P(\n)$ and ${\mathcal D}(\n)$)
denote  the set of inhomogeneous perfect matchings
(resp., partitions and derangements) of $[\n]$.  Note that a set
partition (resp., permutation)  is inhomogeneous if and only in its
diagram, there is no isolated vertex and no arc connecting two
elements in the same \emph{box} $S_j$ ($1\leq j\leq m$). For instance, if
$\n=(2,3,3)$, then in Figure~\ref{fig:diagrams} the matching drawn
is in $\K(\n)$ while the partition and the permutation are not in
$\P(\n)$ and $\D(\n)$ (they have isolated points). Inhomogeneous
objects are drawn in Figure~\ref{fig:diagrams-inhomogeneous}.

\begin{figure}[h]
\unitlength=1mm
\begin{picture}(35,15)(0,-5)
\put(0,0){\line(1,0){5}}\put(10,0){\line(1,0){10}}\put(25,0){\line(1,0){10}}
\qbezier(0,0)(7.5,10)(15,0)\qbezier(5,0)(12.5,12)(20,0)\qbezier(10,0)(15.5,12)(25,0)
\qbezier(15,0)(25,12)(35,0)\qbezier(20,0)(25,7)(30,0)
\put(0,0){\circle*{1,3}}\put(0,0){\makebox(-2,-4)[c]{\footnotesize1}}
\put(5,0){\circle*{1,3}}\put(5,0){\makebox(-2,-4)[c]{\footnotesize2}}
\put(10,0){\circle*{1,3}}\put(10,0){\makebox(-2,-4)[c]{\footnotesize3}}
\put(15,0){\circle*{1,3}}\put(15,0){\makebox(-2,-4)[c]{\footnotesize4}}
\put(20,0){\circle*{1,3}}\put(20,0){\makebox(-2,-4)[c]{\footnotesize5}}
\put(25,0){\circle*{1,3}}\put(25,0){\makebox(-2,-4)[c]{\footnotesize6}}
\put(30,0){\circle*{1,3}}\put(30,0){\makebox(-2,-4)[c]{\footnotesize7}}
\put(35,0){\circle*{1,3}}\put(35,0){\makebox(-2,-4)[c]{\footnotesize8}}
\end{picture}
\hspace{2cm}
\begin{picture}(35,15)(0,-5)
\put(0,0){\line(1,0){5}}\put(10,0){\line(1,0){10}}\put(25,0){\line(1,0){10}}
\qbezier(0,0)(7.5,12)(15,0)\qbezier(15,0)(25,12)(35,0)\qbezier(5,0)(7,5)(10,0)
\qbezier(20,0)(22,5)(25,0)\qbezier(10,0)(20,8)(30,0)
\qbezier(5,0)(12.5,-8)(25,0)\qbezier(20,0)(25,-6)(30,0)
\qbezier(0,0)(17.5,-12)(35,0)
\put(0,0){\circle*{1,3}}\put(0,0){\makebox(-2,-4)[c]{\footnotesize1}}
\put(5,0){\circle*{1,3}}\put(5,0){\makebox(0,-4)[c]{\footnotesize2}}
\put(10,0){\circle*{1,3}}\put(10,0){\makebox(-1,-4)[c]{\footnotesize3}}
\put(15,0){\circle*{1,3}}\put(15,0){\makebox(-2,-4)[c]{\footnotesize4}}
\put(20,0){\circle*{1,3}}\put(20,0){\makebox(-3,-4)[c]{\footnotesize5}}
\put(25,0){\circle*{1,3}}\put(25,0){\makebox(1,-4)[c]{\footnotesize6}}
\put(30,0){\circle*{1,3}}\put(30,0){\makebox(1.5,-4)[c]{\footnotesize7}}
\put(35,0){\circle*{1,3}}\put(35,0){\makebox(1,-4)[c]{\footnotesize8}}
\end{picture}
\caption{Diagrams of, from left to right, a partition in $\P(2,3,3)$
and permutation in $\D(2,3,3)$}\label{fig:diagrams-inhomogeneous}
\end{figure}
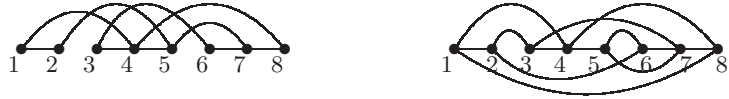

\subsection{Hermite polynomials and inhomogeneous matchings}
The Hermite polynomials
$\{H_n(x)\}_{n\geq 0}$ can be defined by one of the following five equivalent conditions:
\begin{enumerate}
\item (Coefficients) $H_n(x)=\sum_{0\leq 2k\leq n} (-1)^k \frac{n!}{2^kk!(n-2k)!} (2x)^{n-2k}$.
\item  (Generating function)
$
\sum_{n=0}^\infty H_n(x)\frac{t^n}{n!}=\exp(2xt-t^2).
$
\item (Orthogonality relation)
$
\int_\R H_m(x)H_n(x) e^{-x^2} dx= 2^n n! \; \sqrt{\pi}\; \delta_{mn}.
$
\item (Recurrence relation)
$2xH_n(x)= H_{n+1}(x)+2nH_{n-1}(x)$,  with $H_{-1}(x)=0,\;  H_0(x) =1$.
\item (Moments) $\mu_{2n+1}=0$, $\mu_{2n}=1\cdot 3\cdots (2n-1)/2^n$.
\end{enumerate}

Let $K(\n)$  be the number of inhomogeneous perfect matchings of $[\n]$.
 \begin{lem}\label{hermite}
 For  $k, j\in [m]$ and $k\neq j$ the numbers  $K(\n)$ satisfy
\be
\label{eqmatchequation}
 K_j^+({\n})- K_k^+({\n})=n_k K_k^-({\n})-n_j K_j^-({\n}),
 \ee
 and the boundary condition \eqref{eqsys} with $\la_j=1$ for all $j$ and $A_0C_1=A_1$.
\end{lem}
\begin{proof} Let $r\in [m]$  and set $N_r=n_1+\cdots +n_r$.
For any $i\neq r$, the number of matchings in~$\K_r^+(\n)$ in which
 $N_r+1$ is matched with an element in $S_i$ is clearly $n_i K_i^{-}(\n)$. This implies that
for any $r\in [m]$, we have
$$
K_r^+(\n)=\sum_{\substack{ i=1\\
   i \neq r
  }}^m n_i K_i^{-}(\n),
$$
from which we  immediately deduce~\eqref{eqmatchequation}. The
boundary conditions in  \eqref{eqsys} are  obviously satisfied.
\end{proof}

\begin{thm}\label{thmAzGiVi}
The numbers  $K({\n})$ have the following integral representation
\begin{align}
\label{eqAzGiVi}
K({\n})=2^{-(n_1+\cdots +n_m)/2} \int_\R \frac{e^{-x^2}}{\sqrt{\pi}}\prod_{j=1}^mH_{n_j}(x)dx.
\end{align}
\end{thm}
\begin{proof}
 When $\la_j =1$ for all $j$, and
 $$
 A_n=1, \quad B_n=0,\quad  C_n=n \quad \textrm{for all $n$},
$$
by Lemma~\ref{hermite}, the numbers $K({\n})$ satisfy \eqref{eqIplusminusy} and \eqref{eqsys}.
On the other hand,  the corresponding orthogonal polynomials $\H_n(x)$
satisfy the
 recurrence  relation
\begin{align}\label{eq:threenorm}
x\H_n(x)=\H_{n+1}(x)+n\H_{n-1}(x),  \quad\textrm{with}\quad   \H_{-1}(x)=0, \; \H_0(x) =1.
\end{align}
Hence, these 
are the normalized Hermite  polynomials  $\tilde H_{n}(x)=2^{-n/2}H_{n}(x/\sqrt{2})$
and their orthogonality relation is
\begin{align}\label{eq:normal-hermite}
\int_\R \H_m(x)\H_n(x) \frac{e^{-x^2/2}}{\sqrt{2\pi}} dx= n! \; \delta_{mn}.
\end{align}
Therefore
\begin{align*}
K(\n)=\int_\R \frac{e^{-x^2/2}}{\sqrt{2\pi}} \prod_{j=1}^m 2^{-n/2} H_{n_j}(x/\sqrt{2})dx,
\end{align*}
which is equal to \eqref{eqAzGiVi}.
\end{proof}

 Note that the
 \emph{exponential formula} (see \cite[Corollary 5.1.6]{Sta})  implies that
 \begin{align}
 \sum_{n_1, \ldots, n_m\geq 0} K(\n)\frac{x_1^{n_1}}{n_1!}\cdots \frac{x_m^{n_m}}{n_m!}=
 \exp\bigg(\sum_{ i<j}x_ix_j\bigg).
 \end{align}
Hence \eqref{eqAzGiVi}  can also be proved from the generating function of Hermite polynomials.

\subsection{Charlier polynomials and inhomogeneous partitions}

The Charlier polynomials  $C_n^{(a)}(x) $ can be defined by one of the following
five  equivalent conditions:
\begin{enumerate}
\item (Explicit  formula)
$C_n^{(a)}(x)=\sum_{k=0}^n {n\choose k} {x\choose k} k! (-a)^{n-k}$.
\item (Generating function)
$
\sum_{n=0}^\infty C_n^{(a)}(x) \frac{w^n}{n!}=e^{-aw}(1+w)^x.
$
\item (Orthogonality)
$
\int_0^\infty C_n^{(a)}(x)C_m^{(a)}(x)d\psi^{(a)}(x)=a^nn!\delta_{mn},
$
where $\psi^{(a)}$ is the step function of which the jumps at the points $x=0, 1, \ldots$ are
$
\psi^{(a)}(x)=\frac{e^{-a}a^x}{x!}.
$
\item (Recursion relation)
$
C_{n+1}^{(a)}(x)=(x-n-a)C_n^{(a)}(x)-anC_{n-1}^{(a)}(x).
$
\item (Moments) $\mu_n=\sum_{k=1}^n S(n,k)a^k$, where $S(n,k)$ are the Stirling numbers of the second kind.
\end{enumerate}

The number of blocks of a set partition $\pi$ is denoted by $\bl(\pi)$.
Consider the enumerative polynomial of inhomogeneous  partitions
\be
C(\n;a)=\sum_{\pi\in \P(\n)}a^{\bl( \pi)}.
\ee
\begin{lem}\label{lemCharlier} For  $k, j\in [m]$ and $k\neq j$ the polynomials $C(\n;a)$ satisfy
\begin{align}\label{eq:charlier}
C_j^{+}(\n; a)
-C_k^+(\n; a)=(n_k-n_j) C(\n;a)+an_k C_k^{-}(\n;a)-an_jC_j^{-}(\n;a),
\end{align}
and the boundary condition \eqref{eqsys} with $\la_j=1$ for all $j$ and $A_0C_1=A_1$.
\end{lem}
\begin{proof}  Let $N_j=n_1+\cdots +n_j$.
The partitions of $\P_j^{+}(\n)$ can be  divided into three categories:
\begin{itemize}
\item $N_j+1$ and one element of $S_k$ form a block of two elements, the corresponding generating function is
 $a  n_k C_k^{-}(\n;a)$;
\item $N_j+1$ and one element of $S_k$  belong to a block  containing at least one another element,
the corresponding generating function is  $\sum_{\pi\in \P(\n)}(n_k-n_{k,j}(\pi))a^{\bl( \pi)}$, where $n_{k,j}(\pi)$ 
is the number of blocks in $\pi$ containing both elements of $S_j$ and $S_k$
(clearly  $n_{k,j}(\pi)=n_{j,k}(\pi)$);
\item   $N_j+1$ is in a block without any element of $S_k$, let $R_{k,j}(\n;a)$ be the corresponding generating function.
\end{itemize}
Thus we have
\begin{align}\label{eqdiff-char}
C_j^+(\n;a)=\sum_{\pi\in \P(\n)}(n_k-n_{k,j}(\pi))a^{\bl( \pi)}+a n_k C_k^{-}(\n;a)+R_{k,j} (\n;a).
\end{align}
Exchanging $k$ and $j$ in the latter identity and subtracting the resulting identity from  the latter identity,
we obtain  \eqref{eq:charlier}  in view of the symmetry relation $R_{k,j}(\n;a)=R_{j,k}(\n;a)$.
This relation can be easily proved, for instance  by observing that a partition in $\P_j^{+}(\n)$
can be seen as an inhomogeneous partition
of the union $S_1\cup S_2\cup \cdots\cup S_m$ with $S_i=[N_{i-1}+1,N_i]$ for $i\neq j$ and $S_j=[N_{j-1}+1,N_j]\cup \{x\}$
where $x$ is any object which is not in $[n]$.
\end{proof}

\begin{rem} We  can also argue as follows.
Let $\n^{*}=(n_{1}, \ldots, n_{j}, 1, n_{j+1}, \ldots, n_{m})$ with $j\in [m]$.  It is fairly easy to show that
\begin{align}\label{eqChar}
C(\n^{*}; a)=C_j^{+}(\n; a)+an_{j}C_j^{-}(\n; a)+n_{j}C(\n; a).
\end{align}
Subtracting the above identity from \eqref{eqChar} with $j=k$  yields immediately \eqref{eq:charlier}.
We will use this argument for the Laguerre and Meixner polynomials in the next sections.
\end{rem}

We can solve the system \eqref{eq:charlier} by applying the method of separation of variables
which naturally leads to the  Charlier polynomials.
\begin{thm} The polynomials $C(\n;a)$ have the following integral representation
\begin{align}\label{zen}
C(\n;a)=\int_0^\infty C_{n_1}^{(a)}(x)\cdots C_{n_m}^{(a)}(x) d\psi^{(a)}(x).
\end{align}
\end{thm}
\begin{proof}
Clearly
\eqref{eqIplusminusy} reduces to \eqref{eq:charlier}  when $\la_j  =1$ for all $j$, and
$$
A_n=1,\quad
B_n=-n-a, \quad
C_n=an  \quad\textrm{for all}\quad n\geq0.
$$
From Lemma~\ref{lemCharlier} and Theorem~\ref{Uniqueness} we deduce \eqref{zen}.
\end{proof}
The above formula was first established  by Zeng{~}\cite{Zen90} using the generating function and  the exponential formula.
A different proof was given by Gessel~\cite{Ges} using rook polynomials.

\section{Linearization coefficients of Laguerre polynomials}
The shifted factorials are
\be
(a)_0=1,\qquad (a)_n=a(a+1)\cdots (a+n-1), \quad n>0.
\ee
The Laguerre polynomials are defined by
\begin{align}
L_n^{(\alpha)}(x)=\frac{(\al+1)_n}{n!} \sum_{k=0}^n \frac{(-n)_k}{k! (\al+1)_k}  x^k,
\end{align}
and have the generating function
\begin{align}\label{eqGFLag}
\sum_{n=0}^\infty  L_n^{(\alpha)}(x){t^n}=(1-t)^{-\alpha-1} \exp\Bigl(\frac{-xt}{1-t}\Bigr).
\end{align}
They satisfy the recurrence relation
\begin{align}
\label{eqLaguerre3trr}
(n+1)L_{n+1}^{(\alpha)}(x)-(2n+\alpha+1-x)L_n^{(\alpha)}(x)+(n+\alpha)L_{n-1}^{(\alpha)}(x)=0,
\end{align}
and the orthogonality
\begin{align}
\int_{0}^\infty \frac{x^\alpha e^{-x}}{\Gamma(\alpha+1)}L_m^{(\alpha)}(x) L_n^{(\alpha)}(x)\,dx=\frac{(\alpha+1)_n}{n!}\,\delta_{m,n}.
\end{align}
The moments are
\be
\mu_n=\frac{1}{\Gamma(\alpha+1)} \int_0^\infty x^{\alpha+n} e^{-x} dx=(\alpha+1)_n.
\ee

In this section we shall  prove the results of Foata and Zeilberger~\cite{Foa:Zei} about the Laguerre polynomials through our  method of separation of variables.
For $\pi\in \s(\n)$, we let $\Fix_{i}\pi=\pi(S_{i})\cap S_i$ for  $i\in [m]$.
For an $m$-tuple $\La=(\la_1,\ldots,\la_m)$, define
\begin{align}\label{eq:fzgeneral}
L(\n; \alpha, \La)=\sum_{\pi\in \s(\n)}(\alpha+1)^{\cyc(\pi)}\prod_{i=1}^m (\lambda_i-1)^{|\Fix_i\pi|} \lambda_i^{|S_i\setminus \Fix_i\pi]},
\end{align}
where $\cyc(\pi)$ is the number of cycles of $\pi$. By definition,
for an inhomogeneous permutation $\pi\in \D(\n)$ we have
$|\Fix_i\pi|=0$. Hence, when $\La={\bf 1}:=(1, \ldots, 1)$  the  summands
in \eqref{eq:fzgeneral}  reduce to $(\alpha+1)^{\cyc(\pi)}$ if
$\pi\in D(\n)$ and 0 otherwise. Thus, we have
\begin{align}
L(\n;\alpha, {\bf 1})=\sum_{\pi\in {\mathcal D}(\n)}(\alpha+1)^{\cyc(\pi)}.
\end{align}

\begin{lem}\label{lemlaguerrelam} For $j,k\in [m]$ such that $j\neq k$ the polynomials $L(\n; \alpha, \La)$ satisfy
\begin{multline}\label{eq:laguerrelam}
\la_kL_j^{+}(\n; \alpha, \La)-\la _j   L_k^+(\n; \alpha, \La) \\
= \left[(2n_k+\al +1) \la_j    - (2n_j+\al +1) \la _k \right] L({\n}; \alpha, \La) \\
+ n_k(n_k+\al){\la_j} L_k^-(\n; \alpha, \La)- n_j(n_j+\al){\la_k}L_j^-(\n; \alpha, \La).
\end{multline}
\end{lem}
\begin{proof} Let $n_0=\la_0=1$ and $\n^*=(n_0, n_1, \ldots, n_m)$ with $S_0=\{1^*\}$. Then
$\la_j L(\n^*; \alpha, \La)$ is the  generating function of $\sigma\in \s(\n^*)$ such that $\sigma(1^*)\neq 1^*$
and the edge $1^*\to \sigma(1^*)$ is weighted by $\la_j$.   We show that
\begin{multline}\label{eqLagPlusMinus4}
\la_j L(\n^*; \alpha, \La)
=L_j^{+}({\n}; \alpha, \La) -(\al +1)(\la_j-1)L({\n}; \alpha, \La) \\
+ 2n_jL({\n};\alpha, \La)+n_j(n_j+\al) L_j^-({\n};\alpha, \La).
\end{multline}
To do so,  we adjoin the element  $1^*$ to $S_j$.  Thus $(\al +1)(\la_j-1)L({\n};\alpha, \La)$ is the generating function of $\sigma\in \s_j^+(\n)$ such that
$\sigma(1^*)=1^*$.  Hence,  the difference
$$L_j^{+}({\n}; \alpha, \La) -(\al +1)(\la_j-1)L({\n};\alpha, \La)$$
 is the generating function of $\sigma\in \s_j^+(\n)$
such that $\sigma(1^*)\neq 1^*$, moreover
the edge $1^*\to \sigma(1^*)$ is weighted by $\la_j-1$ if $\sigma(1^*)\in S_j$ and $\la_j$
otherwise. To compensate the over counting,
we should add
\begin{itemize}
\item  the generating function of $\sigma\in \s_j^+(\n)$
such that $\sigma(1^*)\in S_j$ and
 the edge $1^*\to \sigma(1^*)$ is weighted by $1$;
\item  the generating function of $\sigma\in \s_j^+(\n)$
such that $\sigma^{-1}(1^*)\in S_j$ and
 the edge $\sigma^{-1}(1^*)\to 1^*$ is weighted by $1$.
\end{itemize}
For any $\sigma\in \s_j^+(\n)$, we let $a=\sigma(1^*)$ and $b=\sigma^{-1}(1^*)$.
There are four cases to consider.
\begin{enumerate}
\item $a\in S_j$ and $b\notin S_j$.  We can construct such a permutation $\sigma$  as follows:
 starting from a permutation $\tau\in \s(\n)$ and choosing a point $\xi\in S_j$, we define
 $\sigma(x)=\tau(x)$ if $x\neq 1^*, \tau^{-1}(\xi)$,  and $\sigma(1^*)=\xi$, $\sigma(\tau^{-1}(\xi))=1^*$
 As the weight of the edge $1^*\to \xi$ is 1 and that of $\tau^{-1}(\xi)\to 1^*$ in $\sigma$ is equal to that of $\tau^{-1}(\xi)\to \xi$ in
$\tau$, the weight of $\sigma$ is equal to that of $\tau$, hence
the generating function is $n_jL(\n; \alpha, \La)$.
\item $a\notin S_j$ and $b\in S_j$.  Similar to the above case,  the generating function is $n_jL(\n; \alpha, \La)$.
\item $a\in S_j$ and $b\in S_j$, but $a\neq b$.  Starting from $\sigma\in \s(\n^*)$ we can
 define the permutation  $\tau$ on $[n]\setminus \{a\}$ by
$\tau(j)=\sigma(j)$ for $j\neq a, b$ and $\tau(b)=\sigma(a)$.  Clearly $\cyc(\sigma)=\cyc(\tau)$. Inversely,  starting from      a permutation
$\tau\in \s_j^{-}(\n)$  there are $n_j(n_j-1)$ choices for
$a$ and $b$. Thus,
the corresponding generating function is $n_j(n_j-1)L_j^{-}(\n; \alpha, \La)$.
\item $a=b\in S_j$.
The generating function is $(\alpha+1)n_jL_j^{-}(\n; \alpha, \La)$.
\end{enumerate}
Summing up the above four cases we  obtain  \eqref{eqLagPlusMinus4}. Now,
substituting  $j$ by $k$  in \eqref{eqLagPlusMinus4} yields
\begin{multline}
\label{eqLagPlusMinus5}
\la_kL(\n^*; \alpha, \La)=L_k^{+}({\n};\alpha, \La) -(\al +1)(\la_k-1)L({\n};\alpha, \La) \\
+ 2n_kL({\n};\alpha, \La)+n_k(n_k+\al) L_k^-({\n};\alpha, \La).
\end{multline}
Multiplying  \eqref{eqLagPlusMinus4}  and   \eqref{eqLagPlusMinus5} by $\la_k$ and $\la_j$, respectively, and then subtracting,
we obtain  the  identity \eqref{eq:laguerrelam}.
\end{proof}

We need to state some preliminary results before proving the main result of this section.
Let $A$ and $B$ be two disjoint sets of cardinality $m$ and $n$, respectively. An injection $f$
from $A$ to $A\cup B$ can be depicted by a graph on $A\cup B$ such that there is an edge $x\to y$
if and only if $f(x)=y$. Hence the  connected components of the graph consists of
cycles, i.e., $x\to f(x)\to \cdots \to f^l(x)$ with $f^i(x)\in A$ and $f^l(x)=x$ and paths, i.e.,
$x\to f(x)\to \cdots \to f^l(x)$ with $f^l(x)\in B$.
Let $\cyc(f)$ be the number of cycles of $f$. Then, Foata and Strehl \cite{Foa:Str} proved
\be
\sum_{f: A\to A\cup B \;\textrm{injection}}\beta^{\cyc(f)}=(\beta+n)_m.
\ee

\begin{thm} \label{Thm6.2}The polynomials $L(\n;\alpha, \La)$ have the following integral representation
\begin{align}\label{foa-zei-lam}
L(\n; \alpha, \La)=
(-1)^{n_1+\cdots +n_m}\frac{n_1!\cdots n_m!}{\Gamma(\alpha+1)}  \int_0^\infty x^\alpha e^{-x} \prod_{j=1}^m L_{n_j}^{(\alpha)}(\la_j x)dx.
\end{align}
Moreover, this formula is equivalent to the special $\La={\bf 1}$ case.
\end{thm}
\begin{rem} Let us first explain what we mean by the equivalence of  \eqref{foa-zei-lam} and its special $\La=\bf 1$ case.
We first prove that the definition~\eqref{eq:fzgeneral} implies that $L(\n;\alpha, \La)$ is given by the integral   \eqref{foa-zei-lam}.
By taking $\La={\bf 1}$ in  \eqref{foa-zei-lam}, the formula  reduces to the special $\La=\bf 1$ case.  The point is that we shall prove that knowing the equality
in  \eqref{foa-zei-lam} for  $\La=\bf 1$ proves that the two sides of  \eqref{foa-zei-lam} are equal via the use of the well-known formula
\cite[Theorem 4.6.5]{Ism2}:
\be\label{multilaguerre}
L_n^{(\alpha)}(cx)=(\alpha+1)_n \sum_{k=0}^n \frac{c^k(1-c)^{n-k}}{(n-k)! (\alpha+1)_k} L_k^{(\alpha)}(x).
\ee
The formula~\eqref{foa-zei-lam}  was first proved  by Even and Gillis~\cite{Eve:Gil}  for
 $\alpha=0$ and $\La=\bf 1$.   Foata and Zeilberger~\cite{Foa:Zei}  proved the
general case of \eqref{foa-zei-lam} by  introducing the cycles.
\end{rem}

\begin{proof}
Clearly  \eqref{eqIplusminusy} reduces to \eqref{eq:laguerrelam}
 when $\la_m =1$, and
 $$
 A_n=1, \quad B_n=-(2n+\alpha+1),\quad  C_n=n(n+\alpha)
$$
for all $n$.  That is, the orthogonal polynomials are the normalized Laguerre polynomials
$p_n(x)=(-1)^nn!L_n^{(\alpha)}(x)$, which satisfy the three-term recurrence relation
\be
xp_n(x)=p_{n+1}(x)+(2n+\alpha+1)p_n(x)+n(n+\alpha)p_{n-1}(x),
\ee
and the orthogonal  relation
$$
\int_0^\infty \frac{x^\alpha e^{-x}}{\Gamma(\al+1)}\;  p_n(x)p_m(x)dx
={n!} {(\alpha+1)_n}\delta_{m,n}.
$$
From Lemma~\ref{lemlaguerrelam} and Theorem~\ref{Uniqueness} we deduce \eqref{foa-zei-lam} when $\la_m=1$.
To recover the general $\la_m\neq 1$ case, we can proceed as follows:  let $E$ be a subset of $S_m$ with cardinality $n_m-k$,
we consider the permutations $\pi$ of $\s(\n)$ such that   $\Fix_m(\pi)=E$.
Any such a permutation corresponds to a pair $(\sigma, \tau)$ such that $\sigma$ is the restriction of $\pi$ on $E$, which is an injection from
$E$ to $S_m$, and
$\tau$ is a permutation on $S_1\cup\cdots \cup S_{m-1}\cup (S_m\setminus E)$ defined by
$\tau(x)=\pi(x)$ if $\pi(x)\notin E$ and $\tau(x)=\pi^l(x)$ where $l$ is the minimum integer such that $\pi^l(x)\notin E$.
Clearly, the correspondence $\pi\mapsto (\sigma, \tau)$ is a bijection and
 the
generating function of such permutations is
$$
(\alpha+1+k)_{n_m-k} (\la_m-1)^{n_m-k} \la_m^{k}\,L(\n_m^*;\alpha, \La^*),
$$
where $\n_m^*=(n_1, \ldots, n_{m-1}, k)$ and $\La^*=(\lambda_1, \ldots, \lambda_{m-1}, 1)$.
Applying the result for $\lambda_m=1$ case we obtain
\begin{align*}
L(\n;\alpha, \La) &=\sum_{k=0}^{n_m} {n_m\choose k} (\alpha+1+k)_{n_m-k} (\la_m-1)^{n_m-k} \la_m^{k} \;L(\n_m^*;\alpha, \La^*) \\
&=\prod_{j=1}^{m}
(-1)^{n_j} n_j! \int_\R  \frac{x^\alpha e^{-x}}{\Gamma(\al+1)}\prod_{j=1}^{m-1}\, L_{n_j}^{(\al)}
(\la_j x) \\
&\quad\times \sum_{k=0}^{n_m} \frac{(\alpha+1+k)_{n_m-k} (1-\la_m)^{n_m-k} \la_m^{k}}{k!} L_{k}^{(\alpha)}(x) \;dx.
\end{align*}
Now, invoking the known formula~\eqref{multilaguerre},
 we deduce \eqref{foa-zei-lam}.

It remains to show the special  ${\mathbf \La}=\mathbf 1$  case of \eqref{foa-zei-lam} implies   \eqref{foa-zei-lam} for general ${\mathbf \La}$.
As in the above argument,
instead of operating within the last \emph{box},  applying  the same operation to all the boxes and  using   \eqref{foa-zei-lam}  for $\La=\bf 1$
we obtain
\begin{multline*}
L(\n;\alpha, \La) =\sum_{k_1, \ldots, k_m\geq 0}\prod_{j=1}^m{n_j\choose k_j} (\alpha+1+k_j)_{n_j-k_j} (\la_j-1)^{n_j-k_k} \la_j^{k_j} L({\mathbf  k};\alpha, {\bf 1}) \\
= \int_\R  \frac{x^\alpha e^{-x}}{\Gamma(\al+1)}\Biggl(\prod_{j=1}^{m}(-1)^{n_j} n_j! \sum_{k=0}^{n_j} \frac{(\alpha+1+k)_{n_j-k} (1-\la_j)^{n_j-k} 
\la_j^{k}}{(n_j-k)!} L_{k}^{(\alpha)}(x)\Biggr) \;dx.
\end{multline*}
Thus,  the general formula \eqref{foa-zei-lam} follows by applying the multiplication formula \eqref{multilaguerre}.
\end{proof}

\begin{rem} The analogue of \eqref{foa-zei-lam} for Hermite polynomials \cite[Proposition 5.1]{Kim:Zen2} reads
\begin{align}\label{kim-zen-hermite}
2^{-(n_1+\cdots +n_m)/2}
\int_\R \frac{e^{-x^2}}{\sqrt{\pi}} \prod_{j=1}^m
H_{n_j}(\la_j x) dx
=\sum_{\pi\in {\mathcal M}(\n)}
\prod_{i=1}^m (\la_i^2-1)^{\hom_i(\pi) } \la_i^{|S_i|-2\hom_i (\pi)} ,
\end{align}
where $\hom_i(\pi)$ denotes the number of homogeneous edges in $S_i$ for $1\leq i\leq m$.
When $\la_i=1$, the  right-hand side of  \eqref{kim-zen-hermite}  reduces obviously  to the number  of inhomogeneous matchings of $[\n]$, so
the formula \eqref{kim-zen-hermite}  becomes \eqref{eqAzGiVi}.
As
the analogue of \eqref{multilaguerre}  for Hermite polynomials \cite[(4.6.33)]{Ism2}  is
\be\label{multihermite}
H_n(cx)=\sum_{k=0}^{\lfloor n/2\rfloor} \frac{n!(-1)^k}{k!(n-2k)!} (1-c^2)^k c^{n-2k} H_{n-2k}(x),
\ee
a similar proof of \eqref{kim-zen-hermite} from  \eqref{eqAzGiVi} using \eqref{multihermite}  can be given. We leave this to the interested reader.
\end{rem}

\section{Linearization coefficients of Meixner polynomials}

The Meixner polynomials are \cite{Ism2, Koe:Swa} \be M_n(x;\bet, c)
=(\bet)_n {}_2F_1(-n, -x; \bet; 1-1/c), \ee 
and satisfy the
orthogonality relation
 \be\label{OrthMeixner} \sum_{x=0}^\infty M_m(x;\bet,
c)M_n(x;\bet, c) \frac{(\bet )_x}{x!} \, c^x = \frac{(\bet)_n
\,n!}{c^n \; (1-c)^\bet} \; \delta_{m,n}, \quad \bet >0, \quad 0 < c < 1.
\ee The Meixner polynomials generalize the Laguerre polynomials in
the sense
$$
\lim_{c \to 1} M_n(x/(1-c); \al+1, c) = n! L_n^\al(x).
$$
They have the generating function
\be
\Sum  M_n(x; \bet, c) \frac{t^n}{n!} = (1-t/c)^x (1-t)^{-x-\bet}.
\label{eqgfM}
\ee
The notation here is slightly different from \cite[Chapter 6]{Ism2}.  The three-term recurrence relation is
\begin{multline}\label{eqMeixner3trr}
-xM_n(x;\beta,c)=c(1-c)^{-1}M_{n+1}(x;\beta,c) \\
-[c(\beta+n)+n] (1-c)^{-1}M_n(x; \beta,c)+(1-c)^{-1}(\beta+n-1)nM_{n-1}(x;\beta,c).
\end{multline}
The moments are, see  \cite{Ksa:Zen,Vie,Zen},
\be
\mu_n(\beta,c)=(1-c)^\beta\sum_{k\geq 0} k^n c^k \frac{(\beta)_k}{k!}
=\frac{\sum_{\pi\in \s_n} c^{\wex(\pi)} \beta^{\cyc (\pi)}}{(1-c)^n},
\ee
where  $\wex(\pi)$ is  the number of \emph{weak excedances} of $\pi$, i.e.,
\begin{align}
{\wex}(\pi)=|\{i|\;1\leq i\leq n\;\textrm{and}\; i\leq \pi(i)\}|.
\end{align}

Let $\pi$ be a permutation of $[\n]$. We say that $\pi$ has an excedance (resp.   \emph{box-excedance})
 at $i\in [n]$ if $i<\pi(i)$ (resp.
$i\in S_k$, $\pi(i)\in S_j$ and $j>k$).  Denote by $\exc (\pi)$ (resp. $\exc_b (\pi)$)  the number of excedances (resp. box-excedances) of $\pi$.
Clearly, if $\pi$ is an inhomogeneous derangement, then $\exc(\pi)=\exc_b(\pi)$.
Consider the generating function of the derangements with respect to the numbers of cycles and  (box-)excedances:
\begin{align}
M(\n;\beta,c)=\sum_{\pi\in {\mathcal D}(\n)}\beta^{\cyc(\pi)} c^{\exc (\pi)}.
\end{align}
\begin{lem}\label{lemmeixner} For any $k,j\in [m]$ such that $k\neq j$ we have
\begin{multline}\label{eq:meixner}
M_j^{+}(\n;\beta, c)-M_k^{+}(\n;\beta, c)=(c+1)(n_k-n_j)M(\n;\beta, c)\\
+cn_k(n_k+\beta-1)M_k^{-}(\n;\beta,c)
-cn_j(n_j+\beta-1)M_j^{-}(\n;\beta,c),
\end{multline}
and the boundary condition \eqref{eqsys} with $\la_j=1$ for all $j$ and $A_0C_1=A_1$.
\end{lem}

\begin{proof} 
Let $j\in[m]$ and set $\n^*=(n_1,\ldots,n_j,1,n_{j+1},\ldots, n_m)$.  We 
first show that 
\begin{align}\label{eqMeixPlus}
M(\n^*;\beta, c)=M_j^{+}(\n;\beta, c) + n_j (c+1)M(\n;\beta, c) + n_j(n_j+\beta-1) M_j^{-}(\n;\beta, c).
\end{align}
Let $u=n_1+\cdots +n_j+1$ and for each $\pi\in D(\n^*)$, let   $a:=\pi(u)$ and $b:=\pi^{-1}(u)$.
We partition  the derangements in $D(\n^*)$ into  five  categories:
\begin{itemize}
\item[(1)] $a\notin S_j$ and $b\notin S_j$.  These derangements can be easily identified with  the derangements in $D_j^{+}(\n)$, so 
the corresponding enumerative polynomial is  $M_j^{+}(\n;\beta, c)$.
\item[(2)]   $a\notin S_j$ and $b\in S_j$.
Define the  derangement $\pi'$ on $[\n^*]\setminus \{u\}$ by
$\pi'(j)=\pi(j)$ for $j\neq b$ and $\pi'(b)=a$. Clearly $\cyc (\pi')=\cyc(\pi)$ and $\exc(\pi')=\exc(\pi)-1$.
Conversely, starting with any derangement $\pi'$ of $[\n^*]\setminus \{u\}$, 
we can recover a derangement $\pi\in D(\n^*)$ by
choosing  any  element  in  $S_j$ as  $b$ and breaking  the arrow $b\to \pi'(b)$ into  $b\to u$ and $u\to \pi'(b)$,
so the corresponding enumerative polynomial is 
$c n_j M(\n;\beta, c)$.
\item[(3)]  $a\in S_j$ and $b\notin S_j$.
Define the  derangement $\pi'$ on $[\n^*]\setminus \{u\}$ by
$\pi'(j)=\pi(j)$ for $j\neq b$ and $\pi'(b)=a$. Clearly $\cyc (\pi')=\cyc(\pi)$ and $\exc(\pi')=\exc(\pi)$.
As in the case (2),  the corresponding enumerative polynomial is 
$n_j M(\n;\beta, c)$.
\item[(4)]    $a=b$ and $a\in S_j$. The corresponding enumerative polynomial  is $c\beta n_jM_j^{-}(\n;\beta,c)$.
\item[(5)]   $a\in S_j$,  $b\in S_j$, and $a\neq b$.  Define the  derangement $\pi'$ on $[\n^*]\setminus \{a,u\}$ by
$\pi'(j)=\pi(j)$ for $j\neq  b$ and $\pi'(b)=\pi(a)$. Clearly $\cyc (\pi')=\cyc(\pi)$ and $\exc(\pi')=\exc(\pi)-1$.
Conversely, starting with a derangement on $[\n^*]\setminus \{u,a\}$,  we can reverse this process by choosing any  element  
in  $S_j\setminus \{a\}$ as $b$. As there are $n_j(n_j-1)$ ways to choose
two different elements $a$ and $b$ in $S_j$,
the corresponding enumerative polynomial  is $c n_j(n_j-1)M_j^{-}(\n;\beta,c)$.
\end{itemize}
Summarizing the above five  cases leads to~\eqref{eqMeixPlus}. Specializing~\eqref{eqMeixPlus}
at $j=k$ and then subtracting the resulted equation from~\eqref{eqMeixPlus} ends the proof.
\end{proof}

\begin{thm}\label{thlinmeix} We have
\be\label{eqlinmeixner}
M(\n;\beta, c^{-1})=  (-1)^{n_1+\cdots +n_m}  (1-c)^\beta  \sum_{x=0}^\infty \prod_{j=1}^m M_{n_j}(x;\beta,c) \frac{c^x(\beta)_x}{x!}.
\ee
\end{thm}
\begin{proof}
When $\la_j =1$ for all $j$,  and
$$
A_n=1-1/c, \quad B_n=\beta+n+n/c, \quad C_n=n(\beta+n-1)/c\quad \textrm{for all}\quad n\geq 0,
$$
by Lemma~\ref{lemmeixner}, the polynomials $ (-1)^{n_1+\cdots +n_m}M(\n;\beta, c^{-1})$ satisfy 
\eqref{eqIplusminusy} and \eqref{eqsys}.  Theorem~\ref{Uniqueness} implies then  \eqref{eqlinmeixner}.
\end{proof}
This formula was  first given by Askey and Ismail~\cite{Ask:Ism} when $\beta=1$,  and by Zeng~\cite{Zen90} for general~$\beta$.

\section{Linearization coefficients of Meixner--Pollaczek polynomials}
The Meixner--Pollaczek polynomials $P_n(x):=P_n(x; \delta, \eta)$ can be defined by\cite{Chi,Koe:Swa},
\be
\sum_{n=0}^\infty P_n(x; \delta, \eta)\frac{t^n}{n!}=[(1+\delta t)^2+t^2]^{-\eta/2} \exp\Bigl[x\arctan\Bigl(\frac{t}{1+\delta t}\Bigr)\Bigr].
\ee
They satisfy the recurrence relation:
\be
P_{n+1}(x; \delta, \eta)=(x-(\delta +2n)\eta)P_n(x; \delta, \eta)-n(\eta+n-1)(1+\delta^2)P_{n-1}(x; \delta, \eta).
\ee
The orthogonality relation is
\be
\frac{1}{ \int_\R w(x)\,dx}\int_\R P_n(x)P_m(x)w(x)\,dx=(\delta^2+1)^n n! (\eta)_n \delta_{mn},
\ee
where $w(x)=x(x;\delta,\eta)$ is given by
$$
w(x;\delta,\eta)=[\Gamma(\eta/2)]^{-2}\Bigl| \Gamma\Bigl(\frac{\eta+\imath x}{2}\Bigr)\Bigr|^2 \exp(-x\arctan \delta).
$$

Recall that a    permutation $\pi$ of $[\n]$ has a \emph{drop} (resp. \emph{box-drop})  at $i\in [n]$ if
$i>\pi(i)$ (resp. $i\in S_k$, $\pi(i)\in S_j$ and $j<k$).  Denote by $\drop (\pi)$  (resp. $\drop_b (\pi)$)  the number of drops  (resp. box-drops) of $\pi$.

The moments  of Meixner--Pollaczek polynomials \cite{Zen} are
\be
\mu_n(\delta, \eta)=\frac{1}{ \int_\R w(x)\,dx}\int_\R x^nw(x)\,dx= \sum_{\sigma\in \s_n}(\delta+\imath)^{\drop (\sigma)} (\delta-\imath)^{\exc (\sigma)}  \eta^{\cyc (\sigma)},
\ee
where $\imath^2=-1$.

Consider the enumerative polynomial of the inhomogeneous derangements
\begin{align}
P(\n;\delta,\eta)=\sum_{\pi\in {\mathcal D}(\n)}(\delta+\imath)^{\drop (\pi)} (\delta-\imath)^{\exc (\pi)}  \eta^{\cyc (\pi)}.
\end{align}

\begin{lem}\label{lemMP}  For any $k,j\in [m]$ such that $k\neq j$ we have
\begin{multline}\label{eq:mp}
P_j^{+}(\n;\delta,\eta)-P_k^{+}(\n;\delta,\eta)=2\delta (n_k-n_j)P(\n;\delta,\eta)\\
+n_k(n_k+\eta-1)(\delta^2+1)P_k^{-}(\n;\delta,\eta) -n_j(n_j+\eta-1)(\delta^2+1)P_j^{-}(\n;\delta,\eta),
\end{multline}
and the boundary condition \eqref{eqsys} with $\la_j=1$ for all $j$ and $A_0C_1=A_1$.
\end{lem}
\begin{proof}   
For $j\in[m]$  let  $\n^*=(n_1,\ldots,n_j,1,n_{j+1},\ldots, n_m)$.  
Following the  proof of Lemma~\ref{lemmeixner} we obtain 
\begin{align}\label{eqMeix-Polla}
P(\n^*;\delta, \eta)=P_j^{+}(\n;\delta, \eta) + 2n_j \delta P(\n;\delta, \eta) +(\delta^{2}+1) n_j(n_j+\eta-1) P_j^{-}(\n;\delta, \eta).
\end{align}
Subtracting the last equation  from \eqref{eqMeix-Polla}  with $j=k$ yields
\eqref{eq:mp}.
\end{proof}

By the method of separation of variables we  can solve \eqref{eq:mp} and obtain the following result.
\begin{thm} We have
\be\label{eqlinMP}
P(\n;\delta,\eta)=\frac{1}{\int_{\R} w(x)dx}\int_{\R}\prod_{j=1}^mP_{n_j}(x)w(x)dx.
\ee
\end{thm}
\begin{proof}
Clearly
\eqref{eqIplusminusy} reduces to \eqref{eq:mp}  when $\la_j =1$ for all $j$,  and
$$
A_n=1, \quad B_n=-(\delta+2n)\eta, \quad C_n=n(\eta+n-1)(1+\delta^2)\quad \textrm{for all}\quad n\geq 0.
$$
From Lemma~\ref{lemMP} and Theorem~\ref{Uniqueness} we deduce \eqref{eqlinMP}.
\end{proof}
This formula was first given by Zeng~\cite{Zen}, and  later generalized  by Kim and Zeng~\cite{Kim:Zen}.


\section{Linearization coefficients  of   $q$-Hermite polynomials}
The continuous $q$-Hermite polynomials $H_n(x|\,q)$ are generated  by
\begin{align}\label{eq:recurrence-qhermite}
H_0(x|\,q) :=1, \; \; H_1(x|\,q)= 2x, \; \;   2xH_n(x|\,q)=H_{n+1}(x|\,q)+(1-q^n)H_{n-1}(x|\,q), \; \; n >0,
\end{align}
and have the  orthogonal relation
\begin{align}\label{eq:q-orth}
\int_0^\pi H_n(\cos \theta|\,q)H_m(\cos \theta|\,q) v(\cos \theta |\,q)\,d\theta=(q;q)_n \delta_{mn},
\end{align}
where
$$
v(\cos \theta|\,q)=\frac{(q;q)_\infty}{2\pi}(e^{2\imath \theta}, e^{-2\imath \theta};q)_\infty.
$$
If we rescale  the $q$-Hermite polynomials by
$$
\tilde H_n(x|\,q)=H_n\big(\frac{1}{2}ax|\,q\big)/a^n, \qquad  a=\sqrt{1-q},
$$
then \eqref{eq:recurrence-qhermite}  reads
$$
x\tilde H_n(x|\,q)=\tilde H_{n+1}(x|\,q)+[n]_q\tilde H_{n-1}(x|\,q),
$$
and
the orthogonality relation \eqref{eq:q-orth} becomes
\begin{align}\label{eq:q-orthbis}
\int_{-2/\sqrt{1-q}}^{2/\sqrt{1-q}} \tilde H_n(x |\,q)\tilde H_m(x|\,q)\tilde v(x|\,q)\,d x=n!_q \,\delta_{mn}.
\end{align}
Here $n!_q=(q;q)_n/(1-q)^n$ and
\be
\tilde v(x|\,q)=\frac{\sqrt{1-q}(q;q)_\infty}{\sqrt{1-(1-q)x^2/4}{4\pi}}\prod_{k=0}^\infty \{1+(2-x^2(1-q))q^k+q^{2k}\}.
\ee

 Given a perfect matching $M$ (or more generally, a set partition), a pair of arcs $(e_1,e_2)$ of $M$ is said to cross
if $e_1=(i,j)$, $e_2=(k,\ell)$, and $i<k<j<\ell$. The number of arc
crossings in $M$ is denoted by $\rc(M)$. For instance, if $M$ is the
matching drawn in Figure~\ref{fig:diagrams}, we have $\rc(M)=5$. Let
\begin{align}
K({\n}|\,q)= \sum_{M\in {\mathcal K}(\n)}q^{\rc (M)}.
\end{align}
For any  nonnegative integer $n$ we set
\be
[n]_q:=\frac{1-q^n}{1-q}=1+q+\cdots +q^{n-1}.
\ee
\begin{lem} \label{qhermite}
 For  $k, j\in [m]$ and $k\neq j$ the polynomials $K({\n}|\,q)$ satisfy
 \begin{align}\label{eqqmatchequation}
 K_j^+({\n}|\,q)- K_k^+({\n}|\,q)=[n_k]_q K_k^-({\n}|\,q)-[n_j]_q K_j^-({\n}|\,q),
 \end{align}
  and the boundary condition \eqref{eqsys} with $\la_j=1$ for all $j$ and $A_0C_1=A_1$.
 \end{lem}
 \begin{proof}
Let $u=n_1+\cdots +n_j+1$. The matchings in ${\mathcal K}_j^+(\n)$ (resp. ${\mathcal K}_{j+1}^+(\n)$) can be divided into two categories:
\begin{itemize}
\item the integer $u\in S_j$  (resp, $u\in S_{j+1}$)  is matched with the $\ell$th element  $u+\ell$ in $S_{j+1}$ (resp. $u-\ell$ in $S_j$),
from left (resp., right), with $\ell\in [n_{j+1}]$ (resp. $\ell\in
[n_j]$), then the corresponding arc crosses each of the  $\ell-1$
arcs of which one vertex is $u+t$ (resp. $u-t$) with $1\leq t\leq
\ell-1$. An illustration is given in Figure~\ref{fig:match}(a) (resp., Figure~\ref{fig:match}(b)). 
Hence the generating function of such matchings is
$$
(1+q+\cdots +q^{n_{j+1}-1}) K_{j+1}^-({\n}|\,q) \quad
(\textrm{resp.}\quad  (1+q+\cdots +q^{n_j-1})K_j^-({\n}|\,q));
$$
\begin{figure}[h!]
{\setlength{\unitlength}{1mm}
\begin{picture}(45,22)(5,-10)
\put(5,0){\line(1,0){15}}\put(25,0){\line(1,0){20}}
\put(20,0){\circle*{1,3}}\put(20,0){\makebox(-2,-4)[c]{\tiny $u$}}
\put(25,0){\circle*{1,3}}\put(25,0){\makebox(1,-4)[c]{\tiny $u+1$}}
\put(40,0){\circle*{1,3}}\put(40,0){\makebox(3,-4)[c]{\tiny $u+\ell$}}
\qbezier(20,0)(30,10)(40,0)
\put(25,0){\line(0,2){10}}\put(28,0){\line(0,2){10}}\put(37,0){\line(0,2){10}}
\put(32.5,8){\makebox(0,0)[c]{......}}
\put(24,10){$\overbrace{\hspace{1.4cm}}$}
\put(24,12){\makebox(13,4)[c]{\footnotesize $\ell-1$ crossings}}
\put(-2,-5){\makebox(57,-6)[c]{\footnotesize (a) the blocks $S_j$ and $S_{j+1}$ in $\K_j^+(\n)$}}
\end{picture}}
\hspace{2.5cm}
{\setlength{\unitlength}{1mm}
\begin{picture}(45,22)(0,-10)
\put(0,0){\line(1,0){20}}\put(25,0){\line(1,0){15}}
\put(20,0){\circle*{1,3}}\put(20,0){\makebox(-2,-4)[c]{\tiny $u-1$}}
\put(25,0){\circle*{1,3}}\put(25,0){\makebox(2,-4)[c]{\tiny $u$}}
\put(5,0){\circle*{1,3}}\put(5,0){\makebox(-3,-4)[c]{\tiny $u-\ell$}}
\qbezier(5,0)(15,10)(25,0)
\put(20,0){\line(0,2){10}}\put(17,0){\line(0,2){10}}\put(8,0){\line(0,2){10}}
\put(12.5,8){\makebox(0,0)[c]{......}}
\put(7,10){$\overbrace{\hspace{1.4cm}}$}
\put(7,12){\makebox(13,4)[c]{\footnotesize $\ell-1$ crossings}}
\put(-7,-5){\makebox(57,-6)[c]{\footnotesize (b) the blocks $S_j$ and $S_{j+1}$ in $\K_{j+1}^+(\n)$}}
\end{picture}}
\caption{Crossings in an inhomogeneous perfect matching}\label{fig:match}
\end{figure}

\item the integer $u$ is matched with an element  not in $S_j\cup S_{j+1}$,  
let $R_{u}(\n|\,q)$  be the generating polynomial of  such matchings.
\end{itemize}
It follows that $K_j^+({\n}|\,q)=[n_{j+1}]_q K_{j+1}^-({\n}|\,q)+R_{u}({\n}|\,q)$ 
and $K_{j+1}^+({\n}|\,q)=[n_j]_q K_j^-({\n}|\,q)+R_{u}({\n}|\,q)$.
By subtraction we obtain  \eqref{eqqmatchequation} for adjacent $k$ and $j$. 
The general case follows from the simple identity
$u_k-u_j=\sum_{i=j}^{k-1}(u_{i+1}-u_i)$ for any integers $j$ and $k$ such that $j<k$.
 \end{proof}
 
 \begin{thm} \label{thmqhermite} We have
 \begin{align}\label{eq:isv}
K({\n}|\,q)= \int_\R \tilde v(x|\,q) \prod_{j=1}^m\tilde H_{n_j}(x|\,q)\,dx.
\end{align}
\end{thm}
 \begin{proof}
Clearly
\eqref{eqIplusminusy} reduces to \eqref{eqqmatchequation}  when $\la_j =1$ for all $j$, $B_k=0$, $C_k=[k]_q$  for all $k$, and
$A_k$ is a constant independent of $k$.  From Lemma~\ref{qhermite} and Theorem~\ref{Uniqueness} we deduce \eqref{eq:isv}.
\end{proof}

\begin{rem}
The representation \eqref{eq:isv} is due to Ismail, Stanton and Viennot~\cite{Ism:Sta:Vie}.
Three  different proofs were later given in \cite{Ans, Eff:Pop,Kim:Sta:Zen}. As we can see, the new proof of \eqref{eq:isv}
given above   parallels  our proof in the case  $q=1$.
\end{rem}

Note that $K({\n}|0)$ is the number of perfect inhomogeneous matchings of $[\n]$ without crossings and
$H_n(x|0)$ is  the $n$-th Chebyshev polynomial of the second kind $U_n(x)$.  Hence,
letting  $q=0$ in Theorem~\ref{thmqhermite} we obtain the following result, due to de Sainte--Catherine and  Viennot{~}\cite{Des:Vie}.
\begin{cor} The number of perfect inhomogeneous matchings of $[\n]$  without crossings is given by
 \begin{align} 
K({\n}|0)=\frac{2}{\pi} \int_{-1}^1  U_{n_1}(x)\cdots U_{n_m}(x) (1-x^2)^{1/2}dx.
\end{align}
\end{cor}
Another generalization of the above corollary was given by Kim and Zeng \cite{Kim:Zen2}.


 \section{Linearization coefficients of $q$-Charlier and $q$-Laguerre polynomials}
\subsection{Al--Salam-Chihara polynomials} Since our $q$-Charlier and $q$-Laguerre polynomials are two rescaled
special Al--Salam--Chihara polynomials, we first recall
the definition of these polynomials.
The Al--Salam--Chihara polynomials $Q_n(x):=Q_n(x; t_1, t_2|\,q)$ may
be defined by the recurrence relation~\cite[Chapter 3]{Koe:Swa}:
\begin{align}\label{AC-def}
\begin{cases}
Q_0(x)=1,\quad Q_{-1}(x)=0,\\
Q_{n+1}(x)=(2x-(t_1+t_2)q^n)Q_n(x)-
(1-q^n)(1-t_1t_2{q}^{n-1})Q_{n-1}(x),\quad n\geq 0.
\end{cases}
\end{align}
Let $Q_n(x)=2^np_n(x)$ then
\begin{align}\label{eq:normalizedrecurr}
xp_n(x)=p_{n+1}(x)+\frac{1}{2}
(t_1+t_2)q^np_n(x)+\frac{1}{4}(1-q^n)(1-t_1t_2
q^{n-1})p_{n-1}(x).
\end{align}
They also have the following explicit expressions:
\begin{align*}
Q_{n}(x; t_1,t_2|\,q) &=\frac{(t_1t_2;\,q)_{n}}{t_1^n}\,
{}_3\phi_{2}\left( {{q^{-n},t_1 u,t_1
u^{-1}}\atop{t_1t_2,0}} \Big| \,q;\,q
\right)\\
&=(t_1 u;\,q)_{n}u^{-n}\, {}_2\phi_{1}\left( {{q^{-n},t_2
u^{-1}}\atop{t_1^{-1}q^{-n+1}u^{-1}}} \Big| \,q;t_1^{-1}q u
\right)\\
&=(t_2 u^{-1};\,q)_{n}u^{n}\, {}_2\phi_{1}\left( {{q^{-n},t_1
u}\atop{t_2^{-1}q^{-n+1}u}} \Big| \,q;t_2^{-1}q u^{-1} \right),
\end{align*}
where $x=\frac{u+u^{-1}}2$ or $x=\cos \theta$ if $u=e^{\imath \theta}$.

The Al--Salam--Chihara polynomials have the following generating
function
$$
G(t,x)=\sum_{n=0}^\infty Q_n(x;t_1,t_2|\,q)\frac{t^n}{(q;\,q)_n}=
\frac{(t_1 t, t_2 t;\,q)_\infty}{(te^{\imath \theta},
te^{-\imath \theta};\,q)_\infty}.
$$
They are orthogonal with respect to the  linear functional
$\mathcal L$:
\begin{align}
{\mathcal L}(x^n)=\frac{1}{2\pi}\int_{0}^\pi\cos^n\theta
\frac{(q, t_1t_2,e^{2\imath \theta}, e^{-2\imath \theta};\,q)_\infty}
{(t_1 e^{\imath \theta}, t_1 e^{-\imath \theta}, t_2 e^{\imath \theta}, t_2
e^{-\imath \theta};\,q)_\infty} d\theta,
\end{align}
where $x=\cos \theta$. Equivalently, the Al--Salam--Chihara polynomials $Q_n(x; t_1, t_2|\,q)$ are orthogonal on $[-1,\;1]$ with respect to the probability measure
\be
\frac{(q, t_1t_2;q)_\infty}{2\pi}\; \prod_{k=0}^\infty \frac{1-2(2x^2 -1) q^k + q^{2k}}{[1-2x t_1q^k+ t_1^2q^{2k}][1-2x t_2q^k+ t_2^2q^{2k}]} \frac{dx}{\sqrt{1-x^2}}.
\ee


As in \cite{Ans, Kim:Sta:Zen}, we shall consider  the $q$-Charlier polynomials $C_n(x|\,q):=C_n(x, a,b,c|\,q)$ defined recursively by
\begin{align}\label{eq:recurrqCharlier}
C_{n+1}(x|\,q)=\left(x-c-b[n]_q\right) C_{n}(x|\,q) - a[n]_q C_{n-1}(x|\,q),
\end{align}
where $C_{-1}(x|\,q)=0$ and $C_0(x|\,q)=1$.
Comparing with \eqref{AC-def}  we see that this is a rescaled version of the Al--Salam--Chihara polynomials:
\be
C_n(x|\,q)=\bigg(\frac{a}{1-q}\bigg)^{n/2}Q_n
\Bigl(\frac{1}{2}\sqrt{\frac{1-q}{a}}\bigl(x-c-\frac{b}{1-q}\bigr);\frac{-b}{\sqrt{a(1-q)}}, 0\left | \right. q\Bigr).
\ee

We define $u_1(x)$ and $v_1(x)$ by
\begin{equation}
\begin{split}
u_1(x) &= \frac{1-q}{2a}\,  x^2  - \frac{c(1-q)+b}{a} \, x + \frac{b^2+c^2(1-q)^2+2(1-q)(bc-a)}{2a(1-q)},  \\
v_1(x) &= \frac{1}{2} \sqrt{\frac{1-q}{a}}\;  \Big(x -c -\frac{b}{1-q}\Big).
\end{split}
\end{equation}
The moment functional for $C_n(x|\,q)$ is
\begin{multline}
{\mathcal L}_1(f) = \frac{(q;q)_\infty}{2\pi }\frac{1}{2}\sqrt{\frac{1-q}{a}} \\
\times \int_{A_{-}}^{A_{+}}
\prod_{k=0}^\infty \frac{[1-2u_1(x) q^k + q^{2k}]f(x)}{1+2v_1(x)q^k/(\sqrt{a(1-q)})+ q^{2k}/a(1-q)} \frac{dx }
{\sqrt{1-v_1(x)^2}},
\end{multline}
where
\begin{align*}
A _{\pm}= c + \frac{b}{1-q}\pm {2}\sqrt{\frac{a}{1-q}}.
\end{align*}


As  in \cite{Kas:Sta:Zen},  we shall consider    the  $q$-Laguerre polynomials $L_n(x|\,q):=L_n(x,y|\,q)$ defined by
  the recurrence:
\begin{align}\label{eq:recurrqlaguerre}
L_{n+1}(x|\,q)=(x-y[n+1]_q-[n]_q)L_n(x|\,q)-y[n]_q^2L_{n-1}(x|\,q),
\end{align}
with the initial condition $L_{-1}(x|\,q)=0$ and $L_{0}(x|\,q)=1$. Hence these are the
re-scaled  Al--Salam--Chihara polynomials:
\begin{equation}\label{eq:polydef}
L_n(x|\,q)=\left(\frac{\sqrt{y}}{q-1}\right)^n
Q_n\left(\frac{(q-1)x+y+1}{2\sqrt{y}}; \frac{1}{\sqrt{y}},
\sqrt{y}q|\,q\right).
\end{equation}
One  deduces then the explicit formula:
\begin{equation}\label{eq:explaguerre}
L_n(x|\,q)=\sum_{k=0}^n (-1)^{n-k}\frac{n!_q}{k!_q}\,{n\brack k}_q
q^{k(k-n)}y^{n-k} \prod_{j=0}^{k-1}\left(x-(1-y q^{-j})[j]_q\right).
\end{equation}

Define $u_2(x)$ and $v_2(x)$ by
\be
u_2(x) = \frac{(1-q)^2}{2y}\,  x^2 -\frac{(1-q) (1+y)}{y} \, x  + \frac{y^2+1}{2y}, \quad v_2(x) = \frac{q-1}{2\sqrt{y}} x +
\frac{y+1}{2\sqrt{y}}.
\ee
Then the moment functional in this case is
\begin{multline}
{\mathcal L}_2(f) = \frac{(q, q;q)_\infty }{2\pi}\frac{1-q}{2\sqrt{y}}\;   \\
\times \int_{B_{-}}^{B_{+}}
\prod_{k=0}^\infty \frac{[1-2u_2(x) q^k + q^{2k}]f(x)}{[1-2v_2(x)q^k/\sqrt{y}+ q^{2k}/y]
[1-2v_2(x)  q^{k+1}\sqrt{y} + q^{2k+2}y]} \frac{dx }{\sqrt{1-v_2(x)^2}},
\end{multline}
where
\be
B_{\pm} = \frac{(1\pm \sqrt{y})^2}{1-q}.
\ee

 For the combinatorial approach to the linearization  coefficients,  the $q$-Hermite and $q$-Charlier cases were proved  by first combining  the combinatorial models for the polynomials and
moments to obtain a messy sum, and then using  a killing involution to reduce it to some nicer models, \cite{Des:Vie, Ism:Sta:Vie, Kim:Sta:Zen}.
 However,  this approach seems  difficult  to deal with the $q$-Laguerre case.
So, a  recursive approach based on the symmetry is used in \cite{Kas:Sta:Zen}, but such  a proof for the $q$-Charlier polynomials is new.
\subsection{Linearization coefficients of  $q$-Charlier polynomials}
Recall that if $\pi$ is a partition of $[n]$,
an arc crossing of $\pi$ is a pair of arcs $(e_1,e_2)$ such that $e_1=(i,j)$, $e_2=(k,\ell)$, and $i<k<j<\ell$.
For instance, if $\pi$ is the partition drawn in Figure~\ref{fig:diagrams} (resp., in Figure~\ref{fig:diagrams-inhomogeneous}),
then $\rc(\pi)=2$ (resp., $\rc(\pi)=6$). We let $\rc(\pi)$ denote the number of arc crossings in $\pi$.

For each partition ${\pi\in \Pi_n} $ we define the weight
\be\label{charlier-weight}
w(\pi)=a^{\bl (\pi) } b^{\tr (\pi) } c^{\sg (\pi)} q^{\rc (\pi)},
\ee
where $\bl(\pi)$, $\sg(\pi)$ and $\tr(\pi)$ are respectively the numbers of blocks, singletons and
transients of $\pi$. Here, a singleton is just a block of size 1 and a transient is an element which is neither
the least nor the greatest element in a block of $\pi$.

Consider the enumerative polynomial of inhomogeneous partitions
\begin{align}\label{eq:defF}
\F(\n|\,q):=\F(\n;a,b,c|\,q)=\sum_{\pi\in \P(\n)}w(\pi).
\end{align}

Note that by the general theory of orthogonal polynomials,  the three-term recurrence relation \eqref{eq:recurrqCharlier}  and Proposition~4.1 in \cite{Kas:Zen2}
imply that
the linear functional ${\mathcal L}_1$  has the following combinatorial interpretation:
\begin{align}\label{defmoment}
{{\mathcal L}}_1(x^n)=\sum_{\pi\in \Pi_n} w(\pi).
\end{align}
To find  the partial difference equations satisfied by $\F(\n|\,q)$ we need the following key result.
\begin{lem}\label{lem:sym-charlier} The polynomials $\F(\n|\,q)$ are symmetric
with respect  to the permutation of indices $n_1, \ldots, n_m$.
\end{lem}

We postpone  the proof of this  crucial lemma to Section~\ref{bijection11}.

\begin{lem} \label{lemqCharlier} For  $j\in [m]$, the polynomials $\F(\n|\,q)$ satisfy
\begin{align}\label{eqIplusminus-char}
\F_j^{+}({\n}|\,q) = \F(\n^*|\,q)-b[n_j]_q  \F({\n}|\,q)
-a[n_j]_q \F_j^-({\n}|\,q),
\end{align}
where $\n^*=(1, n_1, \ldots, n_m)$.
\end{lem}
 \begin{proof} By
Lemma~\ref{lem:sym-charlier},  we can suppose that $j=1$. Hence,
 it   suffices to check that
\begin{align}\label{keyrecurrence}
\sum_{\pi\in \P({\n^*})} w(\pi)
&= \F_1^{+}({\n}|\,q)+b [n_1]_q\F({\n}|\,q)+a [n_1]_q\F_1^{-}({\n}|\,q).
\end{align}
where $w(\pi)=a^{\bl (\pi)} b^{\tr(\pi)} q^{\rc(\pi)}$ since $\sg(\pi)=0$ for any $\pi\in \P(\n)$.

Given a partition $\pi\in \P(\n^{*})$, we denote by $r_1$ the integer $i>1$ which is connected to 1 by an arc. 
We classify the partitions in $\P(\n^{*})$  into three categories according to the value of $r_1$ (The reader is suggested to
draw diagrams as we do in the proof of Lemma~\ref{qhermite}):
\begin{itemize}
\item[(a)] $r_1>n_1+1$; such partitions are exactly the partitions
in $\P_1^+(\n)$, whence the enumerative
polynomial of such partitions is $\F_1^{+}({\n}|\,q)$.
\item[(b)] $2\leq r_1\leq n_1+1$;  then
the arc $(1,r_1)$ crosses with each of the $r_1-2$ arcs of which one vertex is $\ell$ with $2\leq \ell \leq r_1-1$.
Suppose $\{1,r_1\}$ is a block of $\pi$ (resp., is not a block of $\pi$).
Summing over all $r_1=2,3,\ldots,n_1+1$, it is readily seen that the
enumerative polynomial of such partitions is $
\sum_{r_1=2}^{n_1+1} a q^{r_1-2}\F_1^-(\n|\,q)=a [n_1]_q \F_1^{-}({\n}|\,q)
$ (resp., $
\sum_{r_1=2}^{n_1+1} b q^{r_1-2}\F_1(\n|\,q)=b [n_1]_q \F({\n}|\,q)
$).
\end{itemize}
Summing up the above three cases we obtain \eqref{keyrecurrence}.
 \end{proof}

 The following result is due to Anshelevich~\cite{Ans}
 and a combinatorial proof was later given by  Kim, Stanton and Zeng~\cite{Kim:Sta:Zen}.
\begin{thm}\label{thm:qlinear}  For $m\geq 1$ and  $ n_{1},\ldots, n_{m}\geq 0$,  we have
\begin{align}\label{eq:qlin}
\F(\n|\,q)={{\mathcal L}}_1\left(C_{n_1}(x|\,q)\cdots C_{n_m}(x|\,q)\right).
\end{align}
\end{thm}
\begin{proof}

 For $j,k\in [m]$ and $j\neq k$ we deduce from   \eqref{eqIplusminus-char} that
\begin{align}\label{eqIplusminus-char-bis}
\F_j^{+}({\n}| \,q) -   \F_k^{+}({\n}|\,q)
= \left([n_k]_q-[n_j]_q   \right) \F({\n}|\,q)
-a[n_j]_q \F_j^-({\n}|\,q) +  a[n_k]_q\F_k^{-}({\n}|\,q),
\end{align}
and the boundary condition \eqref{eqsys} with $\la_j=1$ for all $j$ and $A_0C_1=A_1$. 
The result then follows by applying Theorem~\ref{Uniqueness}.
\end{proof}

\begin{rem}
When $q=0$,   the polynomials $C_n(x|\,0)$ are the so-called perturbed Chebyshev polynomials of the second kind and
$\F(\n|\,0)$ is the enumerative polynomial of inhomogeneous partitions of $[\n]$ without any arc crossings.
\end{rem}

\begin{rem}
In view of Lemmas~\ref{lem:sym-charlier}, \ref{lemqCharlier} and Theorem~\ref{thmbis},
we can also prove the above theorem by checking \eqref{eq:qlin}  for the special
$\n=1^m:=(1, \ldots, 1)$.  As $C_1(x;q)=x-c$,
the latter identity reads
\be\label{simpleeq}
\F(1^m|\,q)={{\mathcal L}}_1((x-c)^m)\qquad (m\geq 1).
\ee
By the binomial formula,   this  is equivalent to
\be
 \L_1(x^m)
                =\sum_{k=0}^m{m\choose k}c^k \F(1^{m-k}|\,q).
\ee
In view of the combinatorial interpretation of the moments \eqref{defmoment} and  the definition  \eqref{eq:defF} 
the latter identity is obvious if we enumerate the partitions $\pi$ of $[m]$
by the weight \eqref{charlier-weight}  and  according to the number of singletons.
\end{rem}

 \subsection{Linearization coefficients of $q$-Laguerre polynomials}
For $\sigma\in \s_n$  the \emph{number of crossings} of $\sigma$ is
defined by
\be
\rc(\sigma)=\sum_{i=1}^n\#\{j|j<i\leq \sigma(j)<\sigma(i)\}+
\sum_{i=1}^n\#\{j|j>i>\sigma(j)>\sigma(i)\}.
\ee
Note that  the linear functional ${{\mathcal L}}_2$ has the following combinatorial interpretation~\cite{Kas:Sta:Zen}:
\begin{align}\label{qlagmoments}
{{\mathcal L}}_2(x^n)=\sum_{\sigma\in \s_n}y^{\exc(\sigma)}q^{\rc(\sigma)}.
\end{align}

Consider the enumerative polynomial of inhomogeneous derangements
\begin{align}
I(\n|\,q):=I(\n;y|\,q)=\sum_{\sigma\in \D(\n)}y^{\exc(\sigma)}q^{\rc (\sigma)}.
\end{align}
 \begin{lem} \label{lemqLaguerre}
 The polynomials $I(\n;y|\,q)$ satisfy
\begin{multline}\label{eqIplusminus-lag}
I_j^{+}(\n|\,q) -   I_k^{+}(\n|\,q)= (yq+1) \left([n_k]_q-[n_j]_q\right) I(\n|\,q)\\
-y[n_j]_q^2 I_j^-(\n|\,q) +  y[n_k]_q^2I_k^{-}(\n|\,q),
\end{multline}
and the boundary condition \eqref{eqsys} with $\la_j=1$ for all $j$ and $A_0C_1=A_1$.
\end{lem}
 \begin{proof} It is proved in \cite[eq. (38)]{Kas:Sta:Zen} that
 \begin{align}
 I_j^{+}(\n|\,q)=I({\n}^*|\,q)-(yq+1)[n_j]_qI(\n|\,q)-y[n_j]_q^2I_j^-(\n|\,q),
\end{align}
where $\n^*=(1, n_1, \ldots, n_m)$. Replacing $j$ by $k$  in the above equation and then subtracting the resulting equation from the above one
we get \eqref{eqIplusminus-lag}. The boundary condition is obvious.
 \end{proof}
The following  result is due to  Kasraoui, Stanton and Zeng~\cite{Kas:Sta:Zen}.
\begin{thm}\label{eq:qlinear} We have
\begin{align}\label{eq:qlin2}
I(\n|\,q)={{\mathcal L}}_2(L_{n_1}(x|\,q)\cdots L_{n_m}(x|\,q)).
\end{align}
\end{thm}
\begin{proof}
Clearly
\eqref{eqIplusminusy} reduces to \eqref{eqIplusminus-lag}  when $\la_j  =1$ for all $j$,  and
$$
A_n=1, \quad B_n=-(y[n+1]_q+[n]_q), \quad C_n=q[n]_q^2,  \qquad  n\geq 0.
$$
  From Lemma~\ref{lemqLaguerre} and Theorem~\ref{Uniqueness} we deduce \eqref{eq:qlin2}.
\end{proof}

\begin{rem}
In the above proof, we do not require the combinatorial interpretation of the moments \eqref{qlagmoments}, which was needed in
\cite{Kas:Sta:Zen}.
\end{rem}

\section{More  integrals of orthogonal polynomials}

In this section,  for  a sequence of orthogonal polynomials $\{p_n(x)\}$, we shall consider  integrals of type
\be\label{type1}
 \int_\R x^{n_0}\prod_{j=1}^m p_{n_j}(x)d\mu(x), \qquad n_0\in \N,
 \ee
and
\be\label{type2}
\int_\R x(x-1)\cdots (x-{n_0}+1)\prod_{j=1}^m p_{n_j}(x)d\mu(x), \qquad n_0\in \N,
\ee
where $\mu$ is an orthogonal  measure for $\{p_n(x)\}$.

  One important tool used in this work is  MacMahon's  Master theorem, \cite[Vol.1, pp. 93--98]{Mac} and its $\beta$-extension due to
Foata--Zeilberger~\cite{Foa:Zei}, which we now recall.

Let $V_m$ be the determinant $\det(\delta_{ij}-x_ia_{i,j})$ ($1\leq i,j\leq m$).  The MacMahon master theorem asserts that the coefficient of
$x_1^{n_1}x_2^{n_2}\dotsm x_m^{n_m}$ in the expansion of $V_m^{-1}$ is equal to the coefficient of $x_1^{n_1}x_2^{n_2}\dotsm x_m^{n_m}$ in
\be\label{macproduct}
\prod_{k=1}^m \left( a_{k,1}x_1 + \cdots + a_{k,m}x_m\right)^{n_k}.
\ee

It will be convenient to restate this in a slightly different form.
Let ${\mathcal C}({\m})$ be the set of rearrangements of the word $1^{n_1}\ldots m^{n_m}$.
For any rearrangement
$$
\gamma=\gamma(1, 1)\ldots \gamma(1, n_1)\ldots \gamma(m, 1)\ldots \gamma(m, n_m)\in {\mathcal C}({\m}),
$$
 we associate the weight
$$
w(\gamma)=\prod_{i,j}a_{i, \gamma(i,j)}\quad (1\leq i\leq m, \quad 1\leq j\leq n_i).
$$
Then, the coefficient of $x_1^{n_1}x_2^{n_2}\dotsm x_m^{n_m}$ in  \eqref{macproduct} is equal to the sum of
 all the $w(\gamma)$ with $\gamma$ running over all the elements in ${\mathcal C}({\m})$.
On the other hand,  each sequence $\n=(n_1, \ldots, n_m)$ of positive integers defines a unique mapping $\chi$ from $[\n]$ to $[m]$ given by
$\chi(j)=i$ if $j\in S_i$. For each permutation $\pi\in \s(\n)$ we let
$$
w(\pi)=\prod_{j=1}^{n} a_{\chi(j), \chi(\pi(j))}.
$$
Clearly, to each rearrangement $\gamma$ in ${\mathcal C}({\m})$,  there corresponds exactly $n_1!\cdots n_m!$ permutations  $\pi$ in $\s(\n)$ with the property that
$w(\pi)=w(\gamma)$. Therefore, the coefficient of $x_1^{n_1}x_2^{n_2}\dotsm x_m^{n_m}$ in  \eqref{macproduct} is also equal to
$$
\frac{1}{n_1!\cdots n_m!} \sum_{\pi\in \s(\n)} w(\pi).
$$
The MacMahon Master theorem can now be restated as
$$
\sum_{n_{1},\ldots, n_{m}\geq 0} \frac{x_1^{n_1}\dotsm x_m^{n_m}}{n_1!\cdots n_m!} \sum_{\pi \in \s(\n)} w(\pi)= V_m^{-1}.
$$
The $\beta$-extension of the MacMahon Master theorem~\cite{Foa:Zei} reads as follows.
\begin{thm} \label{MacMahon} We have
\be
\label{foa-zeil-mac}
\sum_{n_{1},\ldots, n_{m}\geq 0} \frac{x_1^{n_1}\dotsm x_m^{n_m}}{n_1!\cdots n_m!} \sum_{\pi \in \s(\n)} \beta^{\cyc (\pi)}w(\pi)= V_m^{-\beta}.
\ee
\end{thm}

Now,  we consider  the determinant
\begin{align*}
\Delta_{m+1}&:=\left|\begin{array}{ccccc}
    1 & -cx_1& \cdots & -cx_1&-cx_1 \\
   -x_2 & 1& \cdots  & -cx_2 &-cx_2\\
    \vdots &  \vdots & \ddots&  \vdots &\vdots\\
-x_m&-x_m&\cdots&1&-cx_m\\
  -x_{0} & -x_{0} & \cdots & -x_{0}&1-x_{0} \\
  \end{array}\right|.
\end{align*}
The  proof of the following  determinant formula is left to the reader.
\begin{lem} \label{Lemma1}
Let $a$ and $b$ be any variables in a commutative ring. Then
$$
\left|\begin{array}{ccccc}
    x_1 & a& \cdots & a&a \\
   b & x_2 & \cdots  & a &a\\
    \vdots &  \vdots & \ddots&  \vdots &\vdots\\
b&b&\cdots&x_{n-1}&a\\
  b & b & \cdots & b&x_n \\
  \end{array}\right|=\frac{a\phi_n(b)-b\phi_n(a)}{a-b},
$$
where $\phi_n(x)=(x_1-x)(x_2-x)\cdots (x_n-x)$. When $a=b$, the right side should be taken as the limit
$\phi_n(a)\left(1+a\sum_{j=1}^m\frac{1}{x_j-a}\right)$.
\end{lem}
Applying the above lemma to $\Delta_{m+1}$ we obtain
\begin{align}
\Delta_{m+1}=\frac{1}{c-1}\left[ c(1+x_1)\cdots (1+x_m)-(1+cx_1)\cdots (1+cx_m)(1-(1-c)x_{0}) \right].
\end{align}
Therefore,  denoting  the elementary symmetric functions of the indeterminates $x_1, \ldots, x_m$ by $e_1({\x}),
 \ldots, e_m({\x})$, we have
\be\label{detcomputation}
\Delta_{m+1}=1-\sum_{k=2}^m (c+\cdots + c^{k-1}) e_{k}(\x)-x_{0}\prod_{j=1}^m (1+cx_j).
\ee

Let
 \be
 A^{(\al)}(n_0,\n) =  (-1)^{\sum_{j=1}^m n_j}
 \int_0^\infty {x^{n_{0}}} \frac{x^\al e^{-x}}{\Gamma(\al+1)} \prod_{j=1}^m n_j!L_{n_j}^{(\al)} (x) dx.
 \label{eqgemLag}
 \ee

A main result of this section is     the following theorem.
 \begin{thm}\label{mresult}
   The integrals   $\{A^{(\al)}(n_0, \n)\}$ have  the generating function
 \begin{multline} \label{eqgFgLN}
 \sum_{n_0, \ldots, n_{m}\geq 0}  A^{(\al)}(n_0, \n) \frac{ x_0^{n_0}}{n_0!}  \cdots \frac{x_{m}^{n_m}}{n_{m}!}\\
= \Bigl[ 1  -x_{0} \prod_{j=1}^m(1+x_j)
 -  e_2({\x})  - 2 e_3({\x}) - \cdots - (m-1)e_m({\x})
 \Bigr]^{-\al-1}.
 \end{multline}
 Moreover, we have  the following combinatorial interpretation:
\begin{align}\label{eq:intmresult}
A^{(\al)}(n_0,\n) =\sum_{\pi\in \s^*(\n)}(\al+1)^{\cyc(\pi)},
\end{align}
where $\s^*(\n)$ is the set of permutations of $S_0\cup \cdots \cup S_{m}$ such that
all the elements in box $j$ should not stay in the original box after permutation for $1\leq j\leq m$ and
the objects in box $0$ are not  restricted.
\end{thm}
\begin{proof}
We use \eqref{eqGFLag} to see that
\begin{align*}
&\sum_{n_0, \ldots, n_{m}\geq 0}A^{(\al)}(n_0, \n) \prod_{j=0}^{m} \frac{x_j^{n_j}}{n_j!} \\
&= \frac{1}{\Gamma(\al+1)}\prod_{j=1}^m (1+x_j)^{-\al-1}  \int_0^\infty
\exp\Bigl( -x\bigl(1- x_{0} - \sum_{k=1}^m x_k/(1+x_k)\bigr)\Bigr) x^\al dx \\
&= \prod_{j=1}^m (1+x_j)^{-\al-1}  \Bigl[ 1-x_{0} - \sum_{k=1}^m \frac{x_k}{1+x_k} \Bigr]^{-\al-1},
\end{align*}
which reduces to the right-hand side of \eqref{eqgFgLN} after some simplification using
the following identity, which was proved in \cite{Ask:Ism:Ras}, see also \cite[(2.8)]{Ask:Ism},
\be
\label{eqsimpGF}
\prod_{j=1}^m (1+t_j) \Bigl[1- \sum_{j=1}^m \frac{t_j}{1+t_j} \Bigr] = 1- e_2({\x}) - 2e_3({\x}) -\cdots - (m-1)e_m({\x}).
\ee
This proves  \eqref{eqgFgLN}.  The rest of Theorem \ref{mresult} follows from the $\beta$-MacMahon Master theorem and
\eqref{detcomputation}.
\end{proof}

\begin{rem}
When $\al=0$, $A^{(0)}(n_0, \n)/n_0!n_1!\cdots n_m!$ can be simply interpreted as follows:
we have  boxes  of sizes $n_0, n_1, \ldots, n_{m}$ and box $j$ contains $n_j$ indistinguishable elements and we arrange the
contents such that no object in box $j$ stays in its original box when $1 \le j \le m$ with no restriction on box number $0$.
The number of possible rearrangements is $A^{(0)}(n_0, \n)/n_0!n_1!\cdots n_m!$.
\end{rem}

\begin{cor} \label{pos}
We have
\be
A^{(0)}(m,n,s) = m!n!s!\sum_{j \ge 0} \binom{m}{j}\binom{s}{n+j-m} \binom{s+m-j}{m}.
\ee
\end{cor}
\begin{proof}
By \eqref{mresult} we have the generating function
$$
\sum_{m,n, s \ge 0} A^{(\al)}(m,n,s) \; \frac{x_1^m}{m!}\frac{x_2^n}{n!}\frac{ x_0^s}{s!}   = \frac{1}{[1-x_0 - x_1x_2-x_1x_0-x_2x_0 - x_1x_2x_0]^{\al+1}}.
$$
Since
$$
 V = \begin{vmatrix}
1 & -x_1 & -x_1 \\
-x_2 & 1  & - x_2\\
-x_0 & - x_0 & 1 -x_0
\end{vmatrix}  = 1-x_0 - x_1x_2-x_1x_0-x_2x_0 - x_1x_2x_0
$$
by  the MacMahon Master theorem, Theorem
\ref{MacMahon},   we see that $A^{(0)}(m,n,s)$ is given by  the coefficient of $x_1^m\, x_2^n\, x_0^s$ in
$(x_2+x_0)^m (x_1 +  x_0)^n (x_1+ x_2+ x_0)^s$, which is equal to the claimed expression.
\end{proof}

Motivated by the numbers $A^{(\al)}(n_0,\n)$  we consider  the following generalized linearization coefficients of Meixner polynomials:
\begin{multline}  \label{eqgenMeixner}
 B^{(\bet)}(n_0,\n)  =  (-1)^{\sum_{j=1}^m n_j} c^{-n_{0}} (1-c)^{\bet+n_{0}}  \qquad \qquad  \\
\qquad \qquad  \times
 \sum_{x=0}^\infty {x(x-1) \cdots (x-n_{0}+1)} \frac{c^x (\bet)_x}{x!}
 \prod_{j=1}^m M_{n_j}(x; \bet,c).
 \end{multline}

\begin{thm}\label{mresultM}
The integrals $\{B^{(\bet)}(n_0,\ldots, n_{m})\}$ have the generating function
\begin{multline}\label{eqgfMnumbers}
\sum_{n_0, \ldots, n_{m}\geq 0} B^{(\bet)}(n_0,\n)   \prod_{j=0}^{m} \frac{x_j^{n_j}}{n_j!}
=\Bigg[1-  \sum_{k=2}^m \frac{1-c^{1-k}}{c(1-c^{-1})} \; e_k({\x}) - x_{0} \prod_{j=1}^m(1+x_j/c) \Bigg]^{-\bet}.
\end{multline}
Moreover,
we have  the following combinatorial interpretation:
\begin{align}\label{eq:intMeixnergen}
B^{(\beta)}(n_0, \n) =\sum_{\pi\in \s^*(\n)}\beta^{\cyc(\pi)} c^{-\exc(\pi)},
\end{align}
where $\s^*(\n)$ is the same as in Theorem~\ref{mresult}.
\end{thm}
\begin{proof}
We use  \eqref{eqgfM}  to see that
\begin{align*}
&\sum_{n_0, \ldots, n_{m}\geq 0} B^{(\beta)}(n_0,\n)
 \frac{x_0^{n_0}}{n_0!}\cdots \frac{x_m^{n_m}}{n_m!} \\
&=\sum_{x\geq 0}(1+(1-c)x_{0}/c)^x \frac{c^{x}(\beta)_x}{x!} (1-c)^\beta\prod_{j=1}^m (1+x_j/c)^x (1+x_j)^{-x-\beta}\\
&=\bigg[\frac{\prod_{j=1}^m (1+x_j)-(c+(1-c)x_{0}) \prod_{j=1}^m (1+x_j/c)}{1-c}\bigg]^{-\beta}.
\end{align*}
This gives \eqref{eqgfMnumbers} after simplification.

Comparing with \eqref{detcomputation} we see that  the $\beta=1$ case of~\eqref{eq:intMeixnergen} comes from the
MacMahon Master theorem associated with the matrix   $(a_{ij})$ with $a_{ii}=0,\, a_{ij}=1/c$ for $j>i$
and $a_{ij}=1$ for $j<i$.
The general case follows from using  the $\beta$-extension of MacMahon's Master theorem.
\end{proof}

\begin{rem} For the Charlier polynomials we have a similar result for the integral
\be
C^{(a)}(n_0, \n):=\sum_{x\geq 0} x(x-1)\cdots (x-n_0+1) \frac{e^{-a}a^x}{x!} \prod_{j=1}^m C_{n_j}^{(a)}(x).
\ee
A straight computation shows that
\begin{multline}
\sum_{n_0, \ldots, n_m\geq 0} C^{(a)}(n_0, \n) \frac{x_0^{n_0}}{n_0!} \frac{x_1^{n_1}}{n_1!} \cdots \frac{x_m^{n_m}}{n_m!}\\
=\exp\left( a[x_0+ x_0e_1(\x)+(x_0+1)e_2(\x)+\cdots +(x_0+1)e_m(\x)]  \right).
\end{multline}
We apply the exponential formula to see that
\be
C^{(a)}(n_0, \n)=\sum_{\pi\in \P^*(n_0,\n)}a^{\bl(\pi)},
\ee
where $\P^*(n_0,\n)$ is the set of partitions of $S_0\cup S_1\cup\cdots \cup S_m$ such that
each block is either a singleton of an element in $S_0$ or  inhomogeneous, i.e.,
no two elements of  $S_j$ ($0\leq j\leq m$) can be in  the
 block.
\end{rem}

It is clear that Theorem \ref{mresult} is the limit $c \to 1^-$ of Theorem \ref{mresultM}.
Similarly we have the following analogue of Corollary \ref{pos}.
\begin{cor} \label{pos-meix}
We have
$$
B^{(1)}(m,n,s) = m!n!s!\sum_{j \ge 0} \binom{m}{j}\binom{s}{n+j-m} \binom{s+m-j}{m} c^{n-2m+j}.
$$
\end{cor}

\begin{cor} \label{connection-meix}
We have
\be
x^n=\frac{c^n}{(1-c)^{n}}\sum_{k=0}^n {n\choose k} (\bet+k)_{n-k} (-1)^k M_k(x;\bet,c).
\ee
\end{cor}
\begin{proof} Let
$x^n=\sum_{k=0}^n c(n,k) M_k(x;\bet,c)$.
Using the orthogonality \eqref{OrthMeixner} we obtain
$$
(1-c)^\bet\sum_{x\geq 0} x^n M_k(x;\bet,c) \frac{(\bet)_k}{x!} c^x=c(n,k) \frac{(\bet)_kk!}{c^k}.
$$
Comparing with \eqref{eqgenMeixner} we see that the left side is equal to
$(-1)^kc^n(1-c)^{-n}B^{(\bet)}(k,n)$.  It remains to compute   $B^{(\bet)}(k,n)$, which, by Theorem~\ref{mresultM},
 is the coefficient of $\frac{x_0^n\,x_1^k}{n!\,k!}$
in the expansion
\begin{align*}
\left[1-x_0(1+x_1/c)\right]^{-\bet}&=\sum_{n\geq 0} \frac{(\beta)_n}{n!} x_0^n (1+x_1/c)^n
=\sum_{n,k\geq 0}   \frac{n!(\bet)_n}{(n-k)!} c^{-k}\frac{x_0^n\,x_1^k}{n!\,k!}.
\end{align*}
Hence $B^{(\bet)}(k,n)=\frac{n!(\bet)_n}{(n-k)!} c^{-k}$. This yields the desired result.
\end{proof}

 Let $\varphi$ be the linear functional  defined by $\varphi(f(x))=\int_\R f(x)d\mu(x)$.
 Then the integral \eqref{type1} contains the following four  special cases:
\begin{enumerate}
\item the evaluation  of $\varphi(x^{n})$  corresponds to  the moments,
\item the evaluation  of  $\varphi(\prod_{j=1}^2 p_{n_j}(x))$ corresponds to the orthogonality,
\item the evaluation  of  $\varphi(x^{n}p_{k}(x))$  combined with the orthogonality corresponds to
the  coefficient $c_{n,k}$ in the expansion $x^{n}=\sum_{k=0}^{n} c_{n,k}p_k(x)$,
\item the evaluation  of  $\varphi(\prod_{j=1}^m p_{n_j}(x))$ corresponds to the linearization coefficients.
\end{enumerate}

Since $A_0x=p_1(x)-B_0$, we have
$$
(A_0x)^{n_0}=
\sum_{l=0}^{n_0} {N\choose l} (-B_0)^{N-l} p_1(x)^l, \quad n_0\in \N.
$$
Therefore,
\begin{align}\label{eq:generalintegral}
\varphi\Big( (A_0x)^{n_0}\prod_{j=1}^m p_{n_j}(x) \Big)
=\sum_{l=0}^{n_0} {n_0\choose l} (-B_0)^{n_0-l} \varphi\Big( p_1(x)^l\prod_{j=1}^m p_{n_j}(x) \Big).
\end{align}

We can deduce the combinatorial interpretations
of  the integrals \eqref{eq:generalintegral} for the orthogonal Sheffer polynomials and the three $q$-analogues  
from the combinatorial interpretation of the corresponding linearization coefficients.

For example,  as $H_1(x)=2x$, it follows from Theorem~\ref{thmAzGiVi} that
\begin{align}
 2^{-(n_0+n_1+\cdots +n_m)/2}\int_\R \frac{e^{-x^2}}{\sqrt{\pi}} (2x)^{n_0}  \prod_{j=1}^m H_{n_j}(x) dx
 \end{align}
 is the number of perfect  inhomogeneous matchings in $\K(\n)$ with
 $$\n=(\underbrace{1, \ldots, 1}_{n_0}, n_1,n_2,\ldots, n_m).
 $$

 For the Laguerre polynomials we have  $x=-L_1^{(\alpha)}(x)+\alpha+1$, so
 \begin{multline}\label{gen-foa-zei}
\int_0^\infty \frac{x^\alpha e^{-x}}{\Gamma(\alpha+1)}  x^{n_0}\prod_{j=1}^m (-1)^{n_j} n_j!L_{n_j}^{(\alpha)}(x)dx\\
=\sum_{l=0}^{n_0} {n_0\choose l}
(\alpha+1)^{n_0-l}
\int_0^\infty \frac{x^\alpha e^{-x}}{\Gamma(\alpha+1)}   (-L_1^{(\alpha)}(x))^l\prod_{j=1}^m (-1)^{n_j} n_j!L_{n_j}^{(\alpha)}(x)dx.
\end{multline}
 We can easily recover the combinatorial interpretation~\eqref{eq:intmresult} in Theorem~\ref{mresult} from the above equation and~\eqref{foa-zei-lam}.

 For the Meixner polynomials we have $\frac{1-c}{c}x=\bet-M_1(x;\bet,c)$, so
  \begin{multline}
 \tilde B^{(\beta)}(n_0, \n)=c^{-n_0}(1-c)^{\beta+n_0}\sum_{x=0}^\infty x^{n_0} \frac{c^x(\beta)_x}{x!}\prod_{j=1}^m(-1)^{n_j}M_{n_j}(x;\beta,c)\\
= \sum_{l=0}^{n_0}{n_0\choose l}  \beta^{n_0-l}(1-c)^\beta \sum_{x=0}^\infty  \frac{c^x(\beta)_x}{x!}\left(-M_1(x;\beta,c)\right)^l
 \prod_{j=1}^m(-1)^{n_j}M_{n_j}(x;\beta,c).
 \end{multline}
 Using Theorem~\ref{thlinmeix}, we see  the following combinatorial interpretation
 \be
  \tilde B^{(\beta)}(n_0, \n)=\sum_{\pi\in \s^*(\n)}\beta^{\cyc(\pi)} c^{-\exc(\pi)-\exc_0(\pi)},
\ee
where $\s^*(\n)$ is the same as in Theorem~\ref{mresult} and $\exc_0(\pi)$ is the number of excedances of two elements in $S_0$, i.e.,
$\exc_0(\pi)=|\{i\in S_0: \pi(i)\in S_0 \quad \textrm{and}\quad \pi(i)>i\}|$.

\section{ Laguerre and Meixner polynomials revisited}

Recall \cite[p. 100]{Ism2} that the Hermite polynomials can be viewed as special Laguerre polynomials since
\begin{align*}
H_{2n+1/2\pm 1/2}(x)=(-1)^n 2^{2n} n! (2x)^{1/2\pm 1/2}L_n^{(\pm 1/2)} (x^2).
\end{align*}
Therefore the integral in \eqref{eqAzGiVi} is a special case of the integral

\begin{multline}\label{eqdefW}
W_{j,k} (m; \alpha, \beta; {\m}, {\n}) :=  \frac{(-1)^{\sum_{i=1}^j m_i + \sum_{r=1}^k n_r}}
{\Gamma(\al+1)} \\
\times
\int_0^\infty  x^{m+\al} e^{-x} \bigg[\prod_{i=1}^j m_i! L_{m_i}^{(\alpha)}(x) \bigg]
 \bigg[\prod_{r=1}^k n_r! L_{n_r}^{(\beta)}(x) \bigg]\; dx,
\end{multline}
 where ${\m} = (m_1, m_2, \ldots, m_j)$ and  ${\n} = (n_1, n_2, \ldots, n_k)$.

 In this section we study the combinatorics of  the integrals of the type in \eqref{eqdefW} 
 and their discrete analogues which result by replacing the Laguerre polynomials by Meixner polynomials.

\begin{thm}\label{mresult2}
Let $e_i(\x)$, $i=0,1, \ldots, j+k$, be  the $i$th elementary symmetric polynomial of $x_1,\ldots, x_{j+k}$.
The integrals  $\{W_{j,k} (m; \alpha, \beta; \m, {\n}) \}$ have the generating function
\begin{multline}\label{eqGFW}
\sum_{m, \,  m_i,\,n_r\ge 0} W_{j,k} (m; \alpha, \beta; \m, {\n})\frac{x_0^m}{m!}  \frac{x_1^{m_1}}{m_1!} \cdots \frac{x_j^{m_j}}{m_j!}\;
\frac{x_{j+1}^{n_1}}{n_1!} \cdots \frac{x_{j+k}^{n_k} }{n_k!}\\
=   \prod_{r=1}^k (1+ x_{j+r})^{\al-\beta}\bigg[ 1 -x_0\prod_{i=1}^{j+k}(1+x_i)- \sum_{l=2}^{j+k} (l-1)e_l(\x) \bigg]^{-\al-1}.
\end{multline}
\end{thm}
\begin{proof}
Apply the generating function \eqref{eqGFLag} to see that the left-hand side of \eqref{eqGFW} is given by
\begin{multline}
 \prod_{i=1}^j     (1+x_i)^{-\al -1}    \prod_{r=1}^k (1+ x_{j+r})^{-\beta-1}\int_0^\infty  \frac{x^{\al}}
 {\Gamma(\al+1)}
 \exp\bigg(-x+xx_0+ \sum_{l=1}^{j+k} \frac{x x_l}{1+x_l}\bigg)\; dx      \\
=   \prod_{i=1}^j     (1+x_i)^{-\al -1}    \prod_{r=1}^k (1+ x_{j+r})^{-\beta-1} \Big[1-x_0- \sum_{l=1}^{j+k}
\frac{x_l}{1+x_l}\Big]^{-\al-1}.
 \end{multline}
This establishes \eqref{eqGFW} after some simplification using \eqref{eqsimpGF}.
\end{proof}

\begin{cor}
The numbers $\{W_{j,k} (m; \alpha, \beta; {\m}, {\n}) \}$ are positive when $\al > -1$ and $\al - \beta$ is a nonnegative integer.
\end{cor}

Assuming that $\alpha-\beta$ is a positive integer $N$,
we can  give  a combinatorial interpretation for $W_{j,k} (m; \alpha, \beta; {\m}, {\n}) $.
Let $\s_N^*(\n)$ be  the set of $(k+1)$-tuples $(\pi, f_1, \ldots, f_k)$ such that
\begin{itemize}
\item $\sigma$ is an inhomogeneous permutation of $S_0^*\cup S_1\cup \cdots \cup S_{j}\cup S^*_{j+1}\cup\cdots \cup S^*_{j+k}$, where
$S_0^*\subseteq S_0$ and
$S^*_{j+r}\subseteq S_{j+r}$ for $r=1,\ldots, k$.
\item  $f_r:  S_{j+r}\setminus S^*_{j+r}\to [N]$ is an injection for $r=1,\ldots, k$.
\end{itemize}

{}From Theorems~\ref{mresult}  and \ref{mresult2} we deduce   the following combinatorial interpretation:
\begin{align}
W_{j,k} (m; \alpha, \beta; {\m}, {\n}) =\sum_{(\pi, f_1, \ldots, f_k) \in \s_N^*(\n)}(\al+1)^{\cyc(\pi)}.
\end{align}
%
%
Motivated by the numbers $W_{j,k} (m; \alpha, \beta; {\m}, {\n})$  we consider  the following generalized linearization coefficients of Meixner polynomials:
\begin{multline} \label{eqgenMeixner2}
 Y_{j,k}(m; \alpha, \beta; c; \m,\n):
 =  (-1)^{\sum_{i=1}^j m_i+\sum_{r=1}^k n_r} c^{-m} (1-c)^{\al+m}
  \qquad \qquad  \\
 \times
 \sum_{x=0}^\infty {x(x-1) \cdots (x-m+1)} \frac{c^x (\al)_x}{x!}
\bigg[ \prod_{i=1}^j M_{m_i}(x; \al,c)\bigg]  \bigg[ \prod_{r=1}^k M_{n_r}(x; \bet,c)\bigg],
 \end{multline}
 where ${\m} = (m_1, m_2, \ldots, m_j)$ and  ${\n} = (n_1, n_2, \ldots, n_k)$.

\begin{thm}\label{mresultM2}
The integrals $ Y_{j,k}(m; \alpha, \beta; c; \m,\n)$ have the generating function
\begin{multline}\label{eqgfMnumbers2}
\sum_{m, \, m_i, \, n_r \ge 0} Y_{j,k}(m; \alpha, \beta; c; \m,\n) \frac{x_0^m}{m!}
\prod_{i=1}^{j} \frac{x_i^{m_i}}{m_i!}\; \prod_{r=1}^{k} \frac{x_{j+r}^{n_r}}{n_r!} \\
=\prod_{r=1}^k (1+x_{j+r})^{\al-\bet} \bigg[1-  \sum_{l=2}^{j+k} \frac{1-c^{1-l}}{c(1-c^{-1})} \; e_l(\x) - x_{0} \prod_{i=1}^{j+k}(1+x_i/c) \bigg]^{-\al}.
\end{multline}
\end{thm}
\begin{proof}
Applying \eqref{eqgfM} to  see that the left-hand side of \eqref{eqgfMnumbers2} is
\begin{align*}
&\sum_{x\geq 0}(1+(1-c)x_{0}/c)^x \frac{c^{x}(\al)_x}{x!} (1-c)^\al\\
&\qquad \times \prod_{i=1}^j (1+x_i/c)^x (1+x_i)^{-x-\al} \prod_{r=1}^k (1+x_{j+r}/c)^x (1+x_{j+r})^{-x-\bet}\\
&=(1-c)^\al \prod_{i=1}^j (1+x_i)^{-\al}  \prod_{r=1}^k (1+x_{j+r})^{-\bet}
\bigg[1-(c+(1-c)x_{0}) \prod_{i=1}^{j+k} \frac{ 1+x_i/c}{1+x_i}\bigg]^{-\al}.
\end{align*}
This establishes \eqref{eqgfMnumbers2} after some simplification using \eqref{eqsimpGF}.
\end{proof}

In the same vein, assuming that $\alpha-\beta$ is a positive integer $N$,   Theorems~\ref{mresultM} and \ref{mresultM2} imply
the following combinatorial interpretation:
\begin{align}
Y_{j,k}(m; \alpha, \beta; c; \m,\n)=\sum_{(\pi, f_1, \ldots, f_k)\in \s_N^*(\n)}\al^{\cyc(\pi)} c^{-\exc(\pi)}.
\end{align}

Note that Theorem  \ref{mresultM2}  shows that the numbers $Y_{j,k}(m; \alpha, \beta; c; \m,\n)$ are positive when $\al -\beta$ is a nonnegative integer.



\section{Proof of  Lemma~\ref{lem:sym-charlier}: Symmetry of $\F(\n|\,q)$}\label{bijection11}
Recall that $\n=(n_1, \ldots, n_m)$ is a sequence of positive
integers and $n=n_1+\cdots +n_m$. Clearly we need only to prove the
invariance  of $\F(\n|\,q)$  for the two following permutations of the
indices $n_j$'s: the transposition exchanging 1 and 2, and the
cyclic permutation mapping $i$ to $i+1 \pmod{m}$ for $i=1, \ldots,
m$. Moreover, since $\sg(\pi)=0$ and $\tr(\pi)=n-2\bl(\pi)$ for any
partition $\pi\in \P(\n)$, we see that Lemma~\ref{lem:sym-charlier}
is equivalent to the following result.

\begin{lem}\label{lem:symmetry} We have
\begin{align}\label{eqlem:1}
\sum_{\pi\in \P(\n)}a^{\bl (\pi)}q^{\rc(\pi)}&=
\sum_{\pi\in \P(n_2,n_3,\ldots,n_m,n_1)}a^{\bl (\pi)}q^{\rc(\pi)},\\
\sum_{\pi\in \P(\n)}a^{\bl (\pi)} q^{\rc(\pi)}&=
\sum_{\pi\in \P(n_2,n_1,n_3\ldots,n_m)}a^{\bl (\pi)} q^{\rc(\pi)}.\label{eqlem:2}
\end{align}
\end{lem}

For a positive  integer $k$ such that  $k<n$,  we introduce  two sets of inhomogeneous partitions:
$$
^{(k)}\P_{n}:=\P(k, \underbrace{1, \ldots, 1}_{n-k}), \quad  \P_{n}^{(k)}:=\P(\underbrace{1, \ldots, 1}_{n-k}, k).
$$
In other words,  a partition $\pi$ of $[n]$ is in $^{(k)}\P_{n}$
(resp., $\P_{n}^{(k)}$)  if and only if it has no singleton and
there is no arc in $\pi$ joining two elements in $[1,k]$ (resp.,
$[n-k+1,n]$). For instance, the two partitions $\pi_1$ and $\pi_2$
drawn at the top of Figure~\ref{fig:mapping F and G} are in
${^{(4)}\P_{13}}$ and ${\P_{13}^{(4)}}$. We first show that the
following  result implies \eqref{eqlem:1}.

\begin{prop}\label{prop:Phi}
 For any positive integer $k$,  there is a  bijection  $\Phi_{n,k}: {^{(k)}\P_{n}} \mapsto {\P_{n}^{(k)}}$
 such that for any $\pi\in {^{(k)}\P_{n}}$, we have
\begin{enumerate}
\item[(I)]  for $k< i<j$, the pair $(i,j)$ is an arc of $\pi$ if and
     only if the pair $(i-k,j-k)$ is an arc of $\Phi_{n,k}(\pi)$;
\item[(II)]  $\bl (\Phi_{n,k}(\pi))=\bl(\pi)$ and $\rc (\Phi_{n,k}(\pi))=\rc(\pi)$.
\end{enumerate}
\end{prop}

Indeed, assuming the existence of such a bijection $\Phi_{n,k}$ with $k=n_1$, as
$\P(\n)\subseteq {}^{(n_1)}\P_{n}$,  the
property (I)  implies that
$\Phi_{n,n_1}(\P(\n))\subseteq\P(n_2,n_3,\ldots,n_m,n_1)$.
Since the cardinality of $\P(\n)$ is invariant by permutations of the
$n_i$'s and $\Phi_{n,n_1}$ is bijective, we  deduce that
$$
\Phi_{n,n_1}(\P(\n))=\P(n_2,n_3,\ldots,n_m,n_1),
$$
 and then~\eqref{eqlem:1} by applying the property (II).

 We now turn our attention to~\eqref{eqlem:2}. Define the set of inhomogeneous partitions
$$
{\P_{n}^{\,(n_1,n_2)}}:=\P(n_1,n_2, \underbrace{1, \ldots, 1}_{n-n_1-n_2}).
$$
In other words,  a partition $\pi$ of $[n]$ is in ${\P_{n}^{\,(n_1,n_2)}}$  if
and only if it has no singleton and  there is no arc connecting two integers in $[1,n_1]$
or in $[n_1+1, n_1+n_2]$. For instance,  the partitions $\pi_1$ and $\pi_2$  drawn
 in Figure~\ref{fig:mapping H,Gamma,Theta} are, respectively,  in ${\P_{14}^{(3,4)}}$  and ${\P_{14}^{(4,3)}}$.
Similarly,   we deduce \eqref{eqlem:2} from  the
  following result.

\begin{prop}\label{prop:Theta}
 There is a bijection $\Theta_n^{(n_1,n_2)}:\P^{\,(n_1,n_2)}_{n}\to \P^{\,(n_2,n_1)}_{n}$
  such that for any  $\pi\in \P^{\,(n_1,n_2)}_{n}$, we have
\begin{enumerate}
\item[(I)] for $N_2< i<j$, the pair $(i,j)$ is an arc of $\pi$ if and
     only if the pair $(i,j)$ is an arc of~$\Theta_n^{(n_1,n_2)}(\pi)$, where $N_2:=n_1+n_2$;
\item[(II)]  $\bl (\Theta_n^{(n_1,n_2)}(\pi))=\bl(\pi)$ and $\rc (\Theta_n^{(n_1,n_2)}(\pi))=\rc(\pi)$.
\end{enumerate}
\end{prop}

Indeed, since $\P(\n)\subseteq \P_{n}^{(n_1,n_2)}$, the property (I)
of  $\Theta_n^{(n_1,n_2)}$ implies that
$\Theta_n^{(n_1,n_2)}(\P(\n))\subseteq\P(n_2,n_1,n_3,\ldots,n_m)$.
This, combined with the fact that the cardinality of $\P(\n)$ is
invariant by permutations of the $n_i$'s and $\Theta_n^{(n_1,n_2)}$
is bijective, implies that
$$\Theta_n^{(n_1,n_2)}(\P(\n))=\P(n_2,n_1,n_3,\ldots,n_m).
$$
Equation~\eqref{eqlem:2} then follows by applying the property (II)
of $\Theta_n^{(n_1,n_2)}$.

The next two subsections are dedicated to the proof of
Propositions~\ref{prop:Phi} and \ref{prop:Theta}.

\subsection{Construction of the bijection $\Phi_{n,k}$ }\label{sect:bijection1}

Given a partition $\pi\in \Pi_n$, an element $i\in [n]$ is said to be
\emph{minimal} (resp., \emph{maximal})  if $i$ is the least
(resp., largest) element of a block of $\pi$. The set of the minimal
(resp., maximal) elements in $\pi$ will be denoted $\min(\pi)$
(resp., $\max(\pi)$). For example, for $\pi= 1\,4\,6/2/3\,5$,
$\min(\pi) = \{1, 2, 3\}$ and $\max(\pi) = \{2, 5, 6\}$. Note that
 $\min(\pi)\cap\max(\pi)=\sing(\pi)$ where $\sing(\pi)$ is for the set of singletons of $\pi$.
Let $S$ be a subset of $X$.
 The \textit{restriction} of a partition $\pi=\{B_1,B_2,\ldots,B_k\}$  of $X$ on $S$
is the partition  $\{B_1\cap S, B_2\cap S, \ldots, B_k\cap S\}$ of $S$.

The key idea  for the definition of the mapping  $\Phi_{n,k}$ is some
appropriate decomposition of partitions in ${^{(k)}\P_{n}}$ and
${\P_{n}^{(k)}}$.  Let ${^{(k)} A_{n}}$ (resp., ${A_{n}^{(k)}}$) be the set of 3-tuples
$(\tau, R, \sigma)$ where
\begin{itemize}
\item $\tau\in \Pi_{n-k}$ and  $\sigma\in \s_k$,
\item $\sing(\tau)\subseteq R\subseteq \min(\tau)$ (resp., $\sing(\tau)\subseteq R\subseteq \max(\tau)$) and $|R|=k$.
\end{itemize}
For instance, in Figure~\ref{fig:mapping F and G}, we have
$(\tau_1,O,\sigma_1)\in \,{^{(4)}} A_{13}$  and
$(\tau_2,C,\sigma_2)\in \,A_{13}^{(4)}$.

We first define two simpler  mappings $F_{n,k}: {\P_{n}^{(k)}}\to
{A_{n}^{(k)}}$ and $G_{n,k}: {^{(k)}\P_{n}}\to {^{(k)}A_{n}}$.
\begin{itemize}
\item For $\pi\in {\P_{n}^{(k)}}$,  set $F_{n,k}(\pi)=(\tau, C,\sigma)$,  where
\begin{itemize}
\item $\tau$ is the restriction of $\pi$ on $[n-k]$;
\item $C$ is the set of elements in $\pi$ which are connected to an element $>n-k$ by an arc;
\item By definition of ${\P_{n}^{(k)}}$, we have $|C|=k$. Suppose $C=\{c_1<c_2<\cdots<c_k\}$, 
then $\sigma$ is the unique permutation in $\s_k$ such that  $(c_1,n-k+\sigma(1))$,
$(c_2,n-k+\sigma(2))$,\ldots, $(c_k,n-k+\sigma(k))$ are arcs of $\pi$.
\end{itemize}
\item For  $\pi\in {^{(k)}\P_{n}}$, set $G_{n,k}(\pi)=(\tau, O,\sigma)$, where
\begin{itemize}
\item $\tau\in \Pi_{n-k}$ is the partition obtained by subtracting
 $k$ from each element in the restriction of $\pi$ on $[k+1,n]$;
\item Let $M$ be the set of elements in $\pi$ which are connected to an
 element $j\leq k$ by an arc.  By definition of ${\P_{n}^{(k)}}$, we have $|M|=k$.
 Suppose $M=\{m_1<m_2<\cdots<m_k\}$, then $O$ is obtained  by subtracting $k$ from each element of $M$, i.e.,
$O=\{m_1-k,m_2-k,\ldots,m_k-k\}$;
\item  $\sigma$ is the unique permutation in $\s_k$  such that $(\sigma(1),m_1)$, $(\sigma(2),m_2)$,$\ldots$, $(\sigma(k),m_k)$
 are arcs of $\pi$.
\end{itemize}
\end{itemize}
The mappings $F_{n,k}$ and~$G_{n,k}$ are illustrated  in Figure~\ref{fig:mapping F and G}.

\begin{figure}[h!]
{\setlength{\unitlength}{0.45mm}
\begin{picture}(135,95)(-40,-65)
\put(10,0){\line(1,0){80}}\put(-35,0){\line(1,0){30}}
\put(-35,10){\makebox(-10,-4)[c]{\small $\pi_1=$}}
\put(-35,0){\circle*{1}}\put(-35,0){\makebox(-2,-4)[c]{\tiny 1}}
\put(-25,0){\circle*{1}}\put(-25,0){\makebox(-2,-4)[c]{\tiny 2}}
\put(-15,0){\circle*{1}}\put(-15,0){\makebox(-2,-4)[c]{\tiny 3}}
\put(-5,0){\circle*{1}}\put(-5,0){\makebox(-2,-4)[c]{\tiny 4}}
\put(10,0){\circle*{1}}\put(10,0){\makebox(-2,-4)[c]{\tiny 5}}
\put(20,0){\circle*{1}}\put(20,0){\makebox(-2,-4)[c]{\tiny 6}}
\put(30,0){\circle*{1}}\put(30,0){\makebox(-2,-4)[c]{\tiny 7}}
\put(40,0){\circle*{1}}\put(40,0){\makebox(-2,-4)[c]{\tiny 8}}
\put(50,0){\circle*{1}}\put(50,0){\makebox(-2,-4)[c]{\tiny 9}}
\put(60,0){\circle*{1}}\put(60,0){\makebox(-2,-4)[c]{\tiny 10}}
\put(70,0){\circle*{1}}\put(70,0){\makebox(-2,-4)[c]{\tiny 11}}
\put(80,0){\circle*{1}}\put(80,0){\makebox(-2,-4)[c]{\tiny 12}}
\put(90,0){\circle*{1}}\put(90,0){\makebox(-2,-4)[c]{\tiny 13}}
\qbezier(60,0)(70,15)(80,0)\qbezier(10,0)(25,23)(40,0)
\qbezier(30,0)(55,27)(70,0) \qbezier(80,0)(85,15)(90,0)
\red{\qbezier(-35,0)(-7.5,30)(20,0)\qbezier(-25,0)(12.5,40)(50,0)
\qbezier(-15,0)(-2.5,18)(10,0)\qbezier(-5,0)(12.5,27)(30,0)}
\put(40,-6){\vector(0,-1){20}}\put(18,-16){$G_{13,4}$}
\put(-5,-36){$\tau_1=$}
\put(10,-40){\line(1,0){80}}
\put(10,-40){\circle*{1}}\put(10,-40){\makebox(-2,-4)[c]{\tiny 1}}
\put(20,-40){\circle*{1}}\put(20,-40){\makebox(-2,-4)[c]{\tiny 2}}
\put(30,-40){\circle*{1}}\put(30,-40){\makebox(-2,-4)[c]{\tiny 3}}
\put(40,-40){\circle*{1}}\put(40,-40){\makebox(-2,-4)[c]{\tiny 4}}
\put(50,-40){\circle*{1}}\put(50,-40){\makebox(-2,-4)[c]{\tiny 5}}
\put(60,-40){\circle*{1}}\put(60,-40){\makebox(-2,-4)[c]{\tiny 6}}
\put(70,-40){\circle*{1}}\put(70,-40){\makebox(-2,-4)[c]{\tiny 7}}
\put(80,-40){\circle*{1}}\put(80,-40){\makebox(-2,-4)[c]{\tiny 8}}
\put(90,-40){\circle*{1}}\put(90,-40){\makebox(-2,-4)[c]{\tiny 9}}
\qbezier(60,-40)(70,-25)(80,-40)\qbezier(10,-40)(25,-17)(40,-40)
\qbezier(30,-40)(55,-13)(70,-40) \qbezier(80,-40)(85,-25)(90,-40)
\put(-4,-55){$O=\{1,2,3,5\}$}
\put(-5,-65){$\sigma_1=3\,1\,4\,2$}
\put(96,0){\vector(1,0){30}}  \put(98,4){$\Phi_{13,4}$}
\put(93,-50){\vector(1,0){25}}\put(96,-60){$\Psi_{13,4}$}
\end{picture}}\hspace{1.5cm}
{\setlength{\unitlength}{0.45mm}
\begin{picture}(135,95)(10,-65)
\put(10,0){\line(1,0){80}}\put(105,0){\line(1,0){30}}
\put(10,0){\circle*{1}}\put(10,0){\makebox(-2,-4)[c]{\tiny 1}}
\put(20,0){\circle*{1}}\put(20,0){\makebox(-2,-4)[c]{\tiny 2}}
\put(30,0){\circle*{1}}\put(30,0){\makebox(-2,-4)[c]{\tiny 3}}
\put(40,0){\circle*{1}}\put(40,0){\makebox(-2,-4)[c]{\tiny 4}}
\put(50,0){\circle*{1}}\put(50,0){\makebox(-2,-4)[c]{\tiny 5}}
\put(60,0){\circle*{1}}\put(60,0){\makebox(-2,-4)[c]{\tiny 6}}
\put(70,0){\circle*{1}}\put(70,0){\makebox(-2,-4)[c]{\tiny 7}}
\put(80,0){\circle*{1}}\put(80,0){\makebox(-2,-4)[c]{\tiny 8}}
\put(90,0){\circle*{1}}\put(90,0){\makebox(-2,-4)[c]{\tiny 9}}
\put(105,0){\circle*{1}}\put(105,0){\makebox(-2,-4)[c]{\tiny10}}
\put(115,0){\circle*{1}}\put(115,0){\makebox(-2,-4)[c]{\tiny11}}
\put(125,0){\circle*{1}}\put(125,0){\makebox(-2,-4)[c]{\tiny 12}}
\put(135,0){\circle*{1}}\put(135,0){\makebox(-2,-4)[c]{\tiny 13}}
\put(135,10){\makebox(10,-4)[c]{\small $=\pi_2$}}
\qbezier(60,0)(70,15)(80,0)\qbezier(10,0)(25,23)(40,0)
\qbezier(30,0)(55,27)(70,0) \qbezier(80,0)(85,15)(90,0)
\red{\qbezier(40,0)(72.5,35)(105,0)\qbezier(90,0)(102.5,20)(115,0)
\qbezier(20,0)(72.5,43)(125,0)\qbezier(50,0)(90.5,30)(135,0)}
\put(40,-6){\vector(0,-1){20}}\put(44,-16){$F_{13,4}$}
\put(-5,-36){$\tau_2=$}
\put(10,-40){\line(1,0){80}}
\put(10,-40){\circle*{1}}\put(10,-40){\makebox(-2,-4)[c]{\tiny 1}}
\put(20,-40){\circle*{1}}\put(20,-40){\makebox(-2,-4)[c]{\tiny 2}}
\put(30,-40){\circle*{1}}\put(30,-40){\makebox(-2,-4)[c]{\tiny 3}}
\put(40,-40){\circle*{1}}\put(40,-40){\makebox(-2,-4)[c]{\tiny 4}}
\put(50,-40){\circle*{1}}\put(50,-40){\makebox(-2,-4)[c]{\tiny 5}}
\put(60,-40){\circle*{1}}\put(60,-40){\makebox(-2,-4)[c]{\tiny 6}}
\put(70,-40){\circle*{1}}\put(70,-40){\makebox(-2,-4)[c]{\tiny 7}}
\put(80,-40){\circle*{1}}\put(80,-40){\makebox(-2,-4)[c]{\tiny 8}}
\put(90,-40){\circle*{1}}\put(90,-40){\makebox(-2,-4)[c]{\tiny 9}}
\qbezier(60,-40)(70,-25)(80,-40)\qbezier(10,-40)(25,-17)(40,-40)
\qbezier(30,-40)(55,-13)(70,-40) \qbezier(80,-40)(85,-25)(90,-40)
\put(-4,-65){$\sigma_2=3\,1\,4\,2$}
\put(-3,-55){$C=\{2,4,5,9\}$}
\end{picture}}
\caption{The mappings $\Phi_{n,k}$, $F_{n,k}$, $G_{n,k}$ and $\Psi_{n,k}$ }\label{fig:mapping F and G}
\end{figure}
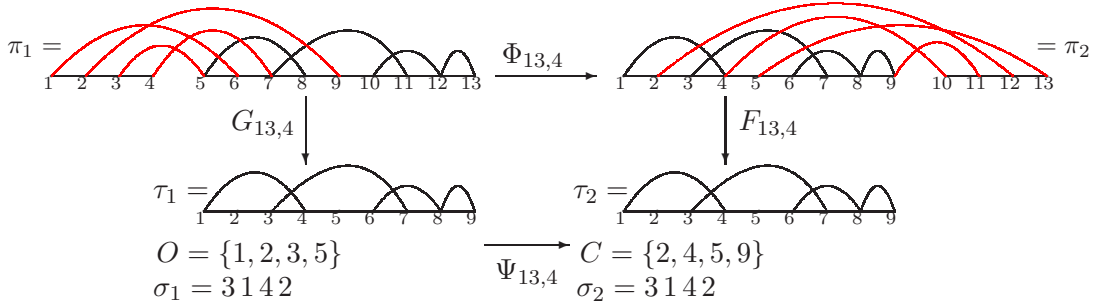

\begin{defi}\label{def: depth}
 Let $\pi$ be a partition of a set $S$ consisting of positive integers.
The \emph{depth} of an element $i$ in $\pi$, denoted $\d_i(\pi)$, is the
number of arcs $(a,b)$ in $\pi$ satisfying $a<i<b$.
\end{defi}

\begin{defi}\label{def: noninversion}
Let $\sigma=\sigma(1)\,\sigma(2)\cdots\sigma(n)$ be a permutation of $[n]$.
A pair $(i,j)$, $1\leq i<j\leq n$,  is said to be a \textit{non-inversion} in $\sigma$ if $\sigma(i)<\sigma(j)$.
The number of non-inversions in $\sigma$ will be denoted $\ninv(\sigma)$.
\end{defi}

Some useful properties of $F_{n,k}$ and $G_{n,k}$  are summarized in
the following result.
\begin{prop}\label{prop:F-lem}
The mappings $F_{n,k}: {\P_{n}^{(k)}} \to {A_{n}^{(k)}}$ and
$G_{n,k}: {^{(k)}\P_{n}} \to {^{(k)}A_{n}}$ are   bijections.
Moreover, for any $\pi\in {\P_{n}^{(k)}}$, if $F_{n,k}(\pi)=(\tau,C,
\sigma)$, then
\begin{align}\label{eq:propF}
\bl(\pi)=\bl(\tau)\quad\text{and}\quad
\rc(\pi)=\rc(\tau)+\ninv(\sigma)+\sum_{i\in C}\d_i(\tau),
\end{align}
and, for any $\pi\in {^{(k)}\P_{n}}$, if
$G_{n,k}(\pi)=(\tau,O, \sigma)$, then
\begin{align}\label{eq:propG}
\bl(\pi)=\bl(\tau)\quad\text{and}\quad
\rc(\pi)=\rc(\tau)+\ninv(\sigma)+\sum_{i\in O}\d_i(\tau).
\end{align}
\end{prop}
\begin{proof}
It is easy to see  that $F_{n,k}$ (resp., $G_{n,k}$) is a bijection
by constructing its inverse (use Figure~\ref{fig:mapping F and G}).

Let $S$ be  a finite subset of positive integers.  Clearly, if $\pi$
is a partition of $S$, then each block~$B$ of $\pi$ is represented
by $|B|-1$ arcs. This easily leads to the following result.
\begin{fact} \label{fact:blocks-arcs}
The number of blocks of a partition $\pi$ of $S$
is equal to $|S|-(\text{number of arcs in $\pi$})$.
\end{fact}
The first equation in~\eqref{eq:propF} and \eqref{eq:propG} is just
a consequence of the above fact.  We now turn our attention to the
second equation in~\eqref{eq:propF} and \eqref{eq:propG}.
 Let $\pi\in {\P_{n}^{(k)}}$.  Clearly, the arc crossings in the partition~$\pi$  can be divided into
three classes $R_1(\pi)$, $R_2(\pi)$ and $R_3(\pi)$ illustrated in
Table~\ref{tab:forme-cr-1}.
\begin{table}[ht]
$$
\begin{array}{|c|c|c|}
\hline
 i&L_i(\pi)&R_i(\pi)\\
 \hline
1&\hspace{1cm}{\setlength{\unitlength}{0.6mm}
\begin{picture}(80,20)(0,-10)
\put(0,0){\line(1,0){25}}\put(35,0){\line(1,0){35}}
\put(0,0){\circle*{1,3}}\put(0,0){\makebox(-2,-6)[c]{\tiny 1}}
\put(25,0){\circle*{1,3}}\put(25,0){\makebox(-2,-6)[c]{\tiny $k$}}
\put(35,0){\circle*{1,3}}
\put(70,0){\circle*{1,3}}\put(70,0){\makebox(-2,-6)[c]{\tiny $n$}}
\qbezier(40,0)(50,10)(60,0) \qbezier(45,0)(55,10)(65,0)
\end{picture}}
&\hspace{1cm} {\setlength{\unitlength}{0.6mm}
\begin{picture}(80,10)(0,-10)
\put(0,0){\line(1,0){35}}\put(45,0){\line(1,0){25}}
\put(0,0){\circle*{1,3}}\put(0,0){\makebox(-2,-6)[c]{\tiny 1}}
\put(35,0){\circle*{1,3}}\put(35,0){\makebox(-2,-6)[c]{\tiny $n-k$}}
\put(45,0){\circle*{1,3}}
\put(70,0){\circle*{1,3}}\put(70,0){\makebox(-2,-6)[c]{\tiny $n$}}
\qbezier(5,0)(15,10)(25,0) \qbezier(10,0)(20,10)(30,0)
\end{picture}
}\\
\hline
2&\hspace{1cm}{\setlength{\unitlength}{0.6mm}
\begin{picture}(80,20)(0,-10)
\put(0,0){\line(1,0){25}}\put(35,0){\line(1,0){35}}
\put(0,0){\circle*{1,3}}\put(0,0){\makebox(-2,-6)[c]{\tiny 1}}
\put(25,0){\circle*{1,3}}\put(25,0){\makebox(-2,-6)[c]{\tiny $k$}}
\put(35,0){\circle*{1,3}}
\put(70,0){\circle*{1,3}}\put(70,0){\makebox(-2,-6)[c]{\tiny $n$}}
\qbezier(5,0)(30,10)(45,0) \qbezier(20,0)(40,10)(60,0)
\end{picture}
}&\hspace{1cm}{\setlength{\unitlength}{0.6mm}
\begin{picture}(80,20)(0,-10)
\put(0,0){\line(1,0){35}}\put(45,0){\line(1,0){25}}
\put(0,0){\circle*{1,3}}\put(0,0){\makebox(-2,-6)[c]{\tiny $1$}}
\put(35,0){\circle*{1,3}}\put(35,0){\makebox(-2,-6)[c]{\tiny $n-k$}}
\put(45,0){\circle*{1,3}}
\put(70,0){\circle*{1,3}}\put(70,0){\makebox(-2,-6)[c]{\tiny $n$}}
\qbezier(10,0)(40,10)(50,0) \qbezier(25,0)(45,10)(65,0)
\end{picture}
}\\
\hline
3&\hspace{1cm}{\setlength{\unitlength}{0.6mm}
\begin{picture}(80,20)(0,-10)
\put(0,0){\line(1,0){25}}\put(35,0){\line(1,0){35}}
\put(0,0){\circle*{1,3}}\put(0,0){\makebox(-2,-6)[c]{\tiny 1}}
\put(25,0){\circle*{1,3}}\put(25,0){\makebox(-2,-6)[c]{\tiny $k$}}
\put(35,0){\circle*{1,3}}
\put(70,0){\circle*{1,3}}\put(70,0){\makebox(-2,-6)[c]{\tiny $n$}}
\qbezier(12.5,0)(37.5,10)(52.5,0) \qbezier(45,0)(52.5,10)(60,0)
\end{picture}
}&\hspace{1cm}{\setlength{\unitlength}{0.6mm}
\begin{picture}(80,10)(0,-10)
\put(0,0){\line(1,0){35}}\put(45,0){\line(1,0){25}}
\put(0,0){\circle*{1,3}}\put(0,0){\makebox(-2,-6)[c]{\tiny 1}}
\put(35,0){\circle*{1,3}}\put(35,0){\makebox(-2,-6)[c]{\tiny $n-k$}}
\put(45,0){\circle*{1,3}}
\put(70,0){\circle*{1,3}}\put(70,0){\makebox(-2,-6)[c]{\tiny $n$}}
\qbezier(10,0)(17.5,10)(25,0) \qbezier(17.5,0)(37.5,10)(57.5,0)
\end{picture}
}\\
\hline
\end{array}
$$
\caption{Sketch of crossings in $L_i(\pi)$ and
$R_i(\pi)$.}\label{tab:forme-cr-1}
\end{table}

They are defined formally as follows:
\begin{align*}
R_1(\pi)&=\{(i_1,j_1)(i_2,j_2)\in\pi\;|\;1\leq i_1<i_2<j_1<j_2\leq n-k \},\\
R_2(\pi)&=\{(i_1,j_1)(i_2,j_2)\in\pi\;|\;1\leq i_1<i_2\leq n-k <j_1<j_2\},\\
R_3(\pi)&=\{(i_1,j_1)(i_2,j_2)\in\pi\;|\;1\leq i_1<i_2< j_1\leq n-k<j_2\};
\end{align*}
and satisfy $\rc(\pi)=|R_1(\pi)|+|R_2(\pi)|+|R_3(\pi)|$. Suppose
$F_{n,k}(\pi)=(\tau,C, \sigma)$. Then it is easily checked that
$|R_1(\pi)|=\rc(\tau)$, $|R_2(\pi)|=\ninv(\sigma)$ and
$|R_3(\pi)|=\sum_{i\in C}\d_i(\tau)$ (see  Figure~\ref{fig:mapping F
and G}). This proves the second equation in~\eqref{eq:propF}.

 Similarly, let $\pi\in {\P_{n}^{(k)}}$.
The arc crossings of the partition $\pi$  can be divided into
three parts $L_1(\pi)$, $L_2(\pi)$ and $L_3(\pi)$ illustrated in
Table~\ref{tab:forme-cr-1}, defined formally as follows:
\begin{align*}
L_1(\pi)&=\{(i_1,j_1)(i_2,j_2)\in\pi\;|\;k< i_1<i_2< j_1<j_2\leq n \},\\
L_2(\pi)&=\{(i_1,j_1)(i_2,j_2)\in\pi\;|\;1\leq i_1 <i_2 \leq k <j_1<j_2 \leq n\},\\
L_3(\pi)&=\{(i_1,j_1)(i_2,j_2)\in\pi\;|\;1\leq i_1 \leq k< i_2< j_1< j_2\leq n\},
\end{align*}
and such that $\rc(\pi)=|L_1(\pi)|+|L_2(\pi)|+|L_3(\pi)|$. Suppose
$G_{n,k}(\pi)=(\tau,O, \sigma)$. Then it is easily checked that
$|L_1(\pi)|=\rc(\tau)$, $|L_2(\pi)|=\ninv(\sigma)$ and
$|L_3(\pi)|=\sum_{i\in O}\d_i(\tau)$ (see  Figure~\ref{fig:mapping F
and G}). This proves the second equation in~\eqref{eq:propG}.
\end{proof}

In view of Proposition~\ref{prop:F-lem}, to prove
Proposition~\ref{prop:Phi}, it suffices  to prove the following
result.
\begin{prop}\label{prop:psi_pi}
 For any partition $\pi$,  there is a bijection  $\psi_{\pi}:\min(\pi)\mapsto \max(\pi)$  such that
$\d_i(\pi)=\d_{\psi(i)}(\pi)$ for each $i\in\min(\pi)$
and $\psi_{\pi}(j)=j$ for $j\in\sing(\pi)$.
\end{prop}
\begin{proof}
 It is worth noting that such a bijection was already described in the literature (e.g., see Remark 7.2 in \cite{Kas:Zen3}).
For reader's convenience we recall the construction of $\psi_{\pi}$.
 The mapping~$\psi_{\pi}$ can be nicely illustrated using \textit{Motzkin paths}.
Recall that a Motzkin path of length $n$ is a lattice path
in the plane of integer lattice $\mathbb Z^2$ from~$(0,0)$ to $(n,0)$,
consisting of NE-steps $(1,1)$, E-steps  $(1,0)$ and SE-steps  $(1,-1)$,
which never passes below the $x$-axis. The usual way to associate a set partition to a Motzkin path works as follows:
to a partition $\pi$ of $[n]$ we associate the Motzkin path $M$ of length $n$ whose $i$-th step is NE
if $i\in \min(\pi)\setminus \sing(\pi)$, SE if $i\in \max(\pi)\setminus \sing(\pi)$ and E otherwise.
An illustration of this correspondence
is given in Figure~\ref{fig:MotzkinPath}.
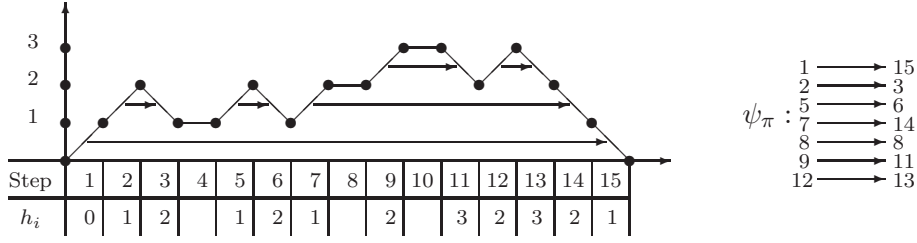
\begin{figure}[h]
\begin{center}
{\setlength{\unitlength}{0.5mm}
\begin{picture}(222,65)(0,-20)
\put(0,0){\circle*{3}}\put(-15,0){\linethickness{0.3mm}\vector(1,0){176}}
\put(-15,-10){\linethickness{0.3mm}\line(1,0){165}}
\put(0,-20){\linethickness{0.3mm}\vector(0,1){62}}\put(0,0){\linethickness{0.2mm}\line(1,1){11.5}}
\put(10,10){\circle*{3}}\put(10,10){\linethickness{2mm}\line(1,1){10.5}}
\put(20,20){\circle*{3}}\put(20,20){\linethickness{2mm}\line(1,-1){10.5}}
\put(30,10){\circle*{3}}\put(30,10){\linethickness{0.2mm}\line(1,0){10}}
\put(40,10){\circle*{3}}\put(40,10){\linethickness{2mm}\line(1,1){10.5}}
\put(50,20){\circle*{3}}\put(50,20){\linethickness{0.2mm}\line(1,-1){10.5}}
\put(60,10){\circle*{3}}\put(60,10){\linethickness{0.2mm}\line(1,1){10.5}}
\put(70,20){\circle*{3}}\put(70,20){\linethickness{0.2mm}\line(1,0){10}}
\put(80,20){\circle*{3}}\put(80,20){\linethickness{0.2mm}\line(1,1){10.5}}
\put(90,30){\circle*{3}}\put(90,30){\linethickness{0.2mm}\line(1,0){10}}
\put(100,30){\circle*{3}}\put(100,30){\linethickness{0.2mm}\line(1,-1){10.5}}
\put(110,20){\circle*{3}}\put(110,20){\linethickness{0.2mm}\line(1,1){10.5}}
\put(120,30){\circle*{3}}\put(120,30){\linethickness{0.2mm}\line(1,-1){10.5}}
\put(130,20){\circle*{3}}\put(130,20){\linethickness{0.2mm}\line(1,-1){10.5}}
\put(140,10){\circle*{3}}\put(140,10){\linethickness{0.2mm}\line(1,-1){10.5}}
\put(150,0){\circle*{3}}\put(10,0){\linethickness{0.2mm}\line(0,-1){20}}
\put(20,0){\linethickness{0.2mm}\line(0,-1){20}}\put(30,0){\linethickness{0.2mm}\line(0,-1){20}}
\put(40,0){\linethickness{0.2mm}\line(0,-1){20}}\put(50,0){\linethickness{0.2mm}\line(0,-1){20}}
\put(60,0){\linethickness{0.2mm}\line(0,-1){20}}\put(70,0){\linethickness{0.2mm}\line(0,-1){20}}
\put(80,0){\linethickness{0.2mm}\line(0,-1){20}}\put(90,0){\linethickness{0.2mm}\line(0,-1){20}}
\put(100,0){\linethickness{0.2mm}\line(0,-1){20}}\put(110,0){\linethickness{0.2mm}\line(0,-1){20}}
\put(120,0){\linethickness{0.2mm}\line(0,-1){20}}\put(130,0){\linethickness{0.2mm}\line(0,-1){20}}
\put(140,0){\linethickness{0.2mm}\line(0,-1){20}}\put(150,0){\linethickness{0.2mm}\line(0,-1){20}}
\put(-10,30){\scriptsize3}\put(-10,20){\scriptsize2}\put(-10,10){\scriptsize1}
\put(0,30){\circle*{3}}\put(0,20){\circle*{3}}\put(0,10){\circle*{3}}
\put(-15,-7){\scriptsize \textrm{Step}}
\put(-12,-17){\scriptsize $h_i$}
\put(5,-7){\scriptsize 1}\put(15,-7){\scriptsize 2}\put(25,-7){\scriptsize 3}\put(35,-7){\scriptsize 4}
\put(45,-7){\scriptsize 5}\put(55,-7){\scriptsize 6}\put(65,-7){\scriptsize 7}\put(75,-7){\scriptsize 8}
\put(85,-7){\scriptsize 9}\put(92,-7){\scriptsize 10}\put(102,-7){\scriptsize 11}\put(112,-7){\scriptsize 12}
\put(122,-7){\scriptsize 13}\put(132,-7){\scriptsize 14}\put(142,-7){\scriptsize 15}
\put(5,-17){\scriptsize 0}\put(15,-17){\scriptsize 1}\put(25,-17){\scriptsize 2}
\put(45,-17){\scriptsize 1}\put(55,-17){\scriptsize 2}\put(65,-17){\scriptsize 1}
\put(85,-17){\scriptsize 2}\put(104,-17){\scriptsize 3}\put(114,-17){\scriptsize 2}
\put(124,-17){\scriptsize 3}\put(134,-17){\scriptsize 2}\put(144,-17){\scriptsize 1}\put(-20,-35){.}
\put(6,5){\linethickness{0.2mm}\vector(1,0){138}}\put(16,15){\linethickness{0.2mm}\vector(1,0){8}}
\put(46,15){\linethickness{0.2mm}\vector(1,0){8}}\put(66,15){\linethickness{0.2mm}\vector(1,0){68}}
\put(86,25){\linethickness{0.2mm}\vector(1,0){18}}\put(116,25){\linethickness{0.2mm}\vector(1,0){8}}
\put(180,10){\textrm{$\psi_{\pi}:$}}
\put(195,23){\scriptsize 1} \put(200,25){\vector(1,0){18}} \put(220,23){\scriptsize 15}
\put(195,18){\scriptsize 2} \put(200,20){\vector(1,0){18}} \put(220,18){\scriptsize 3}
\put(195,13){\scriptsize 5} \put(200,15){\vector(1,0){18}} \put(220,13){\scriptsize 6}
\put(195,8){\scriptsize 7}  \put(200,10){\vector(1,0){18}} \put(220,8){\scriptsize 14}
\put(195,3){\scriptsize 8}  \put(200,5){\vector(1,0){18}}  \put(220,3){\scriptsize 8}
\put(195,-2){\scriptsize 9} \put(200,0){\vector(1,0){18}}  \put(220,-2){\scriptsize 11}
\put(193,-7){\scriptsize 12}\put(200,-5){\vector(1,0){18}} \put(220,-7){\scriptsize 13}
\end{picture}}
\end{center}
\caption{The Motzkin path associated  to the
partition {$\pi=1\,4\,15/2\,3/5\,6/7\,10\,13/8/9\,11/12\,14$}
and the mapping $\psi_{\pi}$}\label{fig:MotzkinPath}
\end{figure}

A  basic property of the above correspondence is the following fact~\cite{Kas:Zen2}.
\begin{fact}\label{fact:depth-height}
 Suppose $M$ is the Motzkin path associated to a partition $\pi$
and let $h_i$ be the height of the $i$-th step of $M$, i.e.,
the ordinate of its originate point. Then, $\d_{i}(\pi)=h_i$
if the $i$-th step of~$M$ is NE and $\d_{i}(\pi)=h_i-1$
if the $i$-th step of~$M$ is SE.
\end{fact}

We can now describe the mapping $\psi_{\pi}$.
We first set $\psi_{\pi}(j)=j$ for $j\in\sing(\pi)$.
Suppose $\O(\pi):=\min(\pi)\setminus \sing(\pi)=\{o_1<o_2<\cdots<o_r\}$,
$\C(\pi):=\max(\pi)\setminus \sing(\pi)=\{c_1<c_2<\cdots<c_r\}$
and let $M$ be the Motzkin path associated to $\pi$.
Note that the NE (resp., SE) steps in $M$ are exactly the steps
indexed by $\O(\pi)$ (resp., $\C(\pi)$). We then pair the NE-steps with SE-steps in $M$ two by two in the following way.
Suppose the $i$-th NE-step (i.e., the $o_i$-th step) of $M$ is at height $h$. Then, if the
first SE-step to its right at height $h+1$ is the $j$-th SE step (i.e., the $c_j$-th step) in $M$,
then we set $\psi_{\pi}(o_i)=c_j$. An illustration is given in Figure~\ref{fig:MotzkinPath}.
From  the construction of $\psi_{\pi}$ and Fact~\ref{fact:depth-height} it is easy to see that  $\psi_{\pi}$ is the desired bijection.
\end{proof}

For $(\tau,O,\sigma) \in {^{(k)}A_{n}}$, we set
$\Psi_{n,k}(\tau,O,\sigma):=(\tau,\psi_{\tau}(O),\sigma)$.
Clearly $\Psi_{n,k}$  is a mapping from ${^{(k)}A_{n}}$ to $ {A_{n}^{(k)}}$.
An illustration is given in Figure~\ref{fig:mapping F and G}. From Proposition~\ref{prop:psi_pi} we immediately deduce
the following result.

\begin{prop}\label{prop:Psi}
 The mapping $\Psi_{n,k}: {^{(k)}A_{n}} \to {A_{n}^{(k)}}$
is a bijection. Moreover, if $(\tau,O,\sigma) \in {^{(k)}A_{n}}$
and $\Psi_{n,k}(\tau,O,\sigma)=(\tau,C,\sigma)$,
then we have
$$\sum_{i\in C} \d_{i}(\tau)=\sum_{i\in O} \d_{i}(\tau).$$
\end{prop}
Finally, we define the mapping $\Phi_{n,k}: {^{(k)}\P_{n}} \to {\P_{n}^{(k)}}$ by
\be
\Phi_{n,k}:= F_{n,k}^{-1} \circ \Psi_{n,k} \circ G_{n,k}.
\ee
This mapping is illustrated in  Figure~\ref{fig:mapping F and G}.  Combining Propositions~\ref{prop:F-lem}
and \ref{prop:Psi}, we conclude  that the mapping $\Phi_{n,k}$ satisfies the requirements of
Proposition~\ref{prop:Phi}.


\subsection{Construction of the bijection $\Theta_n^{(n_1,n_2)}$}\label{sect:bijection2}

 The key idea for the definition of the mapping  $\Theta_n^{(n_1,n_2)}$ is
some appropriate decomposition of partitions in ${\P_{n}^{(n_1,n_2)}}$.
We first introduce some further definitions.
For any set  $K$, let  $\Pi(K)$ be the
set of partitions of $K$.
\begin{definition}
For  two positive integers $r,s$, we denote by $\P^*(r,s)$ the set of all partitions $\pi$
of $[r+s]$ such that there is no arc in $\pi$ connecting two integers in $[1,r]$ or in $[r+1,r+s]$ but $\pi$
can have singletons. Thus, we have   $\P(r,s)\subsetneq\P^*(r,s)$.
\end{definition}

\begin{definition}
Let ${A_{n}^{(n_1,n_2)}}$ be the set of 3-tuples
$( (\tau,A),  (\gamma,B),\sigma)$ where
\begin{itemize}
\item $\tau$ is a partition in $\P^*(n_1,n_2)$ and  $A$ is a set satisfying $\sing(\tau)\subseteq A\subseteq\max(\tau)$;
\item $\gamma$ is a partition in $\Pi([N_2+1,n])$ and $B$ is a set satisfying  $\sing(\gamma)\subseteq B\subseteq\min(\gamma)$;
\item the sets $A$ and $B$ have the same cardinality. If $k=|A|=|B|$, then $\sigma$ is in $\s_k$.
\end{itemize}
\end{definition}
For instance, in Figure~\ref{fig:mapping H,Gamma,Theta}, we have $(
(\tau_1,A_1), (\gamma_1,B_1),\sigma_1)\in {A_{14}^{\,(3,4)}}$ and $(
(\tau_2,A_2), (\gamma_2,B_2),\sigma_2)\in {A_{14}^{\,(4,3)}}$.

For $\pi\in\P^{\,(n_1,n_2)}_{n}$, we set $H_{n}^{\,(n_1,n_2)}(\pi)=((\tau,A), (\gamma,B),\sigma)$ where
\begin{itemize}
\item $\tau$ is the restriction of $\pi$ on $[1,N_2]$ and $A$ is the set of elements $\leq N_2$
 in $\pi$ which are connected to an element $>N_2$ by an arc;
\item $\gamma$ is the restriction of $\pi$ on $[N_2+1,n]$ and $B$ is the set of elements $> N_2$ in $\pi$ which are connected
    to an element $\leq N_2$ by an arc;
\item Suppose $A=\{a_1<a_2<\cdots<a_k\}$ and $B=\{b_1<b_2<\cdots<b_k\}$.
Then, $\sigma$ is the (unique) permutation in $\s_k$  such that
$(a_1,b_{\sigma(1)})$, $(a_2,b_{\sigma(2)})$, $\ldots$, $(a_k,b_{\sigma(k)})$
are arcs of~$\pi$.
\end{itemize}
Clearly,   $H_{n}^{\,(n_1,n_2)}$  is a mapping  from
${\P_{n}^{\,(n_1,n_2)}}$ to $ {A_{n}^{\,(n_1,n_2)}}$. Two
illustrations are given in Figure~\ref{fig:mapping H,Gamma,Theta}.
\begin{figure}[h!]
\begin{center}
{\setlength{\unitlength}{0.57mm}
\begin{picture}(220,154)(0,-80)
\put(-12,43){ $\pi_1=$}
\put(-5,30){\line(1,0){20}}\put(30,30){\line(1,0){30}}\put(75,30){\line(1,0){60}}
\put(-5,30){\circle*{1}} \put(-5,30){\makebox(-3,-4)[c]{\tiny 1}}
\put(5,30){\circle*{1}}  \put(5,30){\makebox(-3,-4)[c]{\tiny 2}}
\put(15,30){\circle*{1}} \put(15,30){\makebox(-3,-4)[c]{\tiny 3}}
\put(30,30){\circle*{1}} \put(30,30){\makebox(-3,-4)[c]{\tiny 4}}
\put(40,30){\circle*{1}} \put(40,30){\makebox(-3,-4)[c]{\tiny 5}}
\put(50,30){\circle*{1}} \put(50,30){\makebox(-3,-4)[c]{\tiny 6}}
\put(60,30){\circle*{1}} \put(60,30){\makebox(-3,-4)[c]{\tiny 7}}
\put(75,30){\circle*{1}} \put(75,30){\makebox(-3,-4)[c]{\tiny 8}}
\put(85,30){\circle*{1}} \put(85,30){\makebox(-3,-4)[c]{\tiny 9}}
\put(95,30){\circle*{1}} \put(95,30){\makebox(-3,-4)[c]{\tiny 10}}
\put(105,30){\circle*{1}}\put(105,30){\makebox(-3,-4)[c]{\tiny 11}}
\put(115,30){\circle*{1}}\put(115,30){\makebox(-3,-4)[c]{\tiny 12}}
\put(125,30){\circle*{1}}\put(125,30){\makebox(-3,-4)[c]{\tiny 13}}
\put(135,30){\circle*{1}}\put(135,30){\makebox(-3,-4)[c]{\tiny 14}}
\red{\qbezier(5,30)(17.5,50)(30,30)    \qbezier(15,30)(35,60)(50,30)}
\qbezier(75,30)(110,65)(135,30)  \qbezier(105,30)(110,38)(115,30)
\blue{\qbezier(-5,30)(45,80)(95,30)    \qbezier(40,30)(62.5,55)(85,30)
      \qbezier(60,30)(82.5,63)(105,30) \qbezier(30,30)(85,80)(125,30)}
\put(165,60){\line(1,0){20}}\put(200,60){\line(1,0){30}}
\put(165,60){\circle*{1}}\put(165,60){\makebox(-3,-4)[c]{\tiny 1}}
\put(175,60){\circle*{1}}\put(175,60){\makebox(-3,-4)[c]{\tiny 2}}
\put(185,60){\circle*{1}}\put(185,60){\makebox(-3,-4)[c]{\tiny 3}}
\put(200,60){\circle*{1}}\put(200,60){\makebox(-3,-4)[c]{\tiny 4}}
\put(210,60){\circle*{1}}\put(210,60){\makebox(-3,-4)[c]{\tiny 5}}
\put(220,60){\circle*{1}}\put(220,60){\makebox(-3,-4)[c]{\tiny 6}}
\put(230,60){\circle*{1}}\put(230,60){\makebox(-3,-4)[c]{\tiny 7}}
\red{\qbezier(175,60)(187.5,77)(200,60)\qbezier(185,60)(202.5,77)(220,60)}
\put(160,67){\small $\tau_1=$}
\put(160,49){\small $A_1=\{1,4,5,7\}$}
\put(165,30){\line(1,0){60}}
\put(165,30){\circle*{1}}\put(165,30){\makebox(-3,-4)[c]{\tiny 8}}
\put(175,30){\circle*{1}}\put(175,30){\makebox(-3,-4)[c]{\tiny 9}}
\put(185,30){\circle*{1}}\put(185,30){\makebox(-3,-4)[c]{\tiny 10}}
\put(195,30){\circle*{1}}\put(195,30){\makebox(-3,-4)[c]{\tiny 11}}
\put(205,30){\circle*{1}}\put(205,30){\makebox(-3,-4)[c]{\tiny 12}}
\put(215,30){\circle*{1}}\put(215,30){\makebox(-3,-4)[c]{\tiny 13}}
\put(225,30){\circle*{1}}\put(225,30){\makebox(-3,-4)[c]{\tiny 14}}
\qbezier(165,30)(195,53)(225,30)\qbezier(195,30)(200,38)(205,30)
\put(160,37){\small $\gamma_1=$}
\put(160,19){\small $B_1=\{9,10,11,13\}$}
\put(161,9){\small $\sigma_1=2\,4\,1\,3$}
\put(140,30){\vector(1,0){18}}\put(140,35){$H_{14}^{(3,4)}$}
\put(192,6){\vector(0,-1){18}}\put(195,-4){$\Gamma_{14}^{(3,4)}$}
\put(140,-60){\vector(1,0){18}}\put(140,-55){$H_{14}^{(4,3)}$}
\put(68,12){\vector(0,-1){30}}\put(50,-3){$\Theta_{14}^{(3,4)}$}
\put(165,-30){\line(1,0){30}}\put(210,-30){\line(1,0){20}}
\put(165,-30){\circle*{1}}\put(165,-30){\makebox(-3,-4)[c]{\tiny 1}}
\put(175,-30){\circle*{1}}\put(175,-30){\makebox(-3,-4)[c]{\tiny 2}}
\put(185,-30){\circle*{1}}\put(185,-30){\makebox(-3,-4)[c]{\tiny 3}}
\put(195,-30){\circle*{1}}\put(195,-30){\makebox(-3,-4)[c]{\tiny 4}}
\put(210,-30){\circle*{1}}\put(210,-30){\makebox(-3,-4)[c]{\tiny 5}}
\put(220,-30){\circle*{1}}\put(220,-30){\makebox(-3,-4)[c]{\tiny 6}}
\put(230,-30){\circle*{1}}\put(230,-30){\makebox(-3,-4)[c]{\tiny 7}}
\red{\qbezier(175,-30)(192.5,-8)(210,-30)\qbezier(195,-30)(207.5,-15)(220,-30)}
\put(160,-23){\small $\tau_2=$}
\put(160,-41){\small $A_2=\{1,3,5,7\}$}
\put(165,-60){\line(1,0){60}}
\put(165,-60){\circle*{1}}\put(165,-60){\makebox(-3,-4)[c]{\tiny 8}}
\put(175,-60){\circle*{1}}\put(175,-60){\makebox(-3,-4)[c]{\tiny 9}}
\put(185,-60){\circle*{1}}\put(185,-60){\makebox(-3,-4)[c]{\tiny 10}}
\put(195,-60){\circle*{1}}\put(195,-60){\makebox(-3,-4)[c]{\tiny 11}}
\put(205,-60){\circle*{1}}\put(205,-60){\makebox(-3,-4)[c]{\tiny 12}}
\put(215,-60){\circle*{1}}\put(215,-60){\makebox(-3,-4)[c]{\tiny 13}}
\put(225,-60){\circle*{1}}\put(225,-60){\makebox(-3,-4)[c]{\tiny 14}}
\qbezier(165,-60)(195,-37)(225,-60)\qbezier(195,-60)(200,-50)(205,-60)
\put(160,-53){\small $\gamma_2=$}
\put(160,-71){\small $B_2=\{9,10,11,13\}$}
\put(161,-81){\small $\sigma_2=2\,4\,1\,3$}
\put(-12,-47){ $\pi_2=$}
\put(-5,-60){\line(1,0){30}}
\put(40,-60){\line(1,0){20}}\put(75,-60){\line(1,0){60}}
\put(-5,-60){\circle*{1}}\put(-5,-60){\makebox(-3,-4)[c]{\tiny 1}}
\put(5,-60){\circle*{1}} \put(5,-60){\makebox(-3,-4)[c]{\tiny 2}}
\put(15,-60){\circle*{1}}\put(15,-60){\makebox(-3,-4)[c]{\tiny 3}}
\put(25,-60){\circle*{1}}\put(25,-60){\makebox(-3,-4)[c]{\tiny 4}}
\put(40,-60){\circle*{1}}\put(40,-60){\makebox(-3,-4)[c]{\tiny 5}}
\put(50,-60){\circle*{1}}\put(50,-60){\makebox(-3,-4)[c]{\tiny 6}}
\put(60,-60){\circle*{1}}\put(60,-60){\makebox(-3,-4)[c]{\tiny 7}}
\put(75,-60){\circle*{1}}\put(75,-60){\makebox(-3,-4)[c]{\tiny 8}}
\put(85,-60){\circle*{1}}\put(85,-60){\makebox(-3,-4)[c]{\tiny 9}}
\put(95,-60){\circle*{1}}\put(95,-60){\makebox(-3,-4)[c]{\tiny 10}}
\put(105,-60){\circle*{1}}\put(105,-60){\makebox(-3,-4)[c]{\tiny 11}}
\put(115,-60){\circle*{1}}\put(115,-60){\makebox(-3,-4)[c]{\tiny 12}}
\put(125,-60){\circle*{1}}\put(125,-60){\makebox(-3,-4)[c]{\tiny 13}}
\put(135,-60){\circle*{1}}\put(135,-60){\makebox(-3,-4)[c]{\tiny 14}}
\red{\qbezier(5,-60)(22.5,-35)(40,-60)\qbezier(25,-60)(37.5,-40)(50,-60)}
\qbezier(75,-60)(110,-25)(135,-60) \qbezier(105,-60)(110,-52)(115,-60)
\blue{\qbezier(-5,-60)(45,-5)(95,-60)\qbezier(40,-60)(62.5,-35)(85,-60)
\qbezier(60,-60)(82.5,-20)(105,-60) \qbezier(15,-60)(70,-10)(125,-60)}
\end{picture}}
\caption{The mappings $H_{n}^{\,(n_1,n_2)}$, $\Gamma_n^{(n_1,n_2)}$ and $\Theta_{n}^{\,(n_1,n_2)}$.}\label{fig:mapping H,Gamma,Theta}
\end{center}
\end{figure}
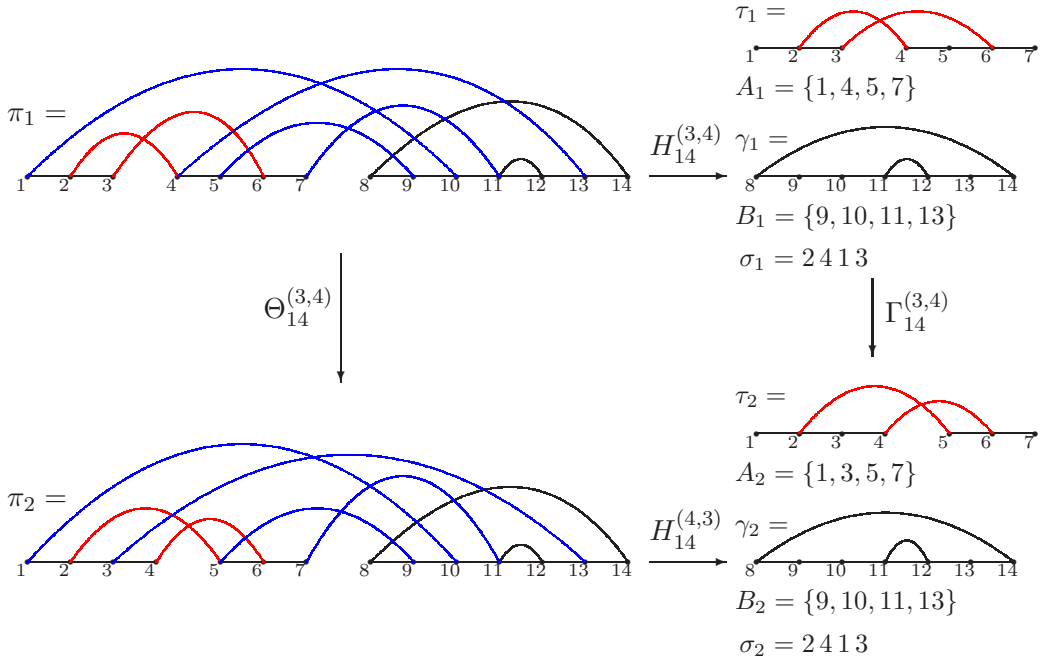

\begin{prop}\label{prop:H}
The mapping $H_{n}^{\,(n_1,n_2)}: {\P_{n}^{(n_1,n_2)}} \to {A_{n}^{(n_1,n_2)}}$ is a bijection.
Moreover, for any $\pi\in {\P_{n}^{(n_1,n_2)}}$, if
$H_{n}^{\,(n_1,n_2)}(\pi)=((\tau,A), (\gamma,B),  \sigma)$ and $k=|A|$, then
\begin{itemize}
\item[(i)] $\bl(\pi)=\bl(\tau)+\bl(\gamma)-k$,
\item[(ii)] $\rc(\pi)=\rc(\tau)+\rc(\gamma)+\ninv(\sigma)+\sum_{i\in A}\d_i(\tau)+\sum_{i\in B}\d_i(\gamma)$.
\end{itemize}
\end{prop}

\pf It is easy to see that $H_{n}^{(n_1,n_2)}$ establishes a
bijection from ${\P_{n}^{(n_1,n_2)}}$ to ${A_{n}^{(n_1,n_2)}}$ by
constructing  its  inverse (use Figure~\ref{fig:mapping
H,Gamma,Theta}), and Property (i)  is a direct consequence of
Fact~\ref{fact:blocks-arcs}.
 Let $\pi\in {\P_{n}^{(n_1,n_2)}}$. The arc crossings of the partition $\pi$ can be divided
into five parts $C_i(\pi)$, $1\leq i\leq 5$, illustrated in Table~\ref{tab:forme-cr-2}.  They are defined formally as follows:
\begin{align*}
C_1(\pi)&=\{(i_1,j_1)(i_2,j_2)\in\pi\;|\;1 \leq i_1< i_2<j_1<j_2\leq N_2\},\\
C_2(\pi)&=\{(i_1,j_1)(i_2,j_2)\in\pi\;|\;N_2<i_1< i_2<j_1<j_2 \leq n\},\\
C_3(\pi)&=\{(i_1,j_1)(i_2,j_2)\in\pi\;|\;1\leq i_1< i_2 \leq N_2<j_1<j_2 \leq n\},\\
C_4(\pi)&=\{(i_1,j_1)(i_2,j_2)\in\pi\;|\;1\leq i_1 \leq N_2 < i_2 <j_1<j_2 \leq n\},\\
C_5(\pi)&=\{(i_1,j_1)(i_2,j_2)\in\pi\;|\;1\leq i_1< i_2 <j_1\leq N_2<j_2 \leq n\},
\end{align*}
and satisfy  $\rc(\pi)=\sum_{i=1}^5 |C_i(\pi)|$. Suppose
$H_{n}^{(n_1,n_2)}(\pi)=((\tau,A),(\gamma,B) ,\sigma)$.  It is
easily checked (use Figure~\ref{fig:mapping H,Gamma,Theta}) that
$|C_1(\pi)|=\rc(\tau)$, $|C_2(\pi)|=\rc(\gamma)$,
$|C_3(\pi)|=\ninv(\sigma)$, $|C_4(\pi)|=\sum_{i\in B}\d_i(\gamma)$
and $|C_5(\pi)|=\sum_{i\in A}\d_i(\tau)$. Altogether, this leads to
Property (ii). \qed

\begin{table}[h]
$$
\begin{array}{|c|c|}
\hline
 i&C_i(\pi)\\
 \hline
1& \hspace{1cm} {\setlength{\unitlength}{0.6mm}
\begin{picture}(70,15)(0,-5)
\put(0,0){\line(1,0){13}}\put(35,0){\line(-1,0){13}}\put(45,0){\line(1,0){25}}
\put(0,0){\circle*{1,3}}\put(0,0){\makebox(-2,-6)[c]{\small 1}}
\put(35,0){\circle*{1,3}}\put(37,0){\makebox(-2,-6)[c]{\tiny $N_2$}}
\put(70,0){\circle*{1,3}}\put(70,0){\makebox(-2,-6)[c]{\small $n$}}
\put(13,0){\circle*{1,3}}\put(13,0){\makebox(-2,-6)[c]{\tiny $n_1$}}
\qbezier(5,0)(15,10)(25,0) \qbezier(10,0)(20,10)(30,0)
\end{picture}}
\hspace{1cm}\\
\hline
2&\hspace{1cm}{
{\setlength{\unitlength}{0.6mm}\begin{picture}(70,15)(0,-5)
\put(0,0){\line(1,0){25}}\put(35,0){\line(1,0){35}}
\put(0,0){\circle*{1,3}}\put(0,0){\makebox(-2,-6)[c]{\small 1}}
\put(25,0){\circle*{1,3}}\put(25,0){\makebox(-2,-6)[c]{\tiny $N_2$}}
\put(70,0){\circle*{1,3}}\put(70,0){\makebox(-2,-6)[c]{\small $n$}}
\qbezier(40,0)(50,10)(60,0) \qbezier(45,0)(55,10)(65,0)
\end{picture}}}
\hspace{1cm}\\
\hline
3&\hspace{1cm} {{\setlength{\unitlength}{0.6mm}
\begin{picture}(70,15)(0,-5)
\put(0,0){\line(1,0){25}}\put(35,0){\line(1,0){35}}
\put(0,0){\circle*{1,3}}\put(0,0){\makebox(-2,-6)[c]{\small 1}}
\put(25,0){\circle*{1,3}}\put(25,0){\makebox(-2,-6)[c]{\tiny $N_2$}}
\put(70,0){\circle*{1,3}}\put(70,0){\makebox(-2,-6)[c]{\small $n$}}
\qbezier(5,0)(30,10)(45,0) \qbezier(20,0)(40,10)(60,0)
\end{picture}}}
\hspace{1cm}\\
\hline
4&\hspace{1cm}{{\setlength{\unitlength}{0.6mm}
\begin{picture}(70,15)(0,-5)
\put(0,0){\line(1,0){25}}\put(35,0){\line(1,0){35}}
\put(0,0){\circle*{1,3}}\put(0,0){\makebox(-2,-6)[c]{\small 1}}
\put(25,0){\circle*{1,3}}\put(25,0){\makebox(-2,-6)[c]{\tiny $N_2$}}
\put(70,0){\circle*{1,3}}\put(70,0){\makebox(-2,-6)[c]{\small $n$}}
\qbezier(12.5,0)(37.5,10)(52.5,0) \qbezier(45,0)(52.5,10)(60,0)
\end{picture}}}
\hspace{1cm}\\
\hline
5&\hspace{1cm}{\setlength{\unitlength}{0.6mm}
\begin{picture}(70,15)(0,-5)
\put(0,0){\line(1,0){13}} \put(35,0){\line(-1,0){13}}
\put(45,0){\line(1,0){25}}
\put(0,0){\circle*{1,3}}\put(0,0){\makebox(-2,-6)[c]{\small 1}}
\put(35,0){\circle*{1,3}}\put(35,0){\makebox(-2,-6)[c]{\tiny $N_2$}}
\put(70,0){\circle*{1,3}}\put(70,0){\makebox(-2,-6)[c]{\small $n$}}
\put(13,0){\circle*{1,3}}\put(13,0){\makebox(-2,-6)[c]{\tiny $n_1$}}
\qbezier(10,0)(17.5,10)(25,0) \qbezier(17.5,0)(37.5,10)(57.5,0)
\end{picture}}
\hspace{1cm}\\
\hline
\end{array}
$$
\caption{Sketchs of crossings  in $C_i(\pi)$.}\label{tab:forme-cr-2}
\end{table}

Let
 \be
 R(n_1,n_2):=\{(\pi,A)\,:\, \pi\in\P^*(n_1,n_2) \quad \textrm{and}\quad \sing(\pi)\subseteq A\subseteq\max(\pi)\}.
 \ee
For instance, the elements $(\pi,A)$ and $(\pi,A')$ drawn in Figure~\ref{fig:psi(a,b)} are, respectively,
in $R(4,6)$ and $R(6,4)$.

In view of Proposition~\ref{prop:H}, to prove
Proposition~\ref{prop:Theta}, it suffices to demonstrate the
following result.

\begin{prop}\label{prop:psi(a,b)}
There is a bijection $\psi_{(n_1,n_2)}: R(n_1,n_2) \to R(n_2,n_1)$
such that,
for  $(\pi,A)\in R(n_1,n_2)$,  if $\psi_{(n_1,n_2)}(\pi,A)=(\pi',A')$,  then
\begin{align}\label{eq:psi(a,b)}
 \rc(\pi')=\rc(\pi),\quad |A'|=|A|,\quad \sum_{i\in A'} \d_{i}(\pi')=\sum_{i\in A} \d_{i}(\pi).
\end{align}
\end{prop}
\begin{proof}
To any $(\pi,A)\in R(n_1,n_2)$ we associate an element  $(\pi',A')$ in $R(n_2,n_1)$ as follows:
 \begin{itemize}
\item By definition of $\P^*(n_1,n_2)$, the arcs of $\pi$
are  $(i_1,j_{\rho (1)})$, $(i_2,j_{\rho (2)})$, $\ldots$, $(i_k,j_{\rho (k)})$
for some integers $k\geq 0$, $1\leq i_1<i_2<\cdots<i_k\leq n_1$, $n_1+1\leq j_1<j_2<\cdots<j_k\leq  N_2$
and some permutation $\rho\in\s_k$.
We use $\overline{i}$ for the complement of~$i$ in~$[1,N_2]$, i.e.,  $\overline{i}=N_2+1-i$.
Then, we define $\pi'$ as the partition  of~$[1,N_2]$ which consists  of the
arcs $(\overline{j_r}, \overline{i_{\rho (r)}})$ for $1\leq r \leq k$. It is clear that $\pi'\in \P^*(n_2,n_1)$.
Moreover, we have $\rc(\pi')=\ninv(\rho)$ and $\rc(\pi)=\ninv(\rho)$ whence  $\rc(\pi')=\rc(\pi)$.

\item  Since $\sing(\pi)\subseteq A\subseteq\max(\pi)$, we have $A=\sing(\pi)\cup B$
with $B=\{j_{\ell(1)}<j_{\ell(2)}<\cdots<j_{\ell(t)}\}$ for some
increasing sequence $(\ell(s))_{1\leq s\leq t}$.
Suppose $\overline{I}:=\{\overline{i_1}, \overline{i_2}, \ldots, \overline{i_k}\}=\{u_1<u_2<\cdots<u_k\}$.
 We then  set $A':=\sing(\pi')\cup B'$ with $B'=\{u_{\ell(1)}<u_{\ell(2)}<\cdots<u_{\ell(t)}\}$. Clearly, we have
$\sing(\pi')\subseteq A'\subseteq\max(\pi')$ and $|A'|=|A|$. It is also easily checked that
$d_{u_{\ell(t)}}(\pi')=d_{j_{\ell(t)}}(\pi)$ for $s=1,2,\ldots,t$
whence $\sum_{i\in B'} \d_{i}(\pi')=\sum_{i\in B} \d_{i}(\pi)$. Moreover,
since $\sing(\pi')=\overline{\sing(\pi)}$ and $d_{\overline{i}}(\pi')=d_{i}(\pi)$ for $i\in \sing(\pi)$, we see that
$\sum_{i\in \sing(\pi')} \d_{i}(\pi')=\sum_{i\in \sing(\pi)} \d_{i}(\pi)$.
Altogether, this implies that $\sum_{i\in A'} \d_{i}(\pi')=\sum_{i\in A} \d_{i}(\pi)$.
\end{itemize}
 Set $\psi_{(n_1,n_2)}(\pi,A)=(\pi', A')$.  Then  $\psi_{(n_1,n_2)}$
is a well-defined  map from $R(n_1,n_2)$ to $R(n_2,n_1)$ and satisfies~\eqref{eq:psi(a,b)}.
 An illustration is given in Figure~\ref{fig:psi(a,b)}.
Besides, it is easy to see that  the composition $\psi_{(n_2,n_1)} \circ \psi_{(n_1,n_2)}$ is the identity mapping. This proves that   $\psi_{(n_1,n_2)}$
is a bijection.
\end{proof}

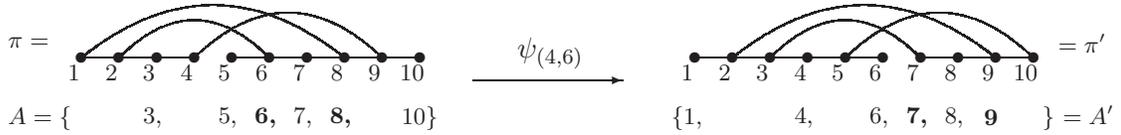
\begin{figure}[h!]
\begin{center}
\unitlength=1mm
\begin{picture}(52,18)(-6,-8)
\put(0,0){\line(1,0){15}}\put(20,0){\line(1,0){25}}
\qbezier(0,0)(17.5,14)(35,0)\qbezier(5,0)(15,10)(25,0)\qbezier(15,0)(27.5,12)(40,0)
\put(0,0){\circle*{1,3}}\put(0,0){\makebox(-2,-4)[c]{\footnotesize1}}
\put(5,0){\circle*{1,3}}\put(5,0){\makebox(-2,-4)[c]{\footnotesize2}}
\put(10,0){\circle*{1,3}}\put(10,0){\makebox(-2,-4)[c]{\footnotesize3}}
\put(15,0){\circle*{1,3}}\put(15,0){\makebox(-2,-4)[c]{\footnotesize4}}
\put(20,0){\circle*{1,3}}\put(20,0){\makebox(-2,-4)[c]{\footnotesize5}}
\put(25,0){\circle*{1,3}}\put(25,0){\makebox(-2,-4)[c]{\footnotesize6}}
\put(30,0){\circle*{1,3}}\put(30,0){\makebox(-2,-4)[c]{\footnotesize7}}
\put(35,0){\circle*{1,3}}\put(35,0){\makebox(-2,-4)[c]{\footnotesize8}}
\put(40,0){\circle*{1,3}}\put(40,0){\makebox(-2,-4)[c]{\footnotesize9}}
\put(45,0){\circle*{1,3}}\put(45,0){\makebox(-2,-4)[c]{\footnotesize10}}
\put(10,-6){\makebox(-1,-4)[c]{\footnotesize 3,}}
\put(20,-6){\makebox(-1,-4)[c]{\footnotesize 5,}}
\put(25,-6){\makebox(-1,-4)[c]{\footnotesize \bf{6},}}
\put(30,-6){\makebox(-1,-4)[c]{\footnotesize 7,}}
\put(35,-6){\makebox(-1,-4)[c]{\footnotesize \bf{8},}}
\put(45,-6){\makebox(0,-4)[c]{\footnotesize 10\}}}
\put(-6,4){\makebox(-2,-4)[c]{\footnotesize $\pi=$}}
\put(-5,-6){\makebox(-1,-4)[c]{\footnotesize $A=\{$}}
\put(52,-3){\vector(1,0){20}}\put(58,0){$\psi_{(4,6)}$}
\end{picture}
\hspace{2.8cm}
\begin{picture}(54,17)(-5,-8)
\put(0,0){\line(1,0){25}}\put(30,0){\line(1,0){15}}
\put(49,4){\makebox(5,-4)[c]{\footnotesize $=\pi'$}}
\qbezier(5,0)(22.5,14)(40,0)\qbezier(10,0)(20,10)(30,0)\qbezier(20,0)(32.5,12)(45,0)
\put(0,0){\circle*{1,3}}\put(0,0){\makebox(-2,-4)[c]{\footnotesize1}}
\put(5,0){\circle*{1,3}}\put(5,0){\makebox(-2,-4)[c]{\footnotesize2}}
\put(10,0){\circle*{1,3}}\put(10,0){\makebox(-2,-4)[c]{\footnotesize3}}
\put(15,0){\circle*{1,3}}\put(15,0){\makebox(-2,-4)[c]{\footnotesize4}}
\put(20,0){\circle*{1,3}}\put(20,0){\makebox(-2,-4)[c]{\footnotesize5}}
\put(25,0){\circle*{1,3}}\put(25,0){\makebox(-2,-4)[c]{\footnotesize6}}
\put(30,0){\circle*{1,3}}\put(30,0){\makebox(-2,-4)[c]{\footnotesize7}}
\put(35,0){\circle*{1,3}}\put(35,0){\makebox(-2,-4)[c]{\footnotesize8}}
\put(40,0){\circle*{1,3}}\put(40,0){\makebox(-2,-4)[c]{\footnotesize9}}
\put(45,0){\circle*{1,3}}\put(45,0){\makebox(-2,-4)[c]{\footnotesize10}}
\put(49,-6){\makebox(4,-4)[c]{\footnotesize \} $=A'$}}
\put(0,-6){\makebox(-2,-4)[c]{\footnotesize $\{1$,}}
\put(15,-6){\makebox(-1,-4)[c]{\footnotesize 4,}}
\put(25,-6){\makebox(-1,-4)[c]{\footnotesize 6,}}
\put(30,-6){\makebox(-1,-4)[c]{\footnotesize \bf{7},}}
\put(35,-6){\makebox(-1,-4)[c]{\footnotesize 8,}}
\put(40,-6){\makebox(0,-4)[c]{\footnotesize \bf{9} }}
\end{picture}
\end{center}
\caption{The mapping $\psi_{(n_1,n_2)}$}\label{fig:psi(a,b)}
\end{figure}

For $((\pi,A),(\gamma,B),\sigma)\in A_n^{(n_1,n_2)}$, we set
\begin{align*}
 \Gamma_n^{(n_1,n_2)}((\pi,A),(\gamma,B),\sigma):=(\psi_{(n_1,n_2)}(\tau,A),(\gamma,B),\sigma).
\end{align*}
Clearly $\Gamma_n^{(n_1,n_2)}$ is a mapping from $A_n^{(n_1,n_2)}$ to $A_n^{(n_2,n_1)}$.
An illustration is given in Figure~\ref{fig:mapping H,Gamma,Theta}.
From  Proposition~\ref{prop:psi(a,b)} we deduce the following result.
\begin{prop}\label{prop:Gamma}
 The mapping $\Gamma_n^{(n_1,n_2)}: A_n^{(n_1,n_2)} \to A_n^{(n_2,n_1)}$
is a bijection. Moreover,  if $((\tau,A),(\gamma,B),\sigma) \in A_n^{(n_1,n_2)}$
and $\Gamma_n^{(n_1,n_2)}((\pi,A),(\gamma,B),\sigma)=((\tau',A'),(\gamma,B),\sigma)$,
then we have
$$\rc(\tau')=\rc(\tau),\quad |A'|=|A|,\quad \sum_{i\in A'} \d_{i}(\tau')=\sum_{i\in A} \d_{i}(\tau).$$
\end{prop}
Finally,  we define the mapping $\Theta_n^{(n_1,n_2)}: \P_n^{(n_1,n_2)} \to \P_n^{(n_2,n_1)}$ by
\be
\Theta_n^{(n_1,n_2)}:= \left(H_n^{(n_2,n_1)}\right)^{-1} \circ \Gamma_n^{(n_1,n_2)} \circ H_n^{(n_1,n_2)}.
\ee
An illustration is given in Figure~\ref{fig:mapping H,Gamma,Theta}.
Combining Propositions~\ref{prop:H} and \ref{prop:Gamma}, we conclude
that the mapping $\Theta_n^{(n_1,n_2)}$ satisfies the requirements of Proposition~\ref{prop:Theta}. \\

\medskip
{\bf Acknowledgements.} We thank the handling editor for his patience and careful reading of the manuscript.
The third author  gratefully acknowledges the hospitality of the City University of Hong Kong and Korea Advanced
Institute of Science and Technology where part of this work was done in the fall 2011.


\end{document}